\documentclass{amsart}

\usepackage[english]{babel}
\usepackage[utf8]{inputenc}
\usepackage{amsmath}
\usepackage{amsfonts}
\usepackage{mathrsfs}
\usepackage{mathtools}
\mathtoolsset{showonlyrefs}

\usepackage{esint}
\usepackage[colorlinks=true, allcolors=red, linktocpage=true]{hyperref}
\usepackage{amssymb}
\usepackage{graphicx}
\usepackage{commath}
\usepackage{cases}
\usepackage{bm}
\usepackage{dsfont}
\usepackage{xcolor}
\usepackage{verbatim}
\usepackage{todonotes}

\usepackage{amsthm}

\theoremstyle{plain}

\newtheorem{thm}{Theorem}[section]

\newtheorem{lem}[thm]{Lemma}

\newtheorem{cor}[thm]{Corollary}

\numberwithin{equation}{section}

\theoremstyle{definition}

\newtheorem{defn}[thm]{Definition}

\newtheorem{rem}[thm]{Remark}

\makeatletter
\newcommand*{\rom}[1]{\expandafter\@slowromancap\romannumeral #1@}
\makeatother

\newcommand{\e}{\varepsilon}

\newcommand{\B}{\mathbb{B}}
\newcommand{\C}{\mathbb{C}}

\newcommand{\N}{\mathbb{N}}

\newcommand{\R}{\mathbb{R}}

\newcommand{\Z}{\mathbb{Z}}

\def\dist{\mathop\mathrm{dist}} 
 
\def\diam{\mathop\mathrm{diam}}

\DeclarePairedDelimiter\floor{\lfloor}{\rfloor}

\newcommand{\romup}[1]
    {\MakeUppercase{\romannumeral #1}}

\title[$d$-TST for general sets]{A $d$-dimensional Analyst's Travelling Salesman Theorem for general sets in $\R^n$}

\author{Matthew Hyde }

\address{Matthew Hyde\\
School of Mathematics \\ University of Edinburgh \\ JCMB, Kings Buildings \\
Mayfield Road, Edinburgh,
EH9 3JZ, Scotland.}
\email{m.hyde "at" ed.ac.uk}


\thanks{M. Hyde was supported by The Maxwell Institute Graduate School in Analysis and its
Applications, a Centre for Doctoral Training funded by the UK Engineering and Physical
Sciences Research Council (grant EP/L016508/01), the Scottish Funding Council, Heriot-Watt
University and the University of Edinburgh.}

\begin{document}

\begin{abstract}
In his 1990 paper, Jones proved the following: given $E \subseteq \R^2$, there exists a curve $\Gamma$ such that $E \subseteq \Gamma$ and
\[ \mathscr{H}^1(\Gamma) \sim \text{diam}\, E + \sum_{Q} \beta_{E}(3Q)^2\ell(Q).\]
Here, $\beta_E(Q)$ measures how far $E$ deviates from a straight line inside $Q$. This was extended by Okikiolu to subsets of $\R^n$ and by Schul to subsets of a Hilbert space. 

In 2018, Azzam and Schul introduced a variant of the Jones $\beta$-number. With this, they, and separately Villa, proved similar results for lower regular subsets of $\R^n.$ In particular, Villa proved that, given $E \subseteq \R^n$ which is lower content regular, there exists a `nice' $d$-dimensional surface $F$ such that $E \subseteq F$ and 
\begin{align}
\mathscr{H}^d(F) \sim \text{diam}( E)^d + \sum_{Q} \beta_{E}(3Q)^2\ell(Q)^d.
\end{align}
In this context, a set $F$ is `nice' if it satisfies a certain topological non degeneracy condition, first introduced in a 2004 paper of David. 

In this paper we drop the lower regularity condition and prove an analogous result for general $d$-dimensional subsets of $\R^n.$ To do this, we introduce a new $d$-dimensional variant of the Jones $\beta$-number that is defined for any set in $\R^n.$ 
\end{abstract}

\maketitle

\tableofcontents

\section{Introduction}
\subsection{Analyst's Travelling Salesman Theorem}
The 1-dimensional Analyst's Travelling Salesman Theorem was first proven by Peter Jones \cite{jones1990rectifiable} for subsets of $\C$, with the motivation of studying the boundedness of a certain class of singular integral operators. Jones' theorem gives a multi-scale geometric criterion for when a set $E \subseteq \C$ can be contained in a curve of finite length. Roughly speaking, he proved that if $E$ is flat enough at most scales and locations, then there is a curve $\Gamma$ of finite length (with quantitative control on its length) such that $E \subseteq \Gamma.$ Conversely, he proved that if $E$ is a curve of finite length, then there is quantitative control over how often $E$ can be non-flat.

There are many notions of flatness one could consider, some of which will appear later in this introduction. The following is the one introduced by Jones. Define for $E,B \subseteq \R^n$,
\begin{align*}
\beta_{E,\infty}^d(B) = \frac{1}{r_B}\inf_L \sup\{\text{dist}(y,L) : y \in E \cap B\}
\end{align*}
where $L$ ranges over $d$-planes in $\R^n.$ This quantity became know as the Jones $\beta$-number. If we unravel the definition then $\beta_{E,\infty}^d(B) r_B$ describes the width of the smallest tube containing $E \cap B,$ see Figure \ref{f:beta}. Jones' theorem goes as follows:
\begin{figure}[ht]
  \centering
  \includegraphics[scale = 0.8]{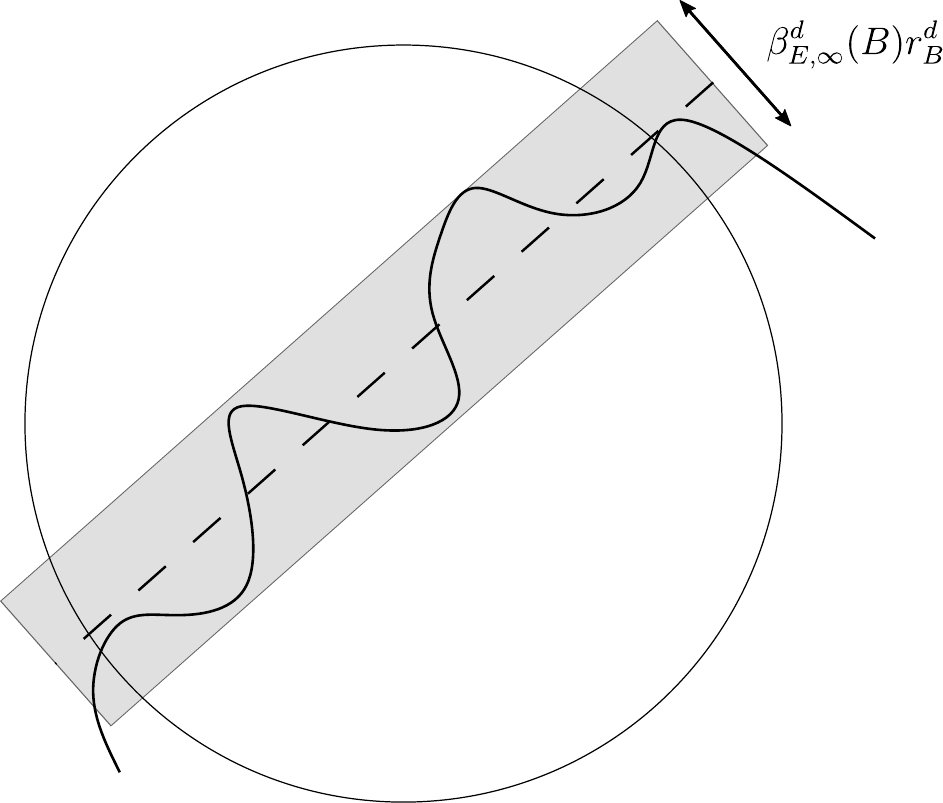}
\caption{A visual representation of $\beta_{E,\infty}^d(B)$. Notice that $\beta_{E,\infty}^d(B)$ only cares about the geometry of $E$ inside $B$, there may be points in $E\setminus B$ which is not contained in the tube.}\label{f:beta}
\end{figure}

\begin{thm}[Jones \cite{jones1990rectifiable}]\label{1DTSP}
There is a constant $C < \infty$ such that the following holds. Let $E \subseteq \C$. Then there is a connected set $\Gamma \supseteq E$ such that 
\begin{align}\label{1DTSP1}
\mathscr{H}^1(\Gamma) \lesssim_n \emph{diam}\,E + \sum_{\substack{Q \in \Delta \\ Q \cap E \not= \emptyset}} \beta_{E,\infty}^1(3Q)^2\ell(Q),
\end{align}
where $\Delta$ denotes the collection of all dyadic cubes in $\C$. Conversely, if $\Gamma$ is connected and $\mathscr{H}^1(\Gamma) < \infty$, then
\begin{align}\label{1DTSP2}
\emph{diam}\,\Gamma + \sum_{\substack{Q \in \Delta \\ Q \cap \Gamma \not= \emptyset}} \beta_{\Gamma,\infty}^1(3Q)^2\ell(Q) \lesssim_n \mathscr{H}^1(\Gamma).
\end{align}
\end{thm}

The proof of Jones' result relies heavily on the tools of complex analysis and so is unique to curves in $\C.$ A purely geometric proof was given by Okikiolu \cite{okikiolu1992characterization} for subsets of $\R^n$ and by Schul \cite{schul2007subsets} for subsets of a Hilbert space. 

We highlight the following corollary of Theorem \ref{1DTSP} to compare with Theorem  \ref{c:TSP}; if $E \subseteq \R^n$ and the right hand side of \eqref{1DTSP1} is finite, then there exists a curve $\Gamma \supseteq E$ such that 
\begin{align}
\mathscr{H}^1(\Gamma) \sim \text{diam}\,E + \sum_{\substack{Q \in \Delta \\ Q \cap E \not= \emptyset}} \beta_{E,\infty}^1(3Q)^2\ell(Q).
\end{align}

\subsection{For set of dimensions larger than 1} 

Theorem \ref{1DTSP} (along with the work of Okikiolu) gives a complete theory for 1-dimensional subsets of Euclidean space, but what about higher dimensional subsets? It is natural to ask whether a $d$-dimensional analogue of the above Jones' theorem is true, that is,  
\begin{enumerate}
\item Given a set $E,$ can we find a `nice' $d$-dimensional set $F$ containing $E$ such that 
\[ \mathscr{H}^d(F) \lesssim \text{diam}(E)^d + \sum_{ \substack{Q \in \Delta \\ Q \cap \Gamma \not= \emptyset}}\beta^d_{E,\infty}(3Q)^2\ell(Q)^d? \]
\item Given that $E$ is a `nice' $d$-dimensional set, can we say that 
\[\text{diam}(E)^d + \sum_{ \substack{Q \in \Delta \\ Q \cap \Gamma \not= \emptyset}}\beta^d_{E,\infty}(3Q)^2\ell(Q)^d \lesssim \mathscr{H}^d(E)?  \]
\end{enumerate}

Pajot \cite{pajot1996theoreme} proved an analogous result to the first half of Theorem \ref{1DTSP} for 2-dimensional sets in $\R^n.$ In this, he gave a sufficient condition in terms of the Jones $\beta$-number for when a set $E \subseteq \R^n$ can be contained in a surface $f(\R^2)$ for some smooth $f: \R^2 \rightarrow \R^n.$

For dimensions larger than two we encounter a problem, at least when we work with Jones' $\beta$-number. The most natural candidate for a `nice' $d$-dimensional set is a Lipschitz graph of a $d$-dimensional plane. However, in his PhD thesis, Fang \cite{fang1990cauchy} constructed a 3-dimensional Lipschitz graph whose $\beta^3_{\Gamma,\infty}$ sum was infinite. Thus, with the $\beta$-numbers as defined by Jones, a $d$-dimensional analogue of the second half of Theorem \ref{1DTSP} was proven to be false, since for any Travelling Salesman Theorem, we'd like Lipschitz graphs to be considered `nice'. 
\bigbreak
\noindent
\textbf{Ahlfors $d$-regular set.} The issue discovered by Fang was resolved by David and Semmes \cite{david1991singular} who introduced a new $\beta$-number and proved a Travelling Salesman Theorem for Ahlfors $d$-regular sets in $\R^n$. A set $E \subseteq \R^n$ is said to be \textit{Ahlfors} $d$-\textit{regular} if there is $A > 0$ such that
\[r^d/A \leq \mathscr{H}^d(E \cap B(x,r)) \leq Ar^d \ \text{for all} \ x \in E, \ r \in (0,\text{diam}E). \]
They defined their $\beta$-number as follows. For $E \subseteq \R^n$ and $B$ a ball, set
\begin{align*}
\hat\beta_{E}^{d,p}(B) = \inf_L \left( \frac{1}{r_B^d}\int_B \left( \frac{\text{dist}(y,L)}{r_B}\right)^p \, d\mathscr{H}^d|_E(y) \right)^\frac{1}{p},
\end{align*}
where $L$ ranges over all $d$-planes in $\R^n$. In some sense, Jones' $\beta$-number measures the $L^\infty$ deviation of $E$ from a plane whereas David and Semmes $\beta$-number, $\hat\beta_E^{d,p}(B),$ measures the $L^p$-average deviation of $E$ from a plane. In general, these quantities are not comparable. 

Let 
\[ p(d) \coloneqq 
\begin{cases}
	\frac{2d}{d-2}, & d < 2 \\
	\infty, & d \leq 2	
\end{cases}.
\]
Their result goes as follows:
\begin{thm}[David, Semmes \cite{david1991singular}]\label{t:DS}
Let $E \subseteq \R^n$ be Ahlfors $d$-regular. The following are equivalent:
\begin{enumerate}
\item The set $E$ has big pieces of Lipschitz images, meaning, there are constants $L,c>0$ such that for all $x \in E$ and $r \in (0,\diam E),$ there is an $L$-Lipschitz map $f:\R^d \rightarrow \R^n$ satisfying $\mathscr{H}^d(E \cap B(x,r) \cap f(\R^d)) \geq cr^d. $
\item For $1\leq p < p(d), \ \hat\beta_E^{d,p}(x,r)^2\tfrac{dxdr}{r}$ is a Carleson measure on $E \times (0,\infty).$ Recall that $\sigma$ is a Carleson measure on $E \cap (0,\infty)$ if $\sigma(B(x,r) \times (0,r)) \lesssim r^d$ for each $(x,r) \in E \cap (0,\infty).$    
\end{enumerate}
\end{thm}

\begin{rem}
An Ahlfors $d$-regular sets satisfying condition (1) of the above theorem is said to be \textit{uniformly rectifiable} (UR). This is one of many characterizations of UR sets.
\end{rem}

One can draw parallels between the above result and Jones' original theorem. On the one hand there is a condition on the multi-scale geometry of $E$, this time stated in terms of a Carleson condition involving the newly defined $\beta$-number of David and Semmes. On the other hand, there is a condition that $E$ is uniformly rectifiable, which means that large parts of the set can be covered by Lipschitz images. 
\bigbreak
\noindent
\textbf{$(c,d)$-lower content regular sets.} What can be said for sets which are not Ahlfors $d$-regular? In the Euclidean setting, not much progress was made on this until much more recently when Azzam and Schul \cite{azzam2018analyst} proved a Travelling Salesman Theorem for $d$-dimensional sets in $\R^n,$ under the weakened assumption of $(c,d)$-lower content regularity. A set $E \subseteq \R^n$ is said to be $(c,d)$-\textit{lower content regular} in a ball $B$ if for all $x \in E \cap B$ and $r \in (0,r_B),$ 
\begin{align*}
\mathscr{H}^d_\infty(E \cap B(x,r)) \geq cr^d.
\end{align*}
Here, $\mathscr{H}^d_\infty$ denotes the Hausdorff content, see \eqref{e:HausCont}. A general feature of lower regular sets is that they sufficiently spread out in at least $d$-directions at all locations and scales. In particular, they cannot concentrate around some lower dimensional set. Any curve is $(c,1)$-lower content regular for some $c >0$. For a higher dimensional example, any Reifenberg flat set is $(c,d)$-lower content regular. Recall, a set $E \subseteq \R^n$ is said to be $(\e,d)$-\textit{Reifenberg flat} if for each $x \in E$ and $0 < r<\diam(E)$, there exists a $d$-plane $P_{x,r}$ such that 
\[ \max \left\{ \sup_{y \in E \cap B(x,r)} \dist(y,P_{x,r}) , \sup_{y \in P_{x,r} \cap B(x,r)} \dist(y,E) \right\} \leq \e r. \]
To see an example of set which is not lower content regular, see Figure \ref{f:NotLR}.

\begin{figure}
  \centering
  \includegraphics{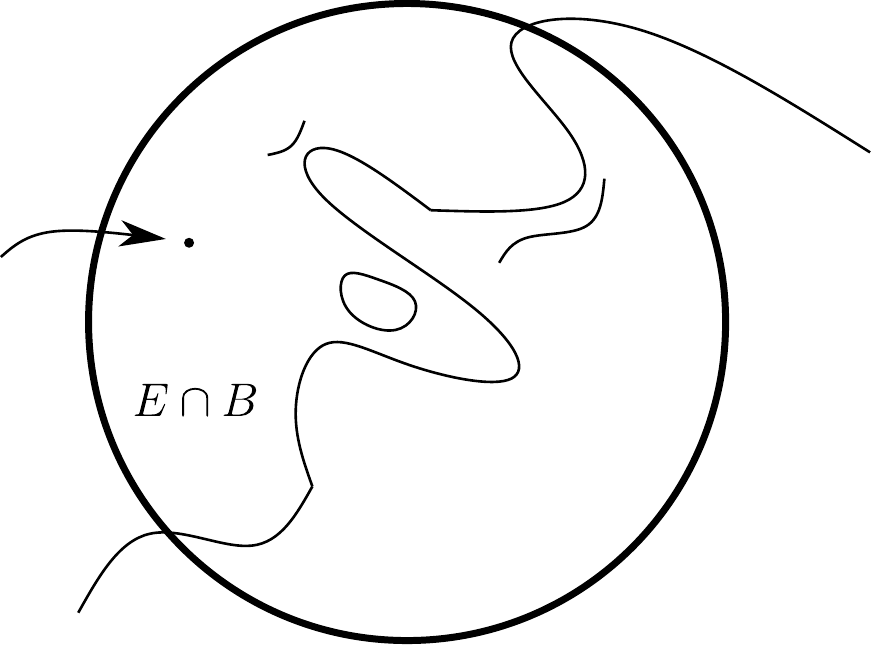}
\caption{An example of a set which is not 1-lower content regular. A ball centred at the highlighted singleton point, with radius small enough, will not satisfy the lower regularity condition for any parameter $c$.}\label{f:NotLR}
\end{figure}

Notice, under this relaxed condition, the measure $\mathscr{H}^d|_E$ may not be locally finite. As a result, Azzam and Schul were required to introduce a new $\beta$-number. The $\beta$-number they defined is analogous to that of David and Semmes but they instead `integrate' with respect to Hausdorff content. For $E \subseteq \R^n$ and a ball $B$, they defined
\begin{align*}
\check\beta_E^{d,p}(B) =  \inf_L \left( \frac{1}{r_B^d}\int_0^1 \mathscr{H}_\infty^d(\{x \in E \cap B : \text{dist}(x,L) > tr_B \})t^{p-1} \, dt \right)^\frac{1}{p},
\end{align*}
where $L$ ranges over all $d$-planes in $\R^n$. If $E$ is Ahlfors $d$-regular, then $\check\beta^{d,p}_E$ agrees with the $\beta$-number of David and Semmes up to a constant. 

To state their results, we need some additional notation. For closed sets $E,F \subseteq \R^n$ and a set $B$, define
\begin{align*}
d_B(E,F) = \frac{2}{\text{diam}(B)}\max \left\{\sup_{y \in E \cap B}\text{dist}(y,F), \sup_{y \in F \cap B}\text{dist}(y,E) \right\}.
\end{align*}
In the case where $B = B(x,r),$ we may write $d_{x,r}(E,F)$ to denote the above quantity. Thus, $d_B(E,F)$ is essentially the renormalised Hausdorff distance between the sets $E$ and $F$, inside $B$.

Some of the following results are stated in terms of Christ-David cubes. Given $E \subseteq \R^n$, the Christ-David cubes of $E$, usually denoted by $\mathscr{D}$, are subsets of $E$ which behave like Euclidean dyadic cubes. Importantly, they have a tree like structure. More shall be said about Christ-David cubes later, see Theorem \ref{cubes}

For $C_0 >0$ and $\e >0,$ let 
\[\text{BWGL}(C_0,\e) = \{Q \in \mathscr{D} : d_{C_0B_Q}(E,P) \geq \e \ \text{for all $d$-planes} \ P\}. \]
Thus, the cubes in BWGL$(C_0,\e)$ are those for which $E$ is not well approximated by a plane, in Hausdorff distance. The acronym BWGL stands for \emph{bi-lateral weak geometric lemma}. David and Semmes \cite{david1993analysis} gave another characterization of uniformly rectifiable sets in terms of BWGL. They showed that an Ahlfors $d$-regular set is uniformly rectifiable if and only if for every $C_0 \geq 1,$ there exists $\e >0$ such that $\text{BWGL}(C_0,\e)$ satisfies a Carleson condition with constant depending on $\e.$ By this, we mean there exists a constant $C = C(\e)$ such that for any $R \in \mathscr{D}$, we have
\[ \sum_{\substack{Q \subseteq R \\ Q \in \text{BWGL}(C_0,\e)}} \ell(Q)^d \leq C \ell(R)^d. \]

We now state the result from \cite{azzam2018analyst}. In fact, we state the reformulation presented in \cite{azzam2019quantitative}.

\begin{thm}[Azzam, Schul \cite{azzam2018analyst}]\label{dTSP}
Let $1 \leq d < n$ and $E \subseteq \R^n$ be a closed set. Suppose that $E$ is $(c,d)$-lower content regular and let $\mathscr{D}$ denote the Christ-David cubes for $E$. Let $C_0 >1.$ Then there is $\e >0$ small enough so that the following holds. Let $1 \leq p < p(d).$ For $R \in \mathscr{D},$ let
\begin{align}
\emph{BWGL}(R) = \emph{BWGL}(R,\e,C_0) = \sum_{\substack{Q \in \emph{BWGL}(\e,C_0) \\ Q \subseteq R}} \ell(Q)^d
\end{align}
and 
\begin{align}
\check\beta_{E,A,p}(R) \coloneqq \ell(R)^d + \sum_{Q \subseteq R} \check\beta^{d,p}_E(AB_Q)^2\ell(Q)^d.
\end{align}
Then, for $R \in \mathscr{D}$,
\begin{align}\label{e:dTSP}
\mathscr{H}^d(R) + \emph{BWGL}(R,\e,C_0) \sim_{A,n,c,p,C_0,\e} \check\beta_{E,A,p}(R). 
\end{align}
\end{thm}

We should mention the work of Edelen, Naber and Valtorta \cite{edelen2016quantitative}. In this paper, they consider generalizations of Reifenberg's topological disk theorem \cite{reifenberg1960solution}. Instead of imposing the Reifenberg flat condition (as Reifenberg did himself) they impose a one sided flatness condition, to allow for example, sets with holes. This condition is stated in terms of $\beta$-numbers. In the paper they describe how well the size of a Radon measure $\mu$ can be bounded from above by the corresponding $\hat\beta_{\mu}^{d,p}$-number (these are defined analogously to $\hat\beta^{d,p}_E$ with the integral taking over $\mu$ instead of $\mathscr{H}^d|_E$). We state a corollary of their results for Hausdorff measure and compare this to Theorem \ref{dTSP}.

\begin{thm}\label{ENV1}
Let $E \subseteq \R^n$. Set $\mu = \mathscr{H}^d|_E$ and assume
\[ \int_0^2 \beta_{\mu,2}^2(x,r) \frac{dr}{r} \leq M \quad \text{for} \  \mu\text{-a.e}\ x \in \B.\]
Then $E$ is rectifiable and for every $x \in \B$ and $0 < r \leq 1,$ we have
\[ \mathscr{H}^d(E \cap B(x,r)) \lesssim_n (1+M)r^d.\]
\end{thm}

As mentioned, this is just one corollary of much more general theorem for Radon measures (see \cite[Theorem 1.3]{edelen2016quantitative}). Both Theorem \ref{dTSP} and Theorem \ref{ENV1} do not require $E$ to be Ahlfors $d$-regular. Instead, Theorem \ref{dTSP} requires that $E$ must be lower content regular, whereas Theorem \ref{ENV1} requires the existence of a locally finite measure. Furthermore, Theorem \ref{dTSP} also provides lower bounds for Hausdorff measure. 

Azzam and Villa further generalise Theorem \ref{dTSP} in \cite{azzam2019quantitative}. Here, they introduce the notion of a \textit{quantitative property} which is a way of splitting the Christ-David cubes of a set $E$ into ``good" and ``bad" parts. They prove estimates of the form of Theorem \ref{dTSP}, where $\text{BWGL}(R)$ is instead replaced with other quantitative properties which `guarantee uniform rectifiability'. Here, a quantitative property is said to guarantee uniform rectifiability if whenever the bad set of cubes is small (quantified by a Carleson packing condition) then $E$ is uniform rectifiable. The BWGL condition is an example of a quantitative property which guarantees uniform rectifiability. We direct the reader to \cite{azzam2019quantitative} for a more precise description and more example of these quantitative properties. We state one of their results for the \emph{bi-lateral approximation uniformly by planes} (BAUP) condition, which we explain below (this will be used later to prove the second main result of this paper).

\begin{thm}[Azzam, Villa \cite{azzam2019quantitative}]\label{AV}
Let $E \subseteq \R^n$ be a $(c,d)$-lower content regular set with Christ-David cubes $\mathscr{D}.$
Let
\begin{align}
\emph{BAUP}(C_0,\e) = \{Q \in \mathscr{D} : d_{C_0B_Q}(E,U) \geq \e, \ U \ \text{is a union of $d$-planes}\}.
\end{align}
For $R \in \mathscr{D},$ define
\begin{align}
\emph{BAUP}(R,C_0,\e) = \sum_{\substack{Q \subseteq R \\ Q \in \emph{BAUP}(C_0,\e)}} \ell(Q)^d.
\end{align}
and
\begin{align}
\check\beta_E(R) = \ell(R)^d + \sum_{Q \subseteq R} \check\beta_E^{d,p}(3B_Q)^2\ell(Q)^d.
\end{align}
Then, for all $R \in \mathscr{D}, \ C_0 >1,$ and $\e >0$ small enough depending on $C_0$ and $c$,
\begin{align}\label{e:AV}
\mathscr{H}^d(R) + \emph{BAUP}(R,C_0,\e) \sim \check\beta_E(R). 
\end{align}
\end{thm}

Notice, Theorem \ref{dTSP} and Theorem \ref{AV} are more concerned with establishing quantitative estimates of the form seen in Theorem \ref{1DTSP}, rather than studying what types of surfaces could mimic the role of finite length curves, as in the 1-dimensional case.  Very recently, Villa \cite{villa2019sets} proved a Travelling Salesman Theorem for lower content regular sets more closely resembling that of Jones' original theorem. The `nice' sets he used were a certain class of topological non-degenerate surfaces, first introduced by David \cite{david2004hausdorff}.

\begin{defn}
Let $0 < \alpha_0 <1.$ Consider a one parameter family of Lipschitz maps $\{\varphi_t\}_{0\leq t \leq1}$, defined on $\R^n.$ We say $\{\varphi\}_{0 \leq t \leq 1}$ is an \textit{allowed Lipschitz deformation} with parameters $\alpha_0$, or an $\alpha_0$-ALD, if it satisfies the following condition:
\begin{enumerate}
\item $\varphi_t(B(x,r)) \subseteq \overline{B}(x,r)$ for each $t \in [0,1];$
\item for each $y \in \R^n, \  t \mapsto \varphi_t(y)$ is a continuous function on $[0,1]$;
\item $\varphi_0(y) = y$ and $\varphi_t(y) = y$ for $t \in [0,1]$ whenever $y \in \R^n \setminus B(x,r);$
\item $\text{dist}(\varphi_t(y),E) \leq \alpha_0r$ for $t \in [0,1]$ and $y \in E \cap B(x,r),$ where $0 < \alpha_0 < 1.$ 
\end{enumerate}
\end{defn}

\begin{defn}
Fix parameters $r_0,\alpha_0,\delta_0$ and $\eta_0.$ We say $E \subseteq \R^n$ is a \textit{topologically stable $d$-surface} with parameters $r_0,\alpha_0,\delta_0$ and $\eta_0,$ if for all $\alpha_0$-ALD $\{\varphi_t\},$ and for all $x_0 \in E$ and $0<r<r_0$, we have
\begin{align}
\mathscr{H}^d(B(x,(1-\eta_0)r) \cap \varphi_1(E)) \geq \delta_0 r^d. 
\end{align}
\end{defn}

Amongst other things, Villa proved the following.

\begin{thm}\label{t:Villa}
Let $E \subseteq \R^n$ be a $(c,d)$-lower content regular set and let $Q_0 \in \mathscr{D}.$  Given two parameters $0 < \e, \kappa <1,$ there exists a set $\Sigma = \Sigma(\e,\kappa,Q_0)$ such that
\begin{enumerate}
\item $Q_0 \subseteq \Sigma.$
\item $\Sigma$ is topologically stable $d$-surface with parameters $r_0 = \emph{diam}(Q_0)/2$, $0 < \eta > 1/100$, and $\alpha_0$ and $\delta_0$ sufficiently small with respect to $\e$ and $\kappa.$
\item We have the estimate
\begin{align}
\mathscr{H}^d(\Sigma) \sim _{c_0,n,d,\e} \emph{diam}(Q_0)^d + \sum_{\substack{Q \in \mathscr{D} \\ Q \subseteq Q_0 }} \check\beta_E^{d,p}(C_0B_Q)^2\ell(Q)^d 
\end{align}
\end{enumerate}
\end{thm}

Before stating our main results, we mention that Travelling Salesman type problems have been considered in a variety of other setting outside of $\R^n.$ For such results in the Heisenberg group see \cite{ferrari2007geometric}, \cite{li2016upper}, \cite{li2016traveling} and for general metric spaces see \cite{hahlomaa2005menger}, \cite{hahlomaa2008menger}, \cite{schul2007ahlfors}, \cite{david2019sharp}.

\subsection{Main Results}\label{MR}

We prove a $d$-dimensional analogue of Theorem \ref{1DTSP} for general sets in $\R^n.$ In particular, we do not assume $E$ to be Ahlfors regular or lower content regular and we do not assume the existence of a locally finite measure on $E$ (as was the case in Theorem \ref{t:DS}, Theorem \ref{dTSP} and Theorem \ref{ENV1}, respectively). We should emphasize that while Theorem \ref{dTSP} and Theorem \ref{ENV1} concentrate on proving bounds for measures, our result differs in the sense that we construct a nice surface which contains our set, and the measure of this surface is controlled by our $\beta$-numbers.

Observe that if $E$ does not satisfy any lower regularity condition, it may be that $\mathscr{H}^d_\infty(E) = 0$. Thus, $\check\beta^{d,p}_E$ may trivially return a zero value even if there is some inherent non-flatness, for example if $E$ is a dense collection of points in some purely unrectifiable set. We shall introduce a new $\beta$-numbers, $\beta_E^{d,p}$, to deal with this. We first define a variant of the Hausdorff content, where we `force' sets to have some lower regularity with respect to this content, and define $\beta^{d,p}_E$ (analogously to Azzam and Schul) by integrating with respect to this new content. 

\begin{defn}\label{IntroDef}
Let $E \subseteq \R^n$, $B$ a ball and $0< c_1 \leq c_2 < \infty$ be constants to be fixed later. We say a collection of balls $\mathscr{B}$ which covers $E \cap B$ is \textit{good} if 
\begin{align}\label{Size}
\sup_{B' \in \mathscr{B}} r_{B'} < \infty
\end{align}
and for all $x \in E \cap B$ and $0 <r <r_B,$ we have
\begin{align}\label{LRi}
\sum_{\substack{B' \in \mathscr{B} \\ B' \cap B(x,r) \cap E \cap B \not= \emptyset}} r_{B'}^d \geq c_1 r^d
\end{align}
and
\begin{align}\label{URi}
\sum_{\substack{B' \in \mathscr{B} \\ B' \cap B(x,r) \cap E \cap B \not=\emptyset \\ r_{B'} \leq r}} r_{B'}^d \leq c_2 r^d.
\end{align}
Then, for $A \subseteq E \cap B,$ define 
\begin{align*}
\mathscr{H}^{d,E}_{B,\infty}(A) = \inf\left\{ \sum_{\substack{B^\prime \in \mathscr{B}\\ B^\prime \cap A \not= \emptyset}} r_{B^\prime}^d : \mathscr{B} \ \text{is good for} \ E \cap B \right\}
\end{align*}
\end{defn}

\begin{defn}
Let $1 \leq p <\infty$, $E \subseteq \R^n,$ $B$ a ball centred on $E$ and $L$ a $d$-plane. Define
\begin{align*}
\beta_E^{d,p}(B,L)^p &= \frac{1}{r_B^d} \int \left( \frac{\text{dist}(x,L)}{r_B}\right)^p \, d\mathscr{H}^{d,E}_{B,\infty} \\
&= \frac{1}{r_B^d} \int_0^1 \mathscr{H}^{d,E}_{B,\infty}(\{x \in E \cap B : \text{dist}(x,L) >tr_B \}) t^{p-1} \, dt,
\end{align*}
and 
\begin{align}
\beta_E^{d,p}(B) = \inf \{ \beta_E^{d,p}(B,L) : L \ \text{is a $d$-plane} \}. 
\end{align}
\end{defn}

In Section \ref{Prelims} we study the above definitions. We are more explicit about the constant $c_1,c_2$ appearing in Definition \ref{IntroDef} and we shall prove some basic properties of $\mathscr{H}^{d,E}_{B,\infty}$ and $\beta_E^{d,p}.$

Our main results read as follows: 

\begin{thm}\label{Thm1}
Let $1\leq d < n$, $C_0 >1$ and $1 \leq p < p(d).$ There exists a constant $c_1 > 0$ such that the following holds. Suppose $E \subseteq F \subseteq \R^n$, where $F$ is $(c,d)$-lower content regular for some $c \geq c_1.$ Let $\mathscr{D}^E$ and $\mathscr{D}^F$ be the Christ-David cubes for $E$ and $F$ respectively. Let $Q_0^E \in \mathscr{D}^E$ and let $Q_0^F$ be the cube in $\mathscr{D}^F$ with the same centre and side length as $Q_0^E.$ Then
\begin{equation}\label{e:Thm1}
\begin{aligned}
&\emph{diam}(Q_0^E)^d +  \sum_{\substack{Q \in \mathscr{D}^E\\ Q \subseteq Q_0^E}} \beta_E^{d,p}(C_0Q)^2 \ell(Q)^d \\
&\hspace{4em} \lesssim_{C_0,c,n,p} \emph{diam}(Q_0^F)^d+ \sum_{\substack{Q \in \mathscr{D}^F \\ Q \subseteq Q_0^F}} \check\beta_F^{d,p}(C_0Q)^2 \ell(Q)^d.
\end{aligned}
\end{equation}
\end{thm}

\begin{rem}
We construct the cubes for $\mathscr{D}^E$ using a sequence of maximally $\rho^n$-separated nets $\{X_n^E\}_{n=0}^\infty$ in $E$. For each $n$, we complete $X_n^E$ to a maximally $\rho^n$-separated net for $F$, which we call $X_n^F$, and construct $\mathscr{D}^F$ using the sequence $\{X_n^F\}_{n=0}^\infty.$ In this way, we have $Q_0^E \subseteq Q_0^F.$
\end{rem}

\begin{thm}\label{Thm4}
Let $1\leq d < n$, $C_0 >1$ and $1 \leq p < p(d).$ Let $E \subseteq \R^n,$ $\mathscr{D}^E$ denote the Christ-David cubes for $E$ and let $Q_0^E \in \mathscr{D}^E$ be such that $\diam(Q_0^E) \geq \lambda \ell(Q_0^E)$ for some $0 < \lambda \leq 1.$ Then there exists a $(c_1,d)$-lower content regular set $F$ (with $c_1$ as in the previous theorem) such that the following holds. Let $\mathscr{D}^F$ denote the Christ-David cubes for $F$ and let $Q_0^F$ denote the cube in $\mathscr{D}^F$ with the same centre and side length as $Q_0^E$. Then 
\begin{equation}\label{e:Thm4}
\begin{aligned}
&\emph{diam}(Q_0^F)^d+ \sum_{\substack{Q \in \mathscr{D}^F \\ Q \subseteq Q_0^F}} \check\beta_F^{d,p}(C_0Q)^2 \ell(Q)^d \\
&\hspace{4em} \lesssim_{C_0,c,n,p,\lambda} \emph{diam}(Q_0^E)^d +  \sum_{\substack{Q \in \mathscr{D}^E\\ Q \subseteq Q_0^E}} \beta_E^{d,p}(C_0Q)^2 \ell(Q)^d.
\end{aligned}
\end{equation}
\end{thm}

\begin{rem}
The condition that $\diam(Q_0^E) \geq \lambda \ell(Q_0)$ is not too strong an assumption. It simply ensures that the top cube is suitably sized for the set $E$. Consider for example the extreme case when $E$ is just a singleton point. The right hand side of \eqref{e:Thm4} will be identically equal to zero whilst the left hand side must be non-zero, by virtue of the fact that any lower regular set must have non-zero diameter. We would be in the same situation if for example $E$ was a very small (in comparison to $\ell(Q_0^E)$) straight line segment centred on $Q_0^E.$
\end{rem}

\begin{rem}
Both Theorem \ref{Thm1} and Theorem \ref{Thm4} are true with $\check\beta_F^{d,p}$ replaced by $\beta_{F}^{d,p}.$ This is because for lower regular sets, the two quantities are comparable, see Corollary \ref{CheckComp}.
\end{rem}

As an immediate corollary of Theorem \ref{Thm1} and Theorem \ref{Thm4}, along with Theorem \ref{t:Villa}, we obtain a Travelling Salesman Theorem for general sets in $\R^n$ resembling that of Jones original theorem.

\begin{thm}\label{c:TSP}
Let $1 \leq d \leq n$, $1 \leq p \leq p(d)$, and $C_0 > 1.$ Suppose $E \subseteq \R^n$, $Q_0 \in \mathscr{D}$ and 
\[ \sum_{\substack{Q \in \mathscr{D} \\ Q \subseteq Q_0}} \beta_E^{d,p}(C_0B_Q)^2\ell(Q)^d < \infty.\]
Then there exists $r_0,\alpha_0,\delta_0$ and $\eta_0$ and a topologically stable $d$-surface $\Sigma,$ with parameters $r_0,\alpha_0,\delta_0$ and $\eta_0$ such that $E \subseteq \Sigma$ and 
\begin{align}
\mathscr{H}^d(\Sigma) \sim \emph{diam}(Q_0)^d + \sum_{\substack{Q \in \mathscr{D} \\ Q \subseteq Q_0}} \beta_E^{d,p}(C_0B_Q)^2\ell(Q)^d. 
\end{align}
\end{thm}

\subsection{Acknowledgment} I would like to thank Jonas Azzam, my supervisor, for his invaluable support, guidance and patience throughout this project. 

\section{Preliminaries}\label{Prelims}
\subsection{Notation}

If there exists $C>0$ such that $a \leq Cb,$ then we shall write $a \lesssim b.$ If the constant $C$ depends of a parameter $t$, we shall write $a \lesssim_t b.$ We shall write $a \sim b$ if $a \lesssim b$ and $b\lesssim a,$ similarly we define $a \sim_t b.$ 

For $A,B \subseteq \R^n$, let
\begin{align*}
\text{dist}(A,B) = \inf\{|x-y| : x \in A, \ y \in B \}
\end{align*}
and 
\[\text{diam}(A) = \sup\{|x-y| : x,y \in A \}. \]

We recall the $d$-dimensional Hausdorff measure and content. For $A \subseteq \R^n, \ d \geq 0$ and $0 <\delta \leq \infty$ define
\begin{align}\label{e:HausCont}
\mathscr{H}^d_\delta(A) = \inf \left\{ \sum_i \diam(A_i)^d : A \subseteq \bigcup_i A_i \ \text{and} \ \diam A_i \leq \delta \right\}.
\end{align}
The $d$-dimensional Hausdorff \textit{content} of $A$ is defined to be $\mathscr{H}^d_\infty(A)$ and the $d$-dimensional Hausdorff \textit{measure} of $A$ is defined to be
\[ \mathscr{H}^d(A) = \lim_{\delta \rightarrow 0} \mathscr{H}^d_\delta(A).\]

\subsection{Chirst-David cubes} For a set $E \subseteq \R^n$ we shall need a version of ``dyadic cubes''. These were first introduced by David \cite{david1988morceaux} and generalised in \cite{christ1990b} and \cite{hytonen2012non}. 

\begin{lem}\label{cubes}
Let $X$ be a doubling metric space and $X_k$ be a sequence of maximal $\rho^k$-separated nets, where $\rho = 1/1000$ and let $c_0 = 1/500.$ Then, for each $k \in \Z$, there is a collection $\mathscr{D}_k$ of cubes such that the following hold.
\begin{enumerate}
\item For each $k \in \Z, \ X = \bigcup_{Q \in \mathscr{D}_k}Q.$
\item If $Q_1,Q_2 \in \mathscr{D} = \bigcup_{k}\mathscr{D}_k$ and $Q_1 \cap Q_2 \not= \emptyset,$ then $Q_1 \subseteq Q_2$ or $Q_2 \subseteq Q_1.$ 
\item For $Q \in \mathscr{D},$ let $k(Q)$ be the unique integer so that $Q \in \mathscr{D}_k$ and set $\ell(Q) = 5\rho^k.$ Then there is $x_Q \in X_k$ such that
\begin{align*}
B(x_Q,c_0\ell(Q)) \subseteq Q \subseteq B(x_Q , \ell(Q)). 
\end{align*}
\end{enumerate}
\end{lem}

Given a collection of cubes $\mathscr{D}$ and $Q  \in \mathscr{D},$ define 
\[\mathscr{D}(Q) = \{R \in \mathscr{D} : R \subseteq Q\}.\]
Let $\text{Child}_k(Q)$ denote the $k^{th}$ generational descendants of $Q$ (where we often write $\text{Child}(Q)$ to mean $\text{Child}_1(Q)$) and $Q^{(k)}$ denote the $k^{th}$ generational ancestor. We shall denote the descendants up to the $k^{th}$ by $\text{Des}_k(Q),$ that is, 
\begin{align}\label{d:Des}
\text{Des}_k(Q) = \bigcup_{i=0}^k \text{Child}_i(Q). 
\end{align}
Finally define a distance function, $d_\mathscr{C}$, to a collection of cubes $\mathscr{C} \subseteq \mathscr{D}$ by setting
\[ d_\mathscr{C}(x) = \inf\{\ell(R) + \text{dist}(x,R) : R \in \mathscr{C} \}, \] 
and for $Q \in \mathscr{D},$ set
\[ d_\mathscr{C}(Q) = \inf\{ d_\mathscr{C}(x) : x \in Q \}. \]
The following lemma is standard and can be found in, for example, \cite{azzam2018analyst}.
\begin{lem}
Let $\mathscr{C} \subseteq \mathscr{D}$ and $Q,Q^\prime \in \mathscr{D}.$ Then
\begin{align}\label{Tr}
d_\mathscr{C}(Q) \leq 2\ell(Q) + \emph{dist}(Q,Q^\prime) + 2\ell(Q^\prime) + d_\mathscr{C}(Q^\prime).
\end{align}
\end{lem}

\subsection{Theorem of David and Toro} The surface $F$ from Theorem \ref{Thm4} will be a union of surfaces constructed using the following Reifenberg parametrization theorem of David and Toro \cite{david2012reifenberg}.
\begin{thm}[{\cite[Sections 1 - 9]{david2012reifenberg}}]\label{DT}
Let $P_0$ a plane. Let $k \in \N$ and set $r_k = 10^{-k}$. Let $\{x_{j,k}\}_{j \in J_k}$ be an $r_k$-separated net. To each $x_{j,k}$, associate a ball $B_{j,k} = B(x_{j,k},r_k)$ and a plane $P_{j,k}$ containing $x_{j,k}.$ Assume 
\[\{x_{j,0}\}_{j \in J_0} \subset P_0,\]
and
\[x_{i,k} \in V^2_{k-1},\]
where $V_k^\lambda \coloneqq \bigcup_{j \in J_k} \lambda B_{j,k}.$ Define
\[\varepsilon_k(x) = \sup\{d_{x_{i,l},10^4r_l}(P_{j,k},P_{i,l}) : j \in J_k, \ \abs{l-k} \leq 2, \ i \in J_k, x \in 100B_{j,k} \cap 100B_{i,l} \}.\]
Then, there is $\e_0 >0$ such that if $\e \in (0,\e_0)$ and 
\begin{align}
\varepsilon_k(x_{j,k}) < \e, \ \text{for all} \ k \geq 0 \ \text{and} \ j \in J_k,
\end{align} 
then there is a bijection $f: \R^n \rightarrow \R^n$ such that: 
\begin{enumerate}
\item We have 
\begin{equation}
E_\infty \coloneqq \bigcap_{K=1}^\infty \overline{\bigcup_{k=K}^\infty \{x_{j,k}\}_{j \in J_k}} \subseteq \Sigma \coloneqq f(\R^n).
\end{equation}
\item $f(x) =x$ when $\emph{dist}(z,P_0) >2.$
\item For $x,y \in \R^n$ with $|x-y| \leq 1,$
\begin{equation}
\frac{1}{4}|x-y|^{1+\tau} \leq |f(x)-f(y)| \leq 10 |x-y|^{1-\tau}. 
\end{equation}
\item $|f(x)-x| \lesssim \e$ for $x \in \R^n.$
\item For $x \in P_0,$ $f(x) = \lim_k \sigma_k \circ \dots \circ \sigma_0,$ where
\begin{equation}
\sigma_k(y) = \psi_k(y) + \sum_{j \in J_k} \theta_{j,k}(y)[\pi_{j,k}(y)-y].
\end{equation}
Here, $\{x_{j,k}\}_{j \in L_k}$ is a maximal $\tfrac{r_k}{2}$-separated set in $\R^n \setminus V_k^9,$ 
\[B_{j,k} = B(x_{j,k},r_k/10) \quad \text{for} \ j \in L_k, \]
$\{ \theta_{j,k}\}_{j \in J_k \cup L_k}$ is a partition of unity such that $\mathds{1}_{9B_{j,k}} \leq \theta_{j,k} \leq \mathds{1}_{10B_{j,k}}$ for all $k$ and $j \in L_k \cup J_k,$ and $\psi_k = \sum_{j \in L_k} \theta_{j,k}.$
\item For $k \geq 0,$
\begin{equation}\label{e:10}
\sigma_k(y) = y \ \text{and} \ D\sigma_k(y)=I \ \text{for} \ y \in \R^n \setminus V^{10}_k.
\end{equation}
\item Let $\Sigma_0 = P_0$ and 
\begin{equation}
\Sigma_k = \sigma_k(\Sigma_{k-1}).
\end{equation}
There is a function $A_{j,k}:P_{j,k} \cap 49B_{j,k} \rightarrow P_{j,k}^\perp$ of class $C^2$ such that $\abs{A_{j,k}(x_{j,k})} \lesssim \e r_k, \ \abs{DA_{j,k}} \lesssim \e$ on $P_{j,k} \cap 49B_{j,k},$ and if $\Gamma_{j,k}$ is its graph over $P_{j,k}$, then 
\begin{equation}
\Sigma_k \cap D(x_{j,k},P_{j,k},49r_k) = \Gamma_k \cap D(x_{j,k},P_{j,k},49r_k)
\end{equation}
where
\begin{equation}
D(x_{j,k},P_{j,k},49r_k) = \{z+w : z \in P \cap B(x,r), \ w \in P^\perp \cap B(0,r)\}.
\end{equation}
In particular, 
\begin{equation}\label{Sigma_kPlane}
d_{x_{j,k},49r_{k}}(\Sigma_k , P_{j,k}) \lesssim \e. 
\end{equation}
\item For $k \geq 0$ and $y \in \Sigma_k$, there is an affine $d$-plane $P$ through $y$ and a $C\e$-Lipschitz and $C^2$-function $A:P \rightarrow P^\perp$ so that if $\Gamma$ is the graph of $A$ over $P$, then 
\begin{equation}
\Sigma_k \cap B(y,19r_k) = \Gamma \cap B(y,19r_k). 
\end{equation}
\item Have $\Sigma = f(P_0)$ is $C\e$-Reifenberg flat in the sense that for all $z \in \Sigma,$ and $t \in (0,1),$ there is $P=P(z,t)$ so that $d_{z,t}(\Sigma,P) \lesssim \e.$ 
\item For all $y \in \Sigma_k$,
\begin{align}\label{Sigma_ksigma_k}
\abs{\sigma_k(y) - y} \lesssim \e_k(y) r_k
\end{align}
and moreover,
\begin{align}\label{Sigma_kSigma}
\emph{dist}(y,\Sigma) \lesssim \e r_k, \quad \text{for} \ y \in \Sigma_k
\end{align}
\item For $k \geq 0$, $y \in \Sigma_j \cap V_k^8,$ choose $i \in J_k$ such that $y \in 10B_{i,k}.$ Then
\begin{align}
\abs{\sigma_k(y) - \pi_{i,k}(y)} \lesssim \e_k(y)r_k. 
\end{align}
\item For $x \in \Sigma$ and $r>0,$
\begin{align}\label{e:lowerreg}
\mathscr{H}^d_\infty(\Sigma \cap B(x,r)) \geq (1-C\e)\omega_d r^d
\end{align}
where $\omega_d$ is the volume of the unit ball in $\R^d.$ 
\end{enumerate}
\end{thm}

\begin{rem}
We conjecture the main results hold for subsets of an infinite dimensional Hilbert space. With this in mind, we have tried to make much of our work here dimension free. We shall indicate the places where the estimates depend on the ambient dimension. As such, many of the volume arguments rely on the following result. Put simply, it states that a collection of disjoint balls lying close enough to a $d$-dimensional plane will satisfy a $d$-dimensional packing condition. 
\end{rem}

\begin{lem}[{\cite[Lemma 3.1]{edelen2018effective}}]\label{ENV}
Let $V$ be an affine $d$-dimensional plane in a Banach space $X$, and $\{B(x_i,r_i)\}_{i \in I}$ be a family of pairwise disjoint balls with $r_i \leq R, \ B(x_i,r_i) \subseteq B(x,R),$ for some $x \in \R^n$ and $\emph{dist}(x_i,V) < r_i/2.$ Then, there is a constant $\kappa = \kappa(d)$ such that
\begin{equation}\label{e:ENV}\sum_{i \in I} r_i^d \leq \kappa R^d.
\end{equation}
\end{lem}

\begin{lem}\label{l:children}
Let $E \subseteq \R^n,$ $Q \in \mathscr{D}$ and $M \geq 1.$ If $0 < \e  \leq \tfrac{c_0\rho}{2M}$ and
\[ \beta_{E,\infty}^d(MB_Q) \leq \e, \]
then $Q$ has at most $K = K(M,d)$ children, i.e. independent of $n.$
\end{lem}

\begin{proof}
The balls $\{c_0B_R\}_{R \in \text{Child}(Q)}$ are pairwise disjoint (recall $c_0$ from Lemma \ref{cubes}), contained in $MB_Q$, and have radius less than or equal to $r_{MB_Q}.$ Since 
\[\beta_{E,\infty}^d(MB_Q) \leq \e,\]
there exists a $d$-plane $P_Q$ such that 
\begin{align}
\text{dist}(y,P_Q) \leq \e M\ell(Q)
\end{align}
for all $y \in MB_Q.$ In particular, for any $R \in \text{Child}(Q),$ we have 
\[ \dist(x_R,P_Q) \leq \e M \ell(Q) \leq r_{c_0B_R}/2. \]
By Lemma \ref{ENV}, 
\[ \# \{R: R \in \text{Child}(Q) \}c_0^d\rho^d \ell(Q)^d = \sum_{R \in \text{Child}(Q)} (c_0\ell(R))^d \overset{\eqref{e:ENV}}{\leq} \kappa (M\ell(Q))^d,\]
from which the lemma follows by dividing through by $c_0^d\rho^d\ell(Q)^d.$
\end{proof}

As a simple corollary of the above lemma, we also get a bound on the number of descendants up to a specified generation. The constant here ends up also depending on the generation. 

\begin{lem}\label{l:des}
Let $E \subseteq \R^n$, $Q \in \mathscr{D}$, $M \geq 1$ and $k \geq 0.$ If $0 < \e < \tfrac{c_0 \rho}{2M}$ and 
\[ \beta^{d}_{E,\infty}(MB_Q) \leq \e \]
for all $R \in \emph{Des}_k(Q)$, then 
\begin{align}\label{e:des}
\sum_{R \in \emph{Des}_k(Q)} \ell(R)^d \lesssim_{d,M,k} \ell(Q)^d.
\end{align}
\end{lem}

\subsection{Hausdorff-type content}
In this section we study the Hausdorff content $\mathscr{H}^{d,E}_{B,\infty}$ that we defined in the introduction. For the convenience of the reader, we state the definition again. 

\begin{rem}\label{r:consts}
Let us be explicit about the constant $c_1$ appearing in Theorem \ref{Thm1} and Theorem \ref{Thm4} and the constant $c_2$ appearing in Definition \ref{IntroDef}. We fix now
\[c_1 \coloneqq \frac{\omega_d \rho^d}{2^{d+1}3^d}\]
where $\rho$ is the constant appearing in Theorem \ref{cubes} and $\omega_d$ is the volume of the unit ball in $\R^d.$ We also fix 
\[c_2 \coloneqq 18^d\kappa,\]
with $\kappa$ as in Lemma \ref{ENV}. We shall comment on this choice of constants in Remark \ref{r:LR}. Basically, we have chosen $c_1$ sufficiently small and $c_2$ sufficiently large. 
\end{rem}

\begin{defn}\label{Cover}
Let $E \subseteq \R^n$, $B$ a ball. We say a collection of balls $\mathscr{B}$ which covers $E \cap B$ is \textit{good} if
\begin{align}\label{Size}
\sup_{B' \in \mathscr{B}} r_{B'} < \infty
\end{align}
and for all $x \in E \cap B$ and $0 <r <r_B,$ we have
\begin{align}\label{LR}
\sum_{\substack{B' \in \mathscr{B} \\ B' \cap B(x,r) \cap E \cap B \not= \emptyset}} r_{B'}^d \geq c_1 r^d
\end{align}
and
\begin{align}\label{UR}
\sum_{\substack{B' \in \mathscr{B} \\ B' \cap B(x,r) \cap E \cap B \not=\emptyset \\ r_{B'} \leq r}} r_{B'}^d \leq c_2 r^d.
\end{align}
Then, for $A \subseteq E \cap B,$ define 
\begin{align*}
\mathscr{H}^{d,E}_{B,\infty}(A) = \inf\left\{ \sum_{\substack{B^\prime \in \mathscr{B}\\ B^\prime \cap A \not= \emptyset}} r_{B^\prime}^d : \mathscr{B} \ \text{is good for} \ E \cap B \right\}
\end{align*}
\end{defn}

See Figure \ref{f:example} for an example of a good cover. Let us make some remarks concerning the above definition. 

\begin{rem}
For the usual Hausdorff content, $\mathscr{H}^d_\infty,$ all coverings of a set are permissible (see \eqref{e:HausCont}). In defining our new content, we restrict the permissible coverings to ensure all sets will have a lower regularity property with respect to this content (this is the role of \eqref{LR}). In addition, we require an upper regularity condition, \eqref{UR}. This is to ensure any cover we choose is sensible. In particular, it stops us constructing lower regular covers by just repeatedly adding the same ball over and over again. 
\end{rem}

\begin{rem}
We include \eqref{Size} because we want all balls in $\mathscr{B}$ to be contained in a bounded region.  
\end{rem}

\begin{rem}
If $E \subseteq \R^n$ and $B$ is a ball, then $\mathscr{B} = \{B\}$ is a good cover for $E \cap B.$ In particular, every set has a good cover. 
\end{rem}

\begin{figure}[ht]
  \centering
  \includegraphics[scale = 0.8]{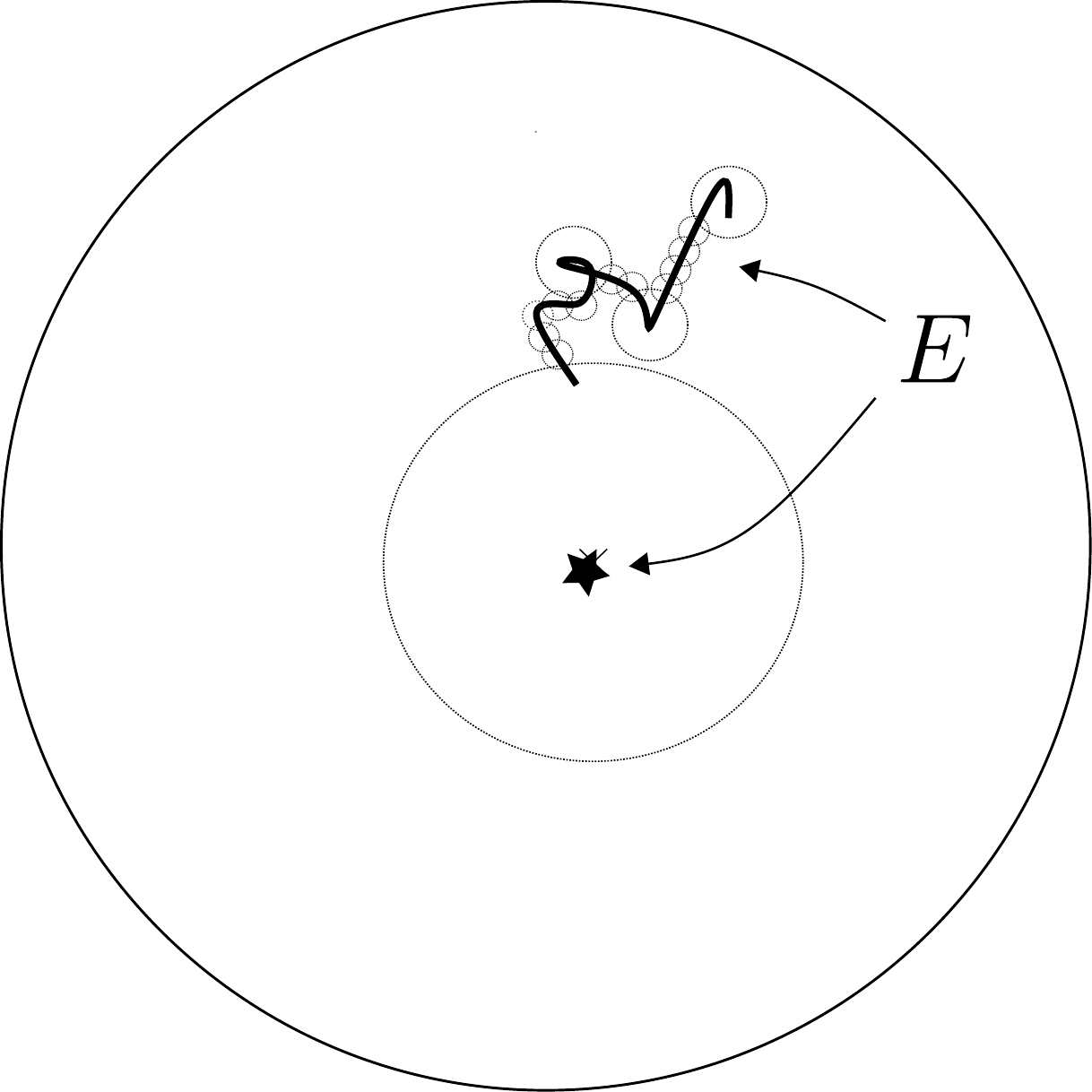}
\caption{Example of a good cover for $E \cap B$. Notice we need to cover the star segment with a large ball to ensure the cover is lower regular. We could equally have added many smaller balls as long as the upper regularity condition is not violated.}\label{f:example}
\end{figure}

\begin{rem}
It is easy to see that for any $A \subseteq E \cap B,$ we have $\mathscr{H}^d_\infty(A) \leq \mathscr{H}^{d,E}_{B,\infty}(A).$
\end{rem}

We shall now prove some basic properties of $\mathscr{H}^{d,E}_{B,\infty}.$ Before doing so we need some preliminary lemmas. The first is \cite[Lemma 2.5]{mattila1999geometry} and the second is modification of \cite[Lemma 2.6]{mattila1999geometry}, whose proof is essentially the same.

\begin{lem}\label{l:Mat}
Suppose $a,b \in \R^2, \ 0 < |a| \leq |a-b|$ and $0 < |b| \leq |a-b|$. Then the angle between the vectors $a$ and $b$ is at least $60^\circ$, that is,
\begin{align}
|a/|a| - b/|b|| \geq 1. 
\end{align}
\end{lem}

\begin{lem}\label{l:overlap}
There is $N(n) \in \N$ with the following property. Let $B$ be a ball and suppose there are $k$ disjoint balls $B_1,\dots,B_k$ such that $r_{B_i} \geq r_B$ and $B \cap B_i \not=\emptyset$ for all $i=1,\dots,k.$ Then $k \leq N(n).$
\end{lem}

\begin{proof}
We may assume $B$ is centred at the origin. If one of the $B_i$ is centred at the origin then $k =1$, so assume this is not the case. Let $B_i = B(x_i,r_i).$ For $i \not=j$, since $B_i \cap B_j =\emptyset$, we have
$$|x_i - x_j| > r_i + r_j > r_i + r_B,$$
and so
\begin{align}
0< |x_i| \leq r_B + r_i \leq |x_i - x_j| \quad \text{for} \ i\not=j. 
\end{align}
Applying Lemma \ref{l:Mat} with $a =x_i$ and $b = x_j$ for $i \not=j$ in the two dimensional plane containing $0,x_i,x_j$, we obtain 
\begin{align}
| x_i/|x_i| - x_j/|x_j|| \geq 1 \quad \text{for} \ i \not=j.
\end{align}
Since the unit sphere $S^{n-1}$ is compact there are at most $N(n)$ such points.
\end{proof}

\begin{lem}\label{HausEasy}
Let $E \subseteq \R^n,$ and $B$ be a ball. Then,
\begin{enumerate}
\item $\mathscr{H}^{d,E}_{B ,\infty}(E \cap B(x,r)) \geq c_1 r^d$ for all $x \in E \cap B, \  0< r \leq r_B.$ 
\item If $A_1 \subseteq A_2 \subseteq E \cap B$, then $\mathscr{H}^{d,E}_{B ,\infty}(A_1) \leq \mathscr{H}^{d,E}_{B ,\infty}(A_2).$
\item If $B_1 \subseteq B_2$  and $A \subseteq E \cap B_1,$ then  $\mathscr{H}^{d,E}_{B_1,\infty}(A) \leq \mathscr{H}^{d,E}_{B_2,\infty}(A)$. 
\item Suppose $E \cap B = E_1 \cup E_2.$ Then $\mathscr{H}^{d,E}_{B,\infty}(E \cap B) \lesssim \mathscr{H}^{d,E}_{B,\infty}(E_1) +  \mathscr{H}^{d,E}_{B,\infty}(E_2).$ 
\end{enumerate}
\end{lem}
\begin{proof}
Property (1) is an immediate consequence of Definition \ref{Cover} since any good cover $\mathscr{B}$ of $E \cap B$ satisfies \eqref{LR}. Property (2) is also clear from Definition \ref{Cover}. If $B_1 \subseteq B_2$ then any good cover for $B_2$ is also a good cover for $B_1,$ and (3) follows. 

Let us prove (4). By scaling and translating, we may prove (4) for $\B,$ the unit ball in $\R^n$ centred at the origin. Let $\e > 0$ and suppose $\mathscr{B}_i, \  i=1,2$ are good covers of $E \cap \B$ such that
\begin{align}\label{e:12.1}
\sum_{\substack{B \in \mathscr{B}_i \\ B \cap E_i \not=\emptyset}} r_B^d \leq \mathscr{H}^{d,E}_{\B,\infty}(E_i) + \e/2.
\end{align}
Let $\mathscr{B}_1^\prime$ be the collection of balls $B \in \mathscr{B}_1$ such that $B \cap E_1 = \emptyset.$ It suffices to show
\begin{align}\label{l:qwe}
\sum_{B \in \mathscr{B}_1'} r_{B}^d \lesssim \sum_{\substack{B' \in \mathscr{B}_2 \\ B' \cap E_2 \not=\emptyset}} r_{B'}^d
\end{align}
since if \eqref{l:qwe} were true then
\begin{align}
\mathscr{H}^{d,E}_{\B,\infty}(E \cap \B) &\leq \sum_{\substack{B \in \mathscr{B}_1 \\ B \cap E \cap \B \not=\emptyset}} r_{B}^d = \sum_{\substack{B \in \mathscr{B}_1 \\ B \cap E_1 \not=\emptyset}} r_{B}^d + \sum_{B \in \mathscr{B}_1^\prime} r_{B}^d \\
& \lesssim  \sum_{\substack{B \in \mathscr{B}_1 \\ B \cap E_1 \not=\emptyset}} r_{B}^d + \sum_{\substack{B' \in \mathscr{B}_2 \\ B' \cap E_2 \not=\emptyset}} r_{B^\prime}^d \\
&\stackrel{\eqref{e:12.1}}{\leq} \mathscr{H}^{d,E}_{\B,\infty}(E_1) + \mathscr{H}^{d,E}_{\B,\infty}(E_2) + \e.
\end{align}
Since $\e$ was arbitrary we get (4). 

For the remainder of the proof, we focus on \eqref{l:qwe}. To begin with, we partition $\mathscr{B}_1^\prime$ into two further collection. Define
\begin{align}
\mathscr{B}^\prime_{1,1} = \{B \in \mathscr{B}^\prime_1 : \text{there is $B^{\prime} \in \mathscr{B}_2$ with} \ B \cap B^{\prime} \cap E \cap \B\not=\emptyset \ \text{and} \ r_{B^{\prime}} \geq r_{B} \}
\end{align}
and
\begin{align}
\mathscr{B}^\prime_{1,2} = \{B \in \mathscr{B}^\prime_1 : r_{B^{\prime}} < r_{B} \ \text{for all $B^{\prime} \in \mathscr{B}_2$ such that} \ B \cap B^{\prime} \cap E \cap \B\not=\emptyset \}.
\end{align}
\begin{figure}
  \centering
  \includegraphics[scale=0.5]{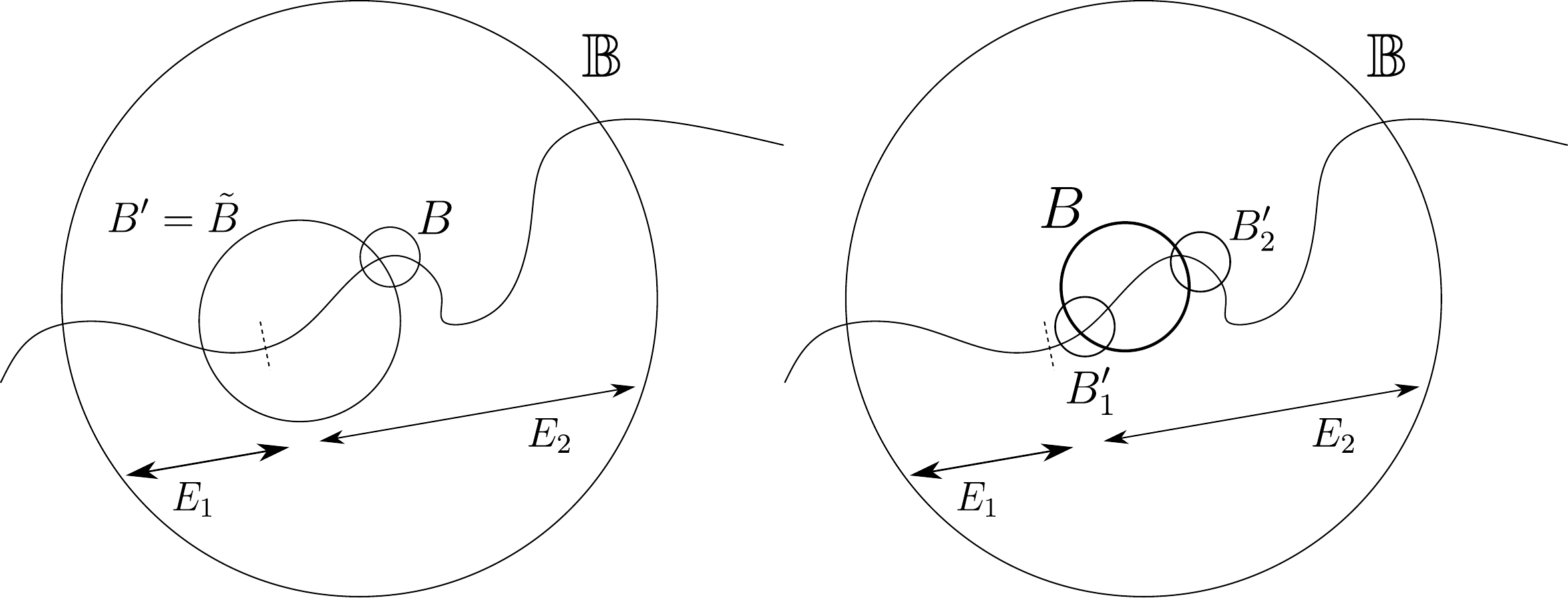}
\caption{Illustration of the balls in $\mathscr{B}_{1,1}'$ (left) and $\mathscr{B}_{1,2}'$ (right).}\label{f:BasicProp}
\end{figure}
See Figure \ref{f:BasicProp}. For $B \in \mathscr{B}_{1,1}^\prime$ let $\tilde{B}$ be the ball in $\mathscr{B}_2$ such that $\tilde{B} \cap B \cap E \cap \B\not=\emptyset$ and $r_{\tilde{B}} \geq r_{B}.$ If there is more than one such ball we can choose $\tilde{B}$ arbitrarily. Since $B \cap E_1 = \emptyset$ and $E \cap \B = E_1 \cup E_2,$ it must be that $B \cap E \cap \B \subseteq E_2,$ hence $\tilde{B} \cap E_2 \not=\emptyset.$  Then 
\begin{align}\label{e:flight}
\sum_{B \in \mathscr{B}^\prime_{1,1}} r_{B}^d 
&= \sum_{\substack{B^{\prime} \in \mathscr{B}_2 \\ B^{\prime} \cap E_2 \not=\emptyset}} \sum_{\substack{B \in \mathscr{B}_{1,1}^\prime \\ \tilde{B} = B^{\prime}}} r_{B}^d 
\leq \sum_{\substack{B^{\prime} \in \mathscr{B}_2 \\ B^{\prime} \cap E_2 \not=\emptyset}} \sum_{\substack{B \in \mathscr{B}_1 \\ B \cap B^{\prime} \cap E \cap \B\not=\emptyset \\ r_{B} \leq r_{B^{\prime}}}} r_{B}^d \\
&\overset{\eqref{UR}}{\leq} c_2 \sum_{\substack{B^{\prime} \in \mathscr{B}_2 \\ B^{\prime} \cap E_2 \not=\emptyset}} r_{B^{\prime}}^d.
\end{align}
We turn our attention to $\mathscr{B}_{1,2}'$. Let $B_1$ be the largest ball in $\mathscr{B}_{1,2}^\prime$. Then, given $B_1,\dots,B_k,$ define $B_{k+1}$ to be the largest ball $B \in \mathscr{B}_{1,2}^\prime$ such that
\begin{align}
E \cap \B \cap B \cap \bigcup_{i=1}^k B_i = \emptyset.
\end{align}
Let $\{B_i\}$ be the resulting (possibly infinite) disjoint collection of balls. By \eqref{Size} each ball in $\mathscr{B}_{1,2}'$ is contained in some bounded region, so a standard volume argument shows that for any $R > 0$ there at most a finite number of disjoint balls $B \in \mathscr{B}_{1,2}'$ such that $r_{B} \geq R$ and $E \cap \B \cap B \not=\emptyset$. Hence, for any $B \in \mathscr{B}_{1,2}^\prime$ there exists $B_i$ such that $B \cap B_i \cap E \cap \B \not=\emptyset.$ Moreover, for any $B_i$ such that $B \cap B_i \cap E \cap \B \not=\emptyset$, we have $r_{B} \leq r_{B_i}.$ Thus,
\begin{align}
\sum_{B \in \mathscr{B}_{1,2}^\prime} r_{B}^d \leq \sum_{i = 1}^\infty \sum_{\substack{B \in \mathscr{B}_{1,2}^\prime \\ B \cap B_i \cap E \cap \B \not= \emptyset}} r_{B}^d \overset{\eqref{UR}}{\leq} c_2 \sum_{i=1}^\infty r_{B_i}^d. 
\end{align}
By Lemma \ref{l:overlap}, for any $B^\prime \in \mathscr{B}_2,$ we have
\begin{align}\label{e:12.2}
\#\{B_i : B^\prime \cap B_i \cap E \cap \B \not=\emptyset\} \lesssim_n 1.
\end{align}
As before, since $B_i \cap E_1 = \emptyset,$ if $B' \in \mathscr{B}_2$ satisfies $B' \cap B_i \cap E \cap \B \not=\emptyset$ for some $i,$ then $B' \cap E_2 \not=\emptyset.$ Hence
\begin{align}
\sum_{B \in \mathscr{B}_{1,2}^\prime} r_B^d 
&\leq c_2 \sum_{i=1}^\infty r_{B_i}^d 
\overset{\eqref{LR}}{\leq} \frac{c_2}{c_1} \sum_{i=1}^\infty \sum_{\substack{B^\prime \in \mathscr{B}_2 \\ B^\prime \cap E_2 \not=\emptyset \\ B^\prime \cap B_i \cap E \cap \B\not= \emptyset }} r_{B^\prime}^d \\
&\leq \frac{c_2}{c_1}\sum_{\substack{B^\prime \in \mathscr{B}_2 \\ B^\prime \cap E_2 \not=\emptyset}} \sum_{\substack{i \\ B^\prime \cap B_i \cap E \cap \B \not=\emptyset}} r_{B^\prime}^d \stackrel{\eqref{e:12.2}}{\lesssim} \frac{c_2}{c_1} \sum_{\substack{B^\prime \in \mathscr{B}_2 \\ B^\prime \cap E_2 \not=\emptyset}} r_{B^\prime}^d.
\end{align}
Combining this with \eqref{e:flight} finishes the proof of \eqref{l:qwe}.
\end{proof}

For a function $f:\R^n \rightarrow [0,\infty),$ define integration with respect to $\mathscr{H}^{d,E}_{B,\infty}$ via the Choquet integral:
\begin{align}
\int f \, d\mathscr{H}^{d,E}_{B,\infty} \coloneqq \int_0^\infty  \mathscr{H}^{d,E}_{B,\infty}(\{x \in E \cap B : f(x) > t \}) \, dt.
\end{align}

We state some basic properties of the above Choquet integral, see \cite{wang2011some} for more details. Note, in \cite{wang2011some} there are additional upper and lower continuity assumption, but these are not required for the following lemma.

\begin{lem}\label{Choquet}
Let $f,g: \R^n \rightarrow [0,\infty)$ such that $f \leq g$ and $\alpha >0$. Then 
\begin{enumerate}
\item $\int f \, d\mathscr{H}^{d,E}_{B,\infty} \leq  \int g \, d\mathscr{H}^{d,E}_{B,\infty};$
\item $ \int (f + \alpha) \, d\mathscr{H}^{d,E}_{B,\infty} = \int f  \, d\mathscr{H}^{d,E}_{B,\infty} + \alpha \mathscr{H}^{d,E}_{B,\infty}(E\cap B);$
\item $\int \alpha f  \, d\mathscr{H}^{d,E}_{B,\infty} =\alpha \int f \, d\mathscr{H}^{d,E}_{B,\infty}.$
\end{enumerate}
\end{lem}

The Choquet integral also satisfies a Jensen-type inequality. The proof is based on the proof of the usual Jensen's inequality for general measures, see \cite{rudin2006real}. 

\begin{lem}\label{JensenPhi}
Suppose $E,B \subseteq \R^n$, $\phi : \R \rightarrow \R$ is convex and $f : \R^n \rightarrow \R$ is bounded. Then,
\begin{align}\label{e:Jensen}
\phi \left( \frac{1}{\mathscr{H}^{d,E}_{B,\infty}(E\cap B) }\int f \, d\mathscr{H}^{d,E}_{B,\infty} \right) \leq \frac{1}{\mathscr{H}^{d,E}_{B,\infty}(E\cap B)}\int \phi \circ f \, d\mathscr{H}^{d,E}_{B,\infty}.
\end{align}
\end{lem} 

\begin{proof}
Since $f$ is bounded and $\mathscr{H}^{d,E}_{B,\infty}(E \cap B) \gtrsim r_B^d$, we can set 
\[t = \frac{1}{\mathscr{H}^{d,E}_{B,\infty}(E\cap B) }\int f \, d\mathscr{H}^{d,E}_{B,\infty} < \infty.\] 
Since $\phi$ is convex, if $-\infty < s <t <u <\infty$, then 
\begin{align}\label{Jensen1}
\frac{\phi(t)-\phi(s)}{t-s} \leq \frac{\phi(u)-\phi(t)}{u-t}.
\end{align}
Let $\gamma$ be the supremum of the left hand side of \eqref{Jensen1} taken over all $s \in (-\infty,t).$ It is clear then that 
\begin{align}
\phi(t) \leq \phi(s) + \gamma \cdot (t -s)
\end{align}
for all $s \in \R.$ By rearranging the above inequality we have that for any $x \in \R^n,$   
\begin{align}
\gamma f(x) + \phi(t) \leq \phi(f(x)) + \gamma t. 
\end{align}
Integrating both sides with respect to $x$ and using Lemma \ref{Choquet}, we have 
\begin{align}
\gamma \int f \, d\mathscr{H}^{d,E}_{B,\infty} + \mathscr{H}^{d,E}_{B,\infty}(E \cap B) \phi(t) \leq  \int \phi \circ f \, d\mathscr{H}^{d,E}_{B,\infty}  + \gamma \mathscr{H}^{d,E}_{B,\infty}(E\cap B) t,
\end{align}
thus,
\begin{align}
\mathscr{H}^{d,E}_{B,\infty}(E \cap B) \phi(t) \leq \int \phi \circ f \, d\mathscr{H}^{d,E}_{B,\infty} + \gamma \left(\mathscr{H}^{d,E}_{B,\infty}(E\cap B) t - \int f \, d\mathscr{H}^{d,E}_{B,\infty} \right).
\end{align}
By definition of $t$, the second term on the right hand side of the above inequality is zero, from which the lemma follows. 

\end{proof}

Integration with respect to $\mathscr{H}^d_\infty$ can be defined similarly and satisfies identical properties. In the special case of integration with respect to $\mathscr{H}^d_\infty$ we also have the following:
\begin{lem}[{\cite[Lemma 2.1]{azzam2018analyst}}] \label{l:subsum}
Let $0 < p< \infty.$ Let $f_i$ be a countable collection of Borel functions in $\R^n.$ If the sets $\emph{supp} \, f_i = \{f_i > 0\}$ have bounded overlap, meaning there exists a $C < \infty$ such that
\[ \sum \mathds{1}_{\emph{supp} \, f_i} \leq C,\]
then
\begin{align}
\int \left( \sum f_i \right)^p \, d\mathscr{H}^d_\infty \leq C^p \sum \int f_i^p \, d\mathscr{H}^d_\infty.
\end{align}
\end{lem}

\subsection{Preliminaries with \texorpdfstring{$\beta$}{b}-numbers} In this section we prove some basic properties of $\beta^{d,p}_E.$ Again, we restate the definition for the readers convenience. 
\begin{defn}\label{DefBeta}
Let $1 \leq p <\infty$, $E \subseteq \R^n,$ $B$ a ball centred on $E$ and $L$ a $d$-plane. Define
\begin{align*}
\beta_E^{d,p}(B,L)^p &= \frac{1}{r_B^d} \int \left( \frac{\text{dist}(x,L)}{r_B}\right)^p \, d\mathscr{H}^{d,E}_{B,\infty} \\
&= \frac{1}{r_B^d} \int_0^1 \mathscr{H}^{d,E}_{B,\infty}(\{x \in E \cap B : \text{dist}(x,L) >tr_B \}) t^{p-1} \, dt,
\end{align*}
and 
\begin{align}
\beta_E^{d,p}(B) = \inf \{ \beta_E^{d,p}(B,L) : L \ \text{is a $d$-plane} \}. 
\end{align}
\end{defn}

\begin{rem}\label{r:ineq}
It is easy to show that $\mathscr{H}_\infty^d \leq \mathscr{H}^{d,E}_{B,\infty}$ which implies $\check\beta_E^{d,p} \leq \beta^{d,p}_E.$ If $c \geq c_1$ and $E$ is a $(c,d)$-lower regular set, we get the reverse inequality up to a constant, see Corollary \ref{CheckComp}.
\end{rem}

\begin{rem}\label{r:LR}
It is possible to define a $\beta$-number $\beta^{d,p,c}_E$ where the additional parameter $c$ replaces the fixed lower regularity constant $c_1$ from \eqref{LR}. We could prove a version of Theorem \ref{Thm4} for $\beta^{d,p,c},$ that is, if $E \subseteq F$ and the sum of these $\beta^{d,p,c}$ coefficients for $E$ is finite, then we can find a $(c_1,d)$-lower content regular set $F$ such that $E \subseteq F$ (notice that the regularity constant for $F$ is independent of $c$). Thus to prove Theorem \ref{c:TSP} we could just as well have used $\beta^{d,p,c}$ instead of $\beta^{d,p}$ for any $c \leq c_1.$ We have fixed our lower regularity parameter for clarity.  
\end{rem}

\begin{lem}\label{l:p}
Let $1 \leq p < \infty, \ E \subseteq \R^n$ and $B$ a ball centred on $E$. Then
\begin{align}\label{e:p}
\beta_E^{d,1}(B)_{d,p} \lesssim \beta_E^{d,p}(B)
\end{align}
\end{lem}

\begin{proof}
This is a direct consequence of Lemma \ref{JensenPhi} and Definition \ref{DefBeta}. The inequality obviously holds for $p=1,$ so assume $p >1.$ Then the map $x \mapsto x^p$ is convex. Let $P$ be the $d$-plane such that $\beta_E^{d,p}(B) = \beta_E^{d,p}(B,P).$ Since $\text{dist}(x,P) \leq 2r_B$ for all $x \in E \cap B$ and $\mathscr{H}^{d,E}_{B,\infty}(E \cap B) \sim r_B^d,$ we have, by Lemma \ref{JensenPhi},
\begin{align}
\beta^{d,1}_E(B) &\leq \frac{1}{r_B^d} \int  \left(\frac{\text{dist}(x,P)}{r_B} \right)\, d\mathscr{H}^{d,E}_{B,\infty} \overset{\eqref{e:Jensen}}{\lesssim} \left(\frac{1}{r_B^d} \int  \left( \frac{\text{dist}(x,P)}{r_B}\right)^p \, d\mathscr{H}^{d,E}_{B,\infty}\right)^\frac{1}{p} \\
&= \beta^{d,p}_E(B). 
\end{align}
\end{proof}

\begin{lem}\label{l:cont}
Let $1 \leq p < \infty$ and $E \subseteq \R^n.$ Then, for all balls $B^\prime \subseteq B$ centred on $E,$
\begin{align}
\beta^{d,p}_E(B^\prime) \leq \left(\frac{r_B}{r_{B^\prime}} \right)^{1+\frac{d}{p}} \beta_E^{d,p}(B). 
\end{align}
\end{lem}

\begin{proof}
Let $P$ be the $d$-plane such that $\beta_E^{d,p}(B) = \beta_E^{d,p}(B,P).$ By Lemma \ref{HausEasy} (2),(3) and a change of variables, we have
\begin{align}
\beta^{d,p}_E(B^\prime)^p &\leq \frac{1}{r_{B^\prime}^d} \int_0^1 \mathscr{H}^{d,E}_{B^\prime,\infty}(\{x \in E \cap B^\prime : \text{dist}(x,P) >tr_{B^\prime} \}) t^{p-1} \, dt \\
&\stackrel{(2),(3)}{\leq} \frac{r_B^d}{r_{B^\prime}^d} \frac{1}{r_B^d}\int_0^1 \mathscr{H}^{d,E}_{B,\infty}(\{x \in E \cap B : \text{dist}(x,P) >tr_{B^\prime} \}) t^{p-1} \, dt \\
&\leq \left(\frac{r_B}{r_{B^\prime}} \right)^{d+p}  \frac{1}{r_B^d}\int_0^1 \mathscr{H}^{d,E}_{B,\infty}(\{x \in E \cap B : \text{dist}(x,P) >tr_{B} \}) t^{p-1} \, dt \\
&= \beta_E^{d,p}(B)^p. 
\end{align}
\end{proof}

By restricting the covers of our sets as we have done, we lose monotonicity of $\beta^{d,p}_\cdot(B,L)$ with respect to set inclusion. This is because we are imposing the condition that all sets have large measure, even singleton points.  We illustrate this point with an example. Let $\e>0$, $\B$ be the unit ball centred at the origin and $L$ a $d$-plane through the origin. Consider a set $F$ consisting of the $d$-plane $L$ with the segment through the origin replaced by two sides of an equilateral triangle with height $\e$ and the set $E \subseteq F$ which is just the singleton point at the tip of the equilateral triangle (see Figure \ref{example}). Since $F$ is lower $1$-regular and we have comparability for $\check\beta^{d,p}$ and $\beta^{d,p}$ on lower regular sets (see Corollary \ref{CheckComp}), it is easy to show
\[\beta_{F}^{1,1}(\B,L) \lesssim \check\beta_F^{1,1}(2\B,L) \sim \epsilon^2.\] 
On the other hand, by Lemma \ref{HausEasy} (1), we must have 
\begin{align}
\beta_E^{1,1}(\B,L) &= \int_0^1 \mathscr{H}^{1,E}_{\B,\infty}( \{x \in E \cap \B : \dist(x,L) \geq t \}) \, dt \\
&= \int_0^\e \mathscr{H}^{1,E}_{\B,\infty}(E \cap \B) \, dt \stackrel{(1)}{\gtrsim} \epsilon,
\end{align}
which for $\epsilon$ small enough implies 
\[ \beta_F^{1,1}(\B,L) \leq \beta_E^{1,1}(\B,L). \] 

\begin{figure}
  \centering
  \includegraphics[scale = 0.7]{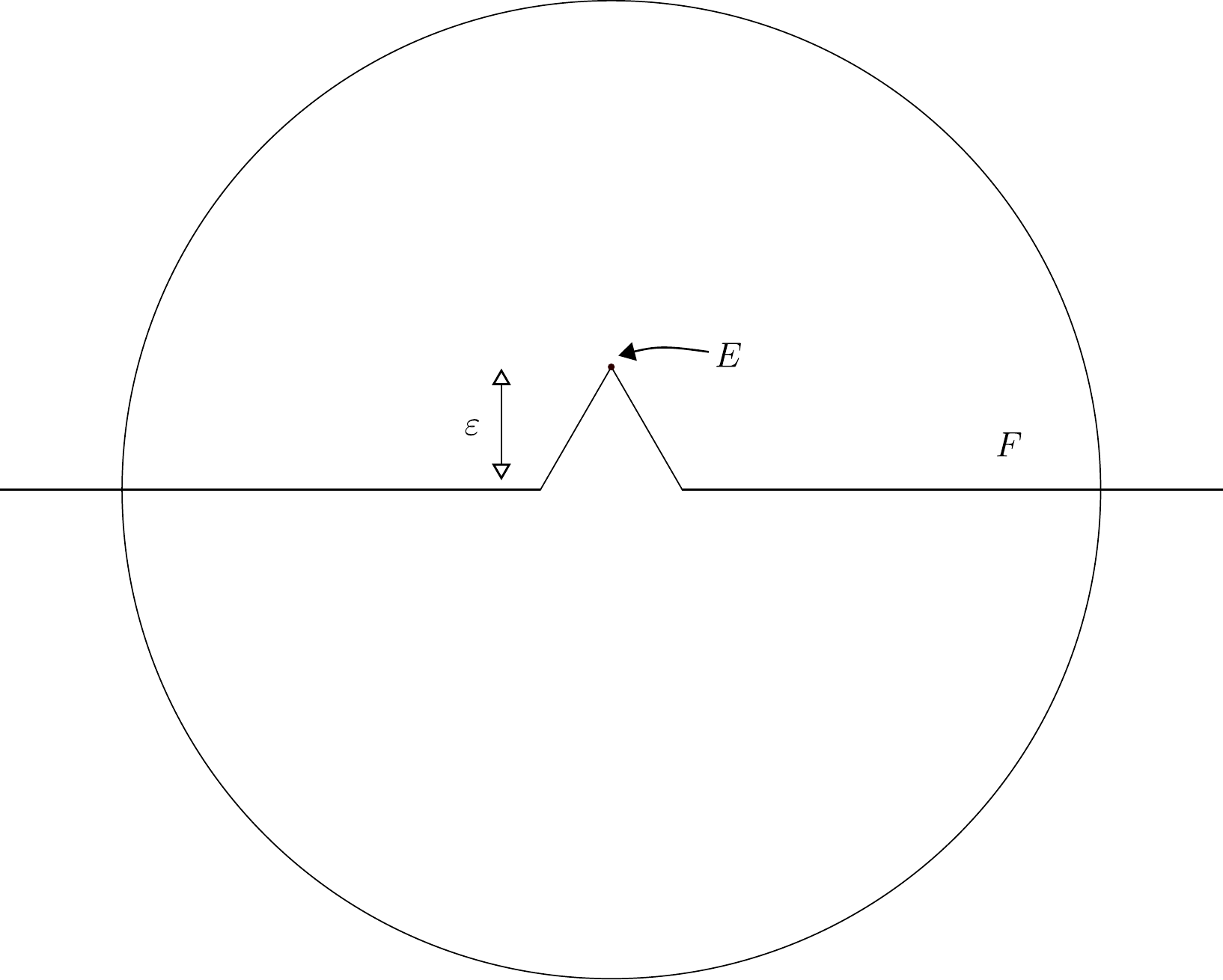}
\caption{$\beta^{d,p}_\cdot$ is not monotone.}\label{example}
\end{figure}

We do however have the following result at least in the case where the larger set is lower regular. Roughly speaking it states that if $E \subseteq F$ and $F$ is lower regular, then we can control the $\beta$-number of $E$ by the $\beta$-number of $F$ with some error term dependent on the average distance of $E$ from $F.$ 
\begin{lem}\label{SubsetError}
Let $F \subseteq \R^n$ be $(c,d)$-lower content regular for some $c \geq c_1$ and let $E \subseteq F.$ Let $B$ be a ball and $L$ a $d$-dimensional plane. Then 
\begin{align}
\beta_E^{d,p}(B,L) \lesssim_{c,d,p} \beta_F^{d,p}(2B,L) + \left(\frac{1}{r_B^d} \int_{F \cap 2B} \left(\frac{\emph{dist}(x,E)}{r_B}\right)^p \, d\mathscr{H}_\infty^d(x) \right)^\frac{1}{p}.
\end{align}
\end{lem}
\begin{proof}

We actually prove a stronger version of the above statement, with $\beta^{d,p}_F$ replaced by $\check\beta^{d,p}_F.$ Note, by re-scaling and translating, we may assume $B = \B.$ For $t > 0$ let
\[E_t = \{x \in E \cap \B : \dist(x,L)>t \}. \]
To prove the lemma, it suffices to show 
\begin{align}\label{SumEst}
\begin{split}
\mathscr{H}_{\B,\infty}^{d,E}(E_t) &\lesssim \mathscr{H}^d_\infty \left( \{x \in F \cap 2\B : \text{dist}(x,L) > t/150\}\right) \\
&\hspace{4em} + \mathscr{H}^d_\infty \left(\{x \in F \cap 2\B : \text{dist}(x,E) > t/150\}\right)
\end{split}
\end{align}
since if the above were true,
\begin{align*}
\beta_E^{d,p}(\B,L)^p &= \int_0^1 \mathscr{H}_{\B,\infty}^{d,E}(E_t) t^{p-1} \, dt \\
&\overset{\eqref{SumEst}}{\lesssim} \int_0^1 \mathscr{H}_\infty^d (\{x \in F \cap 2\B : \text{dist}(x,L) > t/150 \}) t^{p-1} \, dt \\
&\hspace{4em} +  \int_0^1 \mathscr{H}_\infty^d (\{x \in F \cap 2\B : \text{dist}(x,L) > t/150 \}) t^{p-1} \, dt \\
&\lesssim \int_0^1 \mathscr{H}_\infty^d (\{x \in F \cap 2\B : \text{dist}(x,L) > 2t \}) t^{p-1} \, dt \\
&\hspace{4em} + \int_0^1 \mathscr{H}_\infty^d (\{x \in F \cap 2\B : \text{dist}(x,L) > 2t\}) t^{p-1} \, dt \\
&\sim\check\beta_F^{d,p}(2\B,L)^p + \int_{F \cap 2\B} \dist(x,E)^p \, d\mathscr{H}_\infty^d(x).
\end{align*}

For the rest of the proof we focus on \eqref{SumEst}. Fix $t>0.$ We must first construct a suitable good cover for $E \cap \B.$ For $x \in F \cap \B,$ let 
\begin{align}\label{e:delta}
\delta(x) = \max\{\text{dist}(x,L),\text{dist}(x,E) \} + t/120
\end{align}
and let $X = \{x_i\}_{i \in I}$ be a maximal net in $F \cap \B$ such that 
\begin{align}\label{e:sep}
|x_i-x_j| \geq 4\max\{\delta(x_i),\delta(x_j)\}
\end{align}
for all $i \not=j.$ For each $i \in I$ let $B^\prime_i = B(x_i,4\delta(x_i))$ and $\mathscr{B}^\prime = \{B_i^\prime\}_{i \in I}.$ By \eqref{e:delta}, $\delta(x)>0$ for all $x \in F \cap \B,$ so the balls $B_i'$ are non-degenerate. Furthermore, define 
$$\mathscr{B} = \{3B_i'\}_{i \in I} = \{B_i\}_{i \in I}.$$ 
\textbf{Claim:} $\mathscr{B}$ is a good cover for $E \cap \B.$ 
\bigbreak
\noindent
Since $E \subseteq F,$ it follows that $\mathscr{B}$ covers $E$ by maximality. It is also clear that \eqref{Size} holds. We are left to verify \eqref{LR} and \eqref{UR}. Let $x \in E \cap \B$ and $0 < r< 1.$ We look first at \eqref{LR}. 

Assume $B^\prime_i \cap B(x,r/3) \cap F \not=\emptyset$ for some $i \in I.$ If $4\delta(x_i) \geq r/3$ then $B(x,r/3) \subseteq 3B_i^\prime,$ hence $3B^\prime_i \cap B(x,r/3) \cap E \not=\emptyset.$ If $4\delta(x_i) <r/3$, then $B^\prime_i \subseteq B(x,r)$ and since $\text{dist}(x_i,E) \leq \delta(x_i)$ there exists some $y \in E \cap B(x,r) \cap B^\prime_i $, in particular, $B^\prime_i \cap B(x,r) \cap E \not= \emptyset.$ In either case, we conclude $E \cap B(x,r) \cap 3B^\prime_i \not=\emptyset.$ By definition of $\mathscr{B}$, this means there exists $B_i \in \mathscr{B}$ such that $E \cap B(x,r) \cap B_i \not=\emptyset.$ Then, since $F$ is $(c,d)$-lower content regular in $\B$, we have
\begin{align*}
\sum_{\substack{B \in \mathscr{B} \\ B \cap E \cap B(x,r) \not=\emptyset}} \left(\frac{r_B}{3}\right)^d \geq \sum_{\substack{B^\prime \in \mathscr{B}^\prime \\ B^\prime \cap F \cap B(x,r/3) \not=\emptyset}} r_{B^\prime}^d \geq c\left(\frac{r}{3}\right)^d, 
\end{align*}
hence $\mathscr{B}$ satisfies the lower bound \eqref{LR}. 

Now for the upper bound \eqref{UR}. If $B \in \mathscr{B}$ satisfies $B \cap B(x,r) \cap E \not=\emptyset$ and $r_B \leq r$, then $B \subseteq B(x,3r)$. Furthermore, since the balls $\{\tfrac{1}{6}B\}_{B\in \mathscr{B}}$ are disjoint and satisfy $\text{dist}(x_B,L) \leq r_{\frac{1}{6}B}/2,$ we have by Lemma \ref{ENV},
\[
\sum_{\substack{B \in \mathscr{B} \\ B \cap E \cap B(x,r) \not=\emptyset \\ r_B \leq r}} r_B^d \leq 6^d\sum_{\substack{B \in \mathscr{B} \\ B \subseteq B(x,3r)}} r_{\frac{1}{6}B}^d \overset{\eqref{e:ENV}}{\leq} 18^d\kappa r^d
\]
Since $c \geq c_1$ and, recalling that $c_2 =18^d \kappa,$ it follows that $\mathscr{B}$ satisfies \eqref{UR}, thus, $\mathscr{B}$ is a good cover for $E \cap \B$ which proves the claim.
\bigbreak
We partition the balls in $\mathscr{B}$ as follows: let
\begin{align*}
\mathscr{B}_E &= \{B_i \in \mathscr{B} : r_i = 12\text{dist}(x_i,E) + t/10\}; \\
\mathscr{B}_L &= \{B_i \in \mathscr{B} : r_i = 12\text{dist}(x_i,L) + t/10\}.
\end{align*}
If $\dist(x_i,E) = \dist(x_i,L)$ then we put $B_i$ in $\mathscr{B}_E$ or $\mathscr{B}_L$ arbitrarily. Since $\mathscr{B}$ is good for $E \cap \B,$ we have
\begin{align}\label{HausComp1}
\mathscr{H}_{\B,\infty}^{d,E}(E_t) \leq \sum_{\substack{B \in \mathscr{B} \\ B \cap E_t \not= \emptyset}}r_B^d =  \sum_{\substack{B \in \mathscr{B}_E \\ B \cap E_t \not= \emptyset}}r_B^d +  \sum_{\substack{B \in \mathscr{B}_L \\ B \cap E_t \not= \emptyset}}r_B^d. 
\end{align}
If we can show that 
\begin{align}\label{Balls_E}
\sum_{\substack{B \in \mathscr{B}_E \\ B \cap E_t \not= \emptyset}} r_B^d \lesssim  \mathscr{H}^d_\infty \left(\{x \in F \cap 2\B : \text{dist}(x,E) > t/150\}\right)
\end{align}
and
\begin{align}\label{Balls_L}
\sum_{\substack{B \in \mathscr{B}_L \\ B \cap E_t \not= \emptyset}} r_B^d \lesssim  \mathscr{H}^d_\infty \left(\{x \in F \cap 2\B : \text{dist}(x,L) > t/150\}\right),
\end{align}
then \eqref{SumEst} follows from the above two inequalities and \eqref{HausComp1}. We first prove \eqref{Balls_E}. The proof of \eqref{Balls_L} is similar and we shall comment on the necessary changes after we are done with \eqref{Balls_E}. 

Let 
\[A \coloneqq \{x \in F \cap 2\B : \text{dist}(x,E) > t/150\}\]
and let $\mathscr{B}_A$ be a cover of $A$ such that each ball $B \in \mathscr{B}_A$ is centred on $A$, has $r_B \leq r_\B$ and 
\begin{align}\label{e:HausA}
\mathscr{H}^d_\infty(A) \sim \sum_{B \in \mathscr{B}_A}r_B^d.
\end{align}
Let $B_i \in \mathscr{B}_E$ satisfy $E_t \cap B_i \not= \emptyset$ and $y_i \in E_t \cap B_i.$ Recall that since $B_i \in \mathscr{B}_E$ we have $\dist(x_i,L) \leq \dist(x_i,E)$ and $B_i = B(x_i,12\dist(x_i,E) +t/10)$. It follows that 
\begin{align}\label{e: delta > t}
t &< \dist(y_i,L) \leq |y_i - x_i| + \dist(x_i,E) \\
& \leq 12\dist(x_i,E) + t/10 + \dist(x_i,E) \\
&\leq 13\dist(x_i,E) + t/2. 
\end{align}
Rearranging, we have
\begin{align}\label{e:x_it}
\dist(x_i,E) >t/26.
\end{align}
This implies that
\begin{align}\label{F in A}
F \cap \tfrac{1}{24}B_i \subseteq A, 
\end{align}
since for any $y \in F \cap \tfrac{1}{24}B_i$ we have
\begin{align}
\dist(y,E) \geq \dist(x_i,E) - \frac{1}{24}r_{B_i} = \frac{1}{2}\dist(x_i,E)  - \frac{1}{240}t \stackrel{\eqref{e:x_it}}{\geq} \frac{t}{150}.
\end{align} 
By \eqref{F in A}, since $\mathscr{B}_A$ covers $A$, there is $B \in \mathscr{B}_A$ such that $\tfrac{1}{24}B_i \cap B \not=\emptyset$. We partition $\mathscr{B}_E$ further by setting
\begin{align*}
\mathscr{C}_1 &= \{B_i  : \text{there exists} \ B \in \mathscr{B}_A \ \text{such that} \ \tfrac{1}{24}B_i \cap B \not=\emptyset \ \text{and} \ r_B \geq r_{B_i}/24 \},\\
\mathscr{C}_2 &= \mathscr{B}_E\setminus\mathscr{C}_1. 
\end{align*}
\begin{figure}[t]
  \centering
  \includegraphics[scale = 0.8]{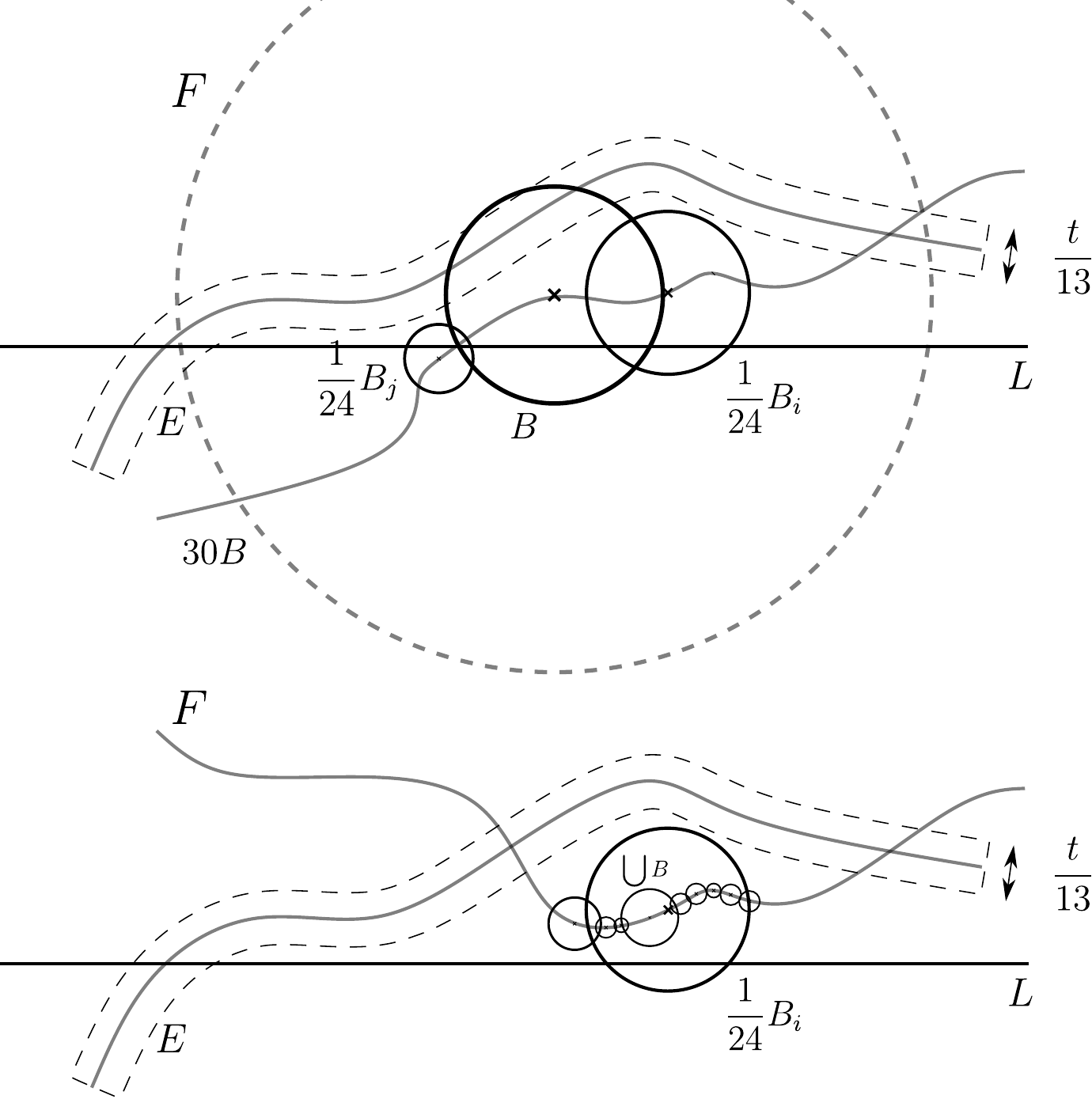}
\caption{Examples of balls in $\mathscr{C}_1$ (left) and $\mathscr{C}_2$ (right).}\label{f:Error}
\end{figure}
See Figure \ref{f:Error}. We first control the sum over balls in $\mathscr{C}_1.$ Let $B \in \mathscr{B}_A$ and assume there exists $B_i \in \mathscr{C}_1$ such that $\tfrac{1}{24}B_i \cap B\not=\emptyset$ and $r_B \geq r_{B_i}/24.$ For $y \in B_i,$ we have
\begin{align*}
|y - x_B| \leq |y-x_i| + |x_i - x_B| \leq r_{B_i} + r_B + r_{B_i/24} \leq 26r_B,
\end{align*}
from which we conclude $B_i \subseteq 26B.$ So, if 
\[
\mathscr{C}^B_1 = \{B_i \in \mathscr{C}_1 : \tfrac{1}{24}B_i \cap B \not=\emptyset \ \text{and} \ r_B \geq r_{B_i}/24\},
\] 
the balls $\{\tfrac{1}{6}B_i\}_{B_i \in \mathscr{C}_1^B}$ are pairwise disjoint, contained in $26B$ and satisfy 
\[\text{dist}(x_i,L) \leq r_{\frac{1}{6}B_i}/2\] 
(recall that $B_i = B(x_i,12\delta(x_i))$). By Lemma \ref{ENV} and because $\mathscr{B}$ is a good cover, this gives
\begin{align*}
\sum_{B_i \in \mathscr{C}_1^B}r_{B_i}^d \lesssim \sum_{B_i \in \mathscr{C}_1^B}r_{\frac{1}{6}B_i}^d \overset{\eqref{e:ENV}}{\lesssim} r_B^d.
\end{align*}
Thus, 
\begin{align}\label{HausComp2}
\sum_{B_i \in \mathscr{C}_1}r_{B_i}^d \leq \sum_{B \in \mathscr{B}_A} \sum_{B_i \in \mathscr{C}_1^B}r_{B_i}^d \lesssim \sum_{B \in \mathscr{B}_A} r_B^d \overset{\eqref{e:HausA}}{\lesssim} \mathscr{H}_\infty^d(A).
\end{align}
Now, the sum over balls in $\mathscr{C}_2.$ If $B_i \in \mathscr{C}_2$ and $B \in \mathscr{B}_A$ is such that $\tfrac{1}{24}B_i \cap B \not=\emptyset,$ then $r_B < r_{B_i}/24.$ Furthermore $B \cap \tfrac{1}{24}B_j = \emptyset$ for all $B_j \in \mathscr{C}_2, \ j \not= i,$ since otherwise 
\begin{align*}
|x_i - x_j| &\leq |x_i - x_B| + |x_B - x_j| \leq r_{B_i}/24 + 2r_B + r_{B_j}/24 \\
& \leq\left(\frac{1}{24} + \frac{1}{12} + \frac{1}{24} \right) \max\{r_{B_i},r_{B_j}\} \leq \frac{1}{6}\max\{r_{B_i},r_{B_j} \} \\
&\leq 4\max\{\delta(x_i),\delta(x_j)\},
\end{align*}
contradicting \eqref{e:sep}. Thus, 
\begin{equation}\label{e:C_2} 
\#\{B_i \in \mathscr{C}_2: B \cap \frac{1}{24}B_i \not= \emptyset\} \leq 1.
\end{equation}
Since $\mathscr{B}_A$ forms a cover for $A$ and $F \cap \tfrac{1}{24}B_i \subseteq A$ by \eqref{F in A}, it follows that $\{B \in \mathscr{B}_A : B \cap \tfrac{1}{24}B_i \not=\emptyset\}$ forms a cover for $F \cap \tfrac{1}{24}B_i.$ Using then that $F$ is $(c,d)$-lower content regular, we have
\begin{align}\label{HausComp3}
\begin{split}
\sum_{B_i \in \mathscr{C}_2}r_{B_i}^d &\lesssim \sum_{B_i \in \mathscr{C}_2} \mathscr{H}_\infty^d(F \cap \tfrac{1}{24}B_i) \leq \sum_{B_i \in \mathscr{C}_2} \sum_{\substack{B \in \mathscr{B}_A\\ B \cap \frac{1}{24}B_i \not=\emptyset}} r_B^d \\
& = \sum_{B \in \mathscr{B}_A} \sum_{\substack{B_i \in \mathscr{C}_2 \\ B \cap \frac{1}{24}B_i \not=\emptyset}} r_B^d \overset{\eqref{e:C_2}}{\leq} \sum_{B \in \mathscr{B}_A} r_B^d \overset{\eqref{e:HausA}}{\lesssim} \mathscr{H}_\infty^d(A).
\end{split}
\end{align}
Combining \eqref{HausComp2} and \eqref{HausComp3} completes the proof of \eqref{Balls_E}. The proof of \eqref{Balls_L} follows exactly the same reasoning: for each $B_i \in \mathscr{B}_L$ we have
\[ \dist(x_i,L) > t/26, \]
the proof of which is the same as \eqref{e:x_it}. Analogously to \eqref{F in A}, this implies
\[ F \cap \tfrac{1}{24}B_i  \subseteq A',\]
where 
\[ A' = \{x \in F \cap 2\B : \dist(x,L) > t/150\}. \]
The rest of the proof is identical. This completes the proof of \eqref{SumEst} which in turn completes the proof of the lemma. 

\end{proof}

As a corollary of the above proof of Lemma \ref{SubsetError} and Remark \ref{r:ineq}, we have the following:

\begin{cor}\label{CheckComp}
Suppose $c \geq c_1$ and $E \subseteq \R^n$ is $(c,d)$-lower content regular. For any ball $B$ and $d$-plane $L,$ we have
\begin{align}
\check\beta^{d,p}_E(B,L) \leq \beta_E^{d,p}(B,L) \lesssim \check\beta^{d,p}_E(2B,L).
\end{align}
\end{cor}

\begin{lem}
Assume $E \subseteq \R^n$ and there is $B$ centred on $E$ so that for all $B^\prime \subseteq B$ centred on $E$ we have $\mathscr{H}_\infty^d(E \cap B^\prime) \geq cr_{B^\prime}^d$. Then
\begin{align}\label{CompL_infty}
\beta_{E,\infty}^d\left( \frac{1}{2}B\right) \lesssim \beta_E^{d,1}(B)^\frac{1}{d+1}.
\end{align}
\end{lem}

\begin{proof}
Azzam and Schul prove the same inequality for $\check\beta_E^{d,p}$ (see \cite[Lemma 2.12]{azzam2018analyst}). Then, by Corollary \ref{CheckComp},
\[
\beta_{E,\infty}^d\left( \frac{1}{2}B\right) \lesssim \check\beta_E^{d,1}(B)^\frac{1}{d+1} \lesssim \beta_E^{d,1}(B)^\frac{1}{d+1}.
\]
\end{proof}

\begin{rem}\label{r:children}
By \eqref{CompL_infty} and Lemma \ref{l:children}, if $\e >0$ is small enough (depending on $M$) and $\beta_E^{d,p}(MB_Q) \leq \e$ for some $Q \in \mathscr{D},$ then $Q$ has at most $K$ children where $K$ depends only on $M$ and $d$ and not the ambient dimension $n.$
\end{rem}

The following is analogous to Lemma 2.21 in \cite{azzam2018analyst}. It says the $\beta$-number of a lower regular set can be controlled by the $\beta$-number of a nearby set, with an error depending on the average distance between the two. The proof is very similar to the proof of Lemma \ref{SubsetError}.  

\begin{lem}\label{betaest}
Let $1 \leq p < \infty.$ Suppose $E,F \subseteq \R^n$, $B^1$ is a ball centred on $E$ and $B^2$ is a ball of same radius but centred on $F$ such that $B^1 \subseteq 2B^2.$ Suppose for all balls $B \subseteq 2B^1$ centred on $E$ we have $\mathscr{H}_\infty^d(B \cap E) \geq c{r_B}^d$  for some $c > 0.$ Then
\begin{align}\label{betaeq}
\check\beta_{E}^{d,p}(B^1,P) &\lesssim_{c,p,d} \beta_{F}^{d,p}(2B^2,P) \\
&\hspace{4em}+ \left(\frac{1}{r^d_{B^1}}\int_{E \cap 2B^1}\left( \frac{\emph{dist}(y,F)}{r_{B^1}}\right)^p \, d\mathscr{H}_\infty^d(y) \right)^\frac{1}{p}.
\end{align} 
\end{lem}

\begin{proof}
By scaling and translating, we can assume that $B^1 = \B^.$ For $t>0$, set
\begin{align*}
E_t = \{x \in E \cap \B : \text{dist}(x,L) > t \}.
\end{align*}
To prove \eqref{betaeq}, it suffices to show 
\begin{align}\label{e:leq}
\mathscr{H}_\infty^d(E_t) &\lesssim \mathscr{H}_{2\B,\infty}^{d,F}(\{x \in F \cap 2\B : \text{dist}(x,L) > tr_B/2 \}) \\ 
&\hspace{4em} + \mathscr{H}_\infty^d(\{x \in E \cap 2\B : \text{dist}(x,F) > tr_B/32 \}),
\end{align}
since
\begin{align*}
\beta_E^{d,p}(\B,L)^p &= \int_0^1 \mathscr{H}_\infty^d(E_t) t^{p-1} \, dt  \\
&\stackrel{\eqref{e:leq}}{\lesssim} \int_0^1 \mathscr{H}_{2\B,\infty}^{d,F} (\{x \in F \cap 2\B : \text{dist}(x,L) > t/2 \}) t^{p-1} \, dt \\
&\hspace{4em} +  \int_0^1 \mathscr{H}_\infty^d (\{x \in E \cap 2\B : \text{dist}(x,F) > t/32 \}) t^{p-1} \, dt \\
&\lesssim \int_0^1 \mathscr{H}_{2\B,\infty}^{d,F} (\{x \in F \cap 2\B : \text{dist}(x,L) > 2t \}) t^{p-1} \, dt \\
&\hspace{4em} +  \int_0^1 \mathscr{H}_\infty^d (\{x \in E \cap 2\B : \text{dist}(x,F) > 2t\}) t^{p-1} \, dt \\
&\sim\beta_F^{d,p}(2\B,L)^p +  \int_{E \cap 2\B} \dist(x,F)^p \, d\mathscr{H}_\infty^d(x),
\end{align*}
Let us prove \eqref{e:leq}. We first need to construct a suitable cover for $E_t.$ For $x \in E_t$, let 
$$\delta(x) = \max\{ \text{dist}(x,L), 16\text{dist}(x,F) \}$$
and set $X_t$ to be a maximally separated net in $E_t$ such that, for $x,y \in X_t,$ we have 
\begin{align}\label{Separation}
|x-y| \geq 4 \max\{\delta(x),\delta(y)\}.
\end{align}
Enumerate $X_t = \{x_i\}_{i \in I}.$ For $x_i \in X_t$ denote $B_i = B(x_i,\delta(x_i)).$ Notice, these balls are non-degenerate since $x_i \in X_t \subseteq E_t.$ By maximality we know $\{4B_i\}$ covers $E_t$ so,
\begin{align}\label{E_t}
\mathscr{H}_\infty^d(E_t) \lesssim \sum_{i \in I} (r_{4B_i})^d. 
\end{align} 
We partition $I = I_1 \cup I_2$, where 
\begin{align*}
I_1 = \{ i \in I : \delta(x_i) = \text{dist}(x_i,L) \} 
\end{align*}
and
\begin{align*}
I_2 = \{ i \in I : \delta(x_i) = 16\text{dist}(x_i,F) \}.
\end{align*}
If $\text{dist}(x_i, L) = 16\text{dist}(x_i,F)$, we put $i$ in $I_1$ or $I_2$ arbitrarily. By \eqref{E_t}, it follows that 
\[\mathscr{H}^d_\infty(E_t) \lesssim \sum_{i \in I_1} (r_{4B_i})^d + \sum_{i \in I_2} (r_{4B_i})^d.\]
We will show that
\begin{align}\label{I_1}
\sum_{i \in I_1} (r_{4B_i})^d \lesssim \mathscr{H}^{d,E}_{2\B,\infty}( \{x \in F \cap 2\B : \dist(x,L) > t/2\})
\end{align}
and
\begin{align}\label{I_2}
\sum_{i \in I_2} (r_{4B_i})^d \lesssim \mathscr{H}^d_\infty ( \{x \in E \cap 2\B : \dist(x,F) > t/32\})
\end{align}
from which \eqref{e:leq} follows. The rest of the proof is dedicated to proving \eqref{I_1} and \eqref{I_2}. Let us begin with \eqref{I_1}. Let 
\[F_t = \{x \in F \cap 2\B : \text{dist}(x,L) > t/2\}\]
and $\mathscr{B}$ be a good cover for $F \cap 2\B$ such that 
\begin{align}\label{e:good}
\mathscr{H}_{2\B,\infty}^{d,F}(F_t) \sim \sum_{\substack{B \in \mathscr{B}\\B \cap F_t \not=\emptyset}} r_B^d.
\end{align}
If $i \in I_1$ then 
\begin{align}\label{e:use}
F \cap \tfrac{1}{2}B_i \not= \emptyset \quad \text{and} \quad F \cap \tfrac{1}{2}B_i \subseteq F_t
\end{align}
since $\text{dist}(x_i,L) >t$ (by virtue of that fact that $x_i \in E_t$) and 
\[\tfrac{1}{2}B_i = B(x_i,\text{dist}(x_i,L)/2) \supseteq B(x_i , 8\text{dist}(x_i,F)).\]
Since $\mathscr{B}$ forms a cover of $F_t$ there exists at least one $B \in \mathscr{B}$ such that $B \cap \tfrac{1}{2}B_i \not=\emptyset.$ We further partition $I_1.$ Let
\begin{align*}
I_{1,1} &= \{i \in I_1 : \ \text{there exists} \ B \in \mathscr{B} \ \text{such that} \ \tfrac{1}{2}B_i \cap B \not= \emptyset \ \text{and} \ r_B \geq r_{B_i} \}, \\
I_{1,2} &= I_1 \setminus I_{1,1}. 
\end{align*}
We first control the sum over $I_{1,1}.$ Let $B \in \mathscr{B}$ and assume there is $B_i$ satisfying $\tfrac{1}{2}B_i \cap B \not=\emptyset$ and $r_B \geq r_{B_i}$ (which by definition implies $i \in I_{1,1}$), then $2B_i \subseteq 4B.$ By \eqref{Separation} we know the $\{2B_i\}$ are disjoint and satisfy $\text{dist}(x_i,L) \leq r_{2B_i}/2.$ By Lemma \ref{ENV}, we have
\begin{align*}
\sum_{\substack{i \in I_{1,1} \\ \frac{1}{2}B_i \cap B \not=\emptyset \\ r_{B_i} \leq r_B}} r_{4B_i}^d \lesssim \sum_{\substack{i \in I_{1,1} \\ 2B_i \subseteq 4B}} r_{2B_i}^d \overset{\eqref{e:ENV}}{\lesssim} r_B^d.
\end{align*}
Thus,
\begin{align}\label{I_L^1}
\sum_{i \in I_{1,1}} r_{4B_i}^d \leq \sum_{\substack{B \in \mathscr{B}\\B \cap F_t \not=\emptyset}} \sum_{\substack{i \in I_{1,1} \\ \frac{1}{2}B_i \cap B \not=\emptyset \\ r_{B_i} \leq r_B}} r_{4B_i}^d \lesssim \sum_{\substack{B \in \mathscr{B}\\B \cap F_t \not=\emptyset}}r_B^d. 
\end{align}
We now turn our attention to $I_{1,2}.$ For $i \in I_{1,2}$, let $x_i^\prime$ be the point in $F$ closest to $x_i$ and set $B_i^\prime = B(x_i^\prime, \text{dist}(x_i,L)/4).$ Note that $B_i^\prime \subseteq \tfrac{1}{2}B_i,$ since for $y \in B_i^\prime$ we have
\begin{align*}
|y - x_i| \leq \frac{1}{4}\text{dist}(x_i,L) + |x-x_i| \overset{(i \in I_1)}{\leq} \left(\frac{1}{4} + \frac{1}{16} \right)\text{dist}(x_i,L) \leq \frac{1}{2}\text{dist}(x_i,L).
\end{align*}
Since $F \cap \tfrac{1}{2}B_i \subseteq F_t$ by \eqref{e:use}, and $\mathscr{B}$ forms a cover for $F_t,$ the balls 
\[\{B \in \mathscr{B} : B \cap F \cap \frac{1}{2}B_i\not=\emptyset\}\]
form a cover for $F \cap \tfrac{1}{2}B_i.$ Furthermore, if $B \cap \tfrac{1}{2}B_i \not=\emptyset$ then $B \cap \tfrac{1}{2}B_j = \emptyset$ for all $i\not=j$, that is 
\begin{align}\label{e:1,2}
\# \{i \in I_{1,2} : B_i \cap B \not=\emptyset \} \leq 1,
\end{align}
since otherwise $|x_i - x_j| < 4\max\{\delta(x_i),\delta(x_j)\},$ contradicting \eqref{Separation}. By Lemma \ref{HausEasy} (1), we know
\[\mathscr{H}_{2\B,\infty}^{d,F}(F \cap B_i^\prime) \gtrsim r_{B_i^\prime}^d \gtrsim r_{B_i}^d,\]
and since $\mathscr{B}$ is a good cover for $F \cap 2\B,$ we have
\begin{equation}
\begin{aligned}
\sum_{i \in I_{1,2}} (r_{4B_i})^d &\lesssim \sum_{i \in I_{1,2}} \mathscr{H}_{2\B,\infty}^{d,F}(F \cap B_i^\prime) \leq  \sum_{i \in I_{1,2}} \mathscr{H}_{2\B,\infty}^{d,F}(F \cap \tfrac{1}{2}B_i) \\
&\leq \sum_{i \in I_{1,2}} \sum_{\substack{B \in\mathscr{B} \\ B \cap \frac{1}{2}B_i \cap F \not=\emptyset}} r_B^d  = \sum_{\substack{B \in\mathscr{B}\\ B \cap F_t \not=\emptyset}}  \sum_{\substack{i \in I_{1,2} \\ B \cap \frac{1}{2}B_i \cap F \not=\emptyset}} r_B^d \overset{\eqref{e:1,2}}{\leq}  \sum_{\substack{B \in\mathscr{B}\\ B \cap F_t \not=\emptyset}} r_B^d. 
\end{aligned}\label{I_L^2}
\end{equation}
Combining \eqref{I_L^1} and \eqref{I_L^2}, we conclude
\begin{align}\label{F_t}
\sum_{i \in I_1} r_{4B_i}^d \lesssim \sum_{\substack{B \in \mathscr{B} \\ B \cap F_t \not=\emptyset}} r_B^d \overset{\eqref{e:good}}{\lesssim} \mathscr{H}_{2\B,\infty}^{d,F}(F_t)
\end{align}
which is \eqref{I_1}.

We turn our attention to proving \eqref{I_2}, the proof of which follows much the same as that for \eqref{I_1}. Let 
$$E_t^\prime =\{x \in E \cap 2\B : \text{dist}(x,F) > t/32 \}$$
and $\mathscr{B}^\prime$ be a collection of balls covering $E_t'$ such that each $B \in \mathscr{B}'$ is centred on $E_t'$, has $r_B \leq r_\B$ and 
\begin{align}\label{e:good2}
\mathscr{H}_\infty^d(E_t^\prime) \sim \sum_{B \in \mathscr{B}^\prime} r_B^d.
\end{align}
As before, we partition $I_2$. If $i \in I_2,$ since $x_i \in E_t,$ we have $\text{dist}(x_i,L) >t$ and
\begin{align}
\text{dist}(x_i,F) \geq \text{dist}(x_i,L)/16 \geq t/16,
\end{align}
hence 
\begin{align}\label{use1}
E \cap \tfrac{1}{32}B_i = E \cap B(x_i,\text{dist}(x_i,F)/2) \subseteq E_t^\prime.
\end{align}
Thus for each $B_i,$ since $\mathscr{B}'$ forms a cover for $E_t',$ there exists $B \in \mathscr{B}^\prime$ such that $B \cap \tfrac{1}{32}B_i \not=\emptyset.$ We partition $I_2$ by letting
\begin{align*}
I_{2,1} &= \{i \in I_2 : \ \text{there exists} \ B \in \mathscr{B}^\prime \ \text{such that} \ \tfrac{1}{32}B_i \cap B \not= \emptyset \ \text{and} \ r_B \geq r_{B_i} \}, \\
I_{2,2} &= I_2\setminus I_{2,1}. 
\end{align*}
If $B \in \mathscr{B}^\prime$ and $B \cap \tfrac{1}{32}B_i \not=\emptyset$ with $r_B \geq r_{B_i}$ then $2B_i \subseteq 4B.$ Furthermore, by \eqref{Separation}, we know the $\{2B_i\}$ are disjoint and satisfy $\text{dist}(x_i,L) < \text{dist}(x_i,F)/16 = r_{2B_i}/2,$ so by Lemma \ref{ENV}, we have
\begin{align*}
\sum_{\substack{i \in I_{2,1} \\ \frac{1}{32}B_i \cap B \not=\emptyset \\ r_{B_i} \leq r_B}} r_{4B_i}^d \lesssim \sum_{\substack{i \in I_{2,1} \\ 2B_i \subseteq 3B}} r_{2B_i}^d \overset{\eqref{e:ENV}}{\lesssim} r_B^d.
\end{align*}
Thus,
\begin{align}\label{I_F^1}
\sum_{i \in I_{2,1}} r_{4B_i}^d \leq \sum_{B \in \mathscr{B}^\prime} \sum_{\substack{i \in I_{2,1} \\ \frac{1}{32}B_i \cap B \not=\emptyset \\ r_{B_i} \leq r_B}} r_{4B_i}^d \leq \sum_{B \in \mathscr{B}^\prime}r_B^d. 
\end{align}
We now deal with $I_{2,2}.$ Since by \eqref{use1}, $E \cap \tfrac{1}{32}B_i \subseteq E_t^\prime$, the balls 
\[\{B \in \mathscr{B}^\prime : B \cap \frac{1}{32}B_i \cap E \not=\emptyset\}\]
form a cover for $E \cap \tfrac{1}{32}B_i.$ As before, if $B \cap \tfrac{1}{32}B_i \not=\emptyset$ then $B \cap \tfrac{1}{32}B_j = \emptyset$ for all $i\not=j$ by \eqref{Separation}. By lower regularity of $E$, we know $\mathscr{H}_\infty^d(E \cap \tfrac{1}{32}B_i) \gtrsim r_{B_i}^d,$ from which we conclude
\begin{align}\label{I_F^2}
\begin{split}
\sum_{i \in I_{2,2}} r_{4B_i}^d &\lesssim \sum_{i \in I_{2,2}} \mathscr{H}_\infty^d(E \cap \tfrac{1}{32} B_i) \leq   \sum_{i \in I_{2,2}} \sum_{\substack{B \in\mathscr{B}^\prime \\ B \cap \frac{1}{32}B_i \cap E \not=\emptyset}} r_B^d \\
&= \sum_{B \in\mathscr{B}^\prime}  \sum_{\substack{i \in I_{2,2} \\ B \cap \frac{1}{32}B_i \cap E \not=\emptyset}} r_B^d \lesssim \sum_{B \in \mathscr{B}^\prime} r_B^d. 
\end{split}
\end{align}
The proof of \eqref{I_2} (and hence the proof of the lemma) is completed since
\[ \sum_{i\in I_2} r_{4B_i}^d =  \sum_{i\in I_{2,1}} r_{4B_i}^d + \sum_{i\in I_{2,2}} r_{4B_i}^d \stackrel{ \substack{\eqref{I_F^1} \\ \eqref{I_F^2}}}{\lesssim} \sum_{B \in \mathscr{B}'} r_B^d \stackrel{\eqref{e:good2}}{\lesssim} \mathscr{H}^d_\infty(E_t'). \]

\end{proof}

In Section \ref{s:Thm1} and Section \ref{s:Thm4}, we want to apply the construction of David and Toro (Theorem \ref{DT}). To do this, we need to control the angles between pairs of planes, by their corresponding $\beta$-numbers. The following series of lemmas, culminating in Lemma \ref{AngControl}, will allow us to do so. We first introduce some more notation.

For two planes $P,P^\prime$ containing the origin, we define
\begin{align}
\angle(P,P^\prime) = d_{B(0,1)}(P,P^\prime).
\end{align}
If $P,P^\prime$ are general affine planes with $x \in P$ and $y \in P^\prime$, we define 
\begin{align}
\angle(P,P^\prime) = \angle(P-x,P^\prime - y). 
\end{align}
For planes $P_1,P_2$ and $P_3,$ it is not difficult to show that 
\begin{align}\label{Triangle}
\angle(P_1,P_3) \leq \angle(P_1,P_2) + \angle(P_2,P_3).
\end{align}

\begin{lem}[{\cite[Lemma 6.4]{azzam2015characterization}}]\label{AzzamTolsa}
Suppose $P_1$ and $P_2$ are $d$-planes in $\R^n$ and $X = \{x_0,\dots, x_d\}$ are points so that 
\begin{enumerate}
\item $\eta \in (0,1)$, where 
\[\eta = \eta(X) = \min \{\emph{dist}(x_i, \emph{span}(X\setminus \{x_i\}) \}/\emph{diam}(X) \]
\item $\emph{dist}(x_i,P_j) < \e \emph{diam}(X)$ for $i=0,\dots,d$ and $j=1,2,$ where $\e < \eta d^{-1}/2.$ 
\end{enumerate}
Then
\[\emph{dist}(y,P_1) \leq \e \left(\frac{2d}{\eta}\emph{dist}(y,X) + \emph{diam}(X) \right) \ \text{for all} \ y \in P_2. \]
\end{lem}

In order to control angles between $d$-planes, we need to know that $E$ is sufficiently spread out in at least $d$ directions. This is quantified below. 

\begin{defn}\label{Sep}
Let $ 0 <\alpha <1.$ We say a ball $B$ has $(d+1,\alpha)$-separated points if there exist points $X = \{x_0,\dots,x_d\}$ in $E \cap B$ such that, for each $i=1,\dots,d,$ we have
\begin{align}\label{e:sepdef}
\text{dist}(x_{i+1},\text{span}\{x_0,\dots,x_i\}) \geq \alpha r_{B}.
\end{align}
\end{defn}

\begin{lem}\label{AngPre}
Suppose $E \subseteq \R^n$ and there is $B^\prime$ and $B$ both centred on $E$ with $B^\prime \subseteq B.$ Suppose further that there exists $0 < \alpha <1$ such that $B^\prime$ has $(d+1,\alpha)$-separated points. Let $P$ and $P^\prime$ be two $d$-planes. Then
\begin{align*}
d_{B^\prime}(P,P^\prime) \lesssim  \frac{1}{\alpha^{2d+2}}\left[\left( \frac{r_B}{r_{B^\prime}}\right)^{d+1} \beta_E^{d,1}(2B,P) + \beta_E^{d,1}(2B^\prime,P^\prime) \right].
\end{align*}
\end{lem}
\begin{proof}
Since $B^\prime$ has $(d+1,\alpha)$-separated points, we can find $X = \{x_0,\dots, x_d\}$ satisfying \eqref{e:sepdef}. This implies that $\alpha < \eta(X) \leq 1.$ Let $B_i = B(x_i, \alpha^2 r_{B^\prime}/100)$ and for $t>0$ let 
\[E_{t,i} = \{ x \in E \cap B_i : \text{dist}(x,P) > tr_{B^\prime} \ \text{or} \ \text{dist}(x,P^\prime) >tr_{B^\prime}\}. \]
Let $T>0$ and suppose $E_{t,i} = E \cap B_i$ for all $t \leq T.$ We shall bound $T$. By Lemma \ref{HausEasy} (1), 
\[\mathscr{H}^{d,E}_{B_i,\infty}(E \cap B_i) \geq c_1r_{B_i}^d = \frac{c_1 \alpha^{2d}}{100^d}r_{B^\prime}^d. \]
Using this, along with Lemma \ref{HausEasy} (2),(3),(4), we get
\begin{align*}
T &\leq \mathscr{H}^{d,E}_{B_i,\infty}(E \cap B_i)^{-1} \int_0^T \mathscr{H}^{d,E}_{B_i,\infty}(E_{t,i}) \, dt \lesssim  \frac{1}{\alpha^{2d} r_{B^\prime}^d}\int_0^T \mathscr{H}^{d,E}_{B_i,\infty}(E_{t,i}) \, dt \\
&\stackrel{(4)}{\lesssim} \frac{1}{\alpha^{2d} r_{B^\prime}^d}\int_0^T \mathscr{H}^{d,E}_{B_i,\infty}\{x \in E\cap B_i : \text{dist}(x,P) >tr_{B^\prime}\} \, dt \\
& \quad\quad +\frac{1}{\alpha^{2d} r_{B^\prime}^d}\int_0^T \mathscr{H}^{d,E}_{B_i,\infty}\{x \in E\cap B_i : \text{dist}(x,P^\prime) >tr_{B^\prime}\} \, dt \\
&\stackrel{(2),(3)}{\lesssim} \frac{1}{\alpha^{2d} r_{B^\prime}^d}\int_0^T \mathscr{H}^{d,E}_{2B^\prime,\infty}\{x \in E\cap 2B^\prime : \text{dist}(x,P) >tr_{2B^\prime}\} \, dt \\
& \quad\quad +\frac{r_B^{d+1}}{r_{B^\prime}^{d+1}}\frac{1}{ \alpha^{2d} r_{B^\prime}^d}\int_0^T \mathscr{H}^{d,E}_{2B,\infty}\{x \in E\cap 2B : \text{dist}(x,P^\prime) >tr_{2B}\} \, dt \\
& \lesssim \frac{1}{\alpha^{2d}}\left[\left( \frac{r_B}{r_{B^\prime}}\right)^{d+1} \beta_E^{d,1}(2B,P) + \beta_E^{d,1}(2B^\prime,P^\prime)\right] \eqqcolon \lambda \alpha^2 .
\end{align*}
Note, we define $\lambda$ like this for convenience in the forthcoming estimates. Thus, there is a constant $C$ such that $T \leq C \lambda \alpha^2 .$ This implies for each $i = 0,1,\dots,d$, there exists some $y_i \in (E \cap B_i)\setminus E_{2 \lambda \alpha^2, i}$. Let $Y = \{y_0,\dots,y_d\}.$ Since $|x_i - x_j| \geq \alpha r_{B'}$ for all $i \not =j,$ and $y_i \in B_i,$ it follows that 
\begin{align}\label{e:dY>a}
\text{diam}(Y) \geq \alpha r_{B'}/2.
\end{align}
Thus,
\begin{align}
\dist(y_i,P_j) \leq 2C\lambda  \alpha^2 r_{B^\prime} = \frac{2C\lambda  \alpha^2 r_{B^\prime}}{\diam (Y)} \diam (Y) \stackrel{\eqref{e:dY>a}}{\leq} 4C \lambda\alpha  \diam (Y).
\end{align}
Because $d_{B'}(P,P') \leq 1,$ if $\lambda \geq \tfrac{1}{16Cd}$ then the lemma follows. Assume instead that $\lambda < \tfrac{1}{16Cd}.$ By \eqref{e:dY>a} we can show that 
\begin{align}\label{e:eta}
\alpha/2 \leq \eta(Y) \leq 1,
\end{align}
which gives
\[ 4C \lambda \alpha \leq \frac{\alpha d^{-1}}{4} \stackrel{\eqref{e:eta}}{\leq} \eta(Y) d^{-1}/2,\]
so, taking $\e = 4C\lambda \alpha$ in Lemma \ref{AzzamTolsa}, we get 
\[ d_{B'}(P,P^\prime) \leq \epsilon \left(\frac{2d}{\eta(Y)} + 1 \right) \stackrel{\eqref{e:dY>a}}{\leq} 4C\lambda \alpha \left( \frac{4d}{\alpha} +1\right) \leq 20Cd\lambda, \]
which proves the lemma.
\end{proof}


\begin{rem}
The following lemma is essentially Lemma 2.18 from \cite{azzam2018analyst} and the proof is the same. The main difference is that since $E$ is not necessarily lower regular, we need to assume that $E$ has $(d+1,\alpha)$-separated points in each cube. The final constant then also ends up depending on $\alpha.$ When we refer to this lemma for a lower regular set, we will be using Lemma 2.18 from \cite{azzam2018analyst} i.e. we can forget about the separation condition and the $\alpha$ constant in the final inequality. 
\end{rem}

\begin{lem}\label{AngControl}
Let $M >1$, $\alpha >0$ and $E$ a Borel set. Let $\mathscr{D}$ be the cubes for $E$ from Lemma \ref{cubes} and $Q_0 \in \mathscr{D}.$ Let $P_Q$ satisfy $\beta_E^{d,1}(MB_Q) = \beta_E^{d,1}(MB_Q,P_Q)$. Let $Q,R \in \mathscr{D}$, $Q,R \subseteq Q_0$ and suppose for all cubes $T \subseteq Q_0$ such that $T$ contains either $Q$ or $R$ that $\beta_E^{d,1}(MB_T) < \e$ and $T$ has $(\alpha,d+1)$-separated points. Then for $\Lambda >0$, if $\emph{dist}(Q,R) \leq \Lambda \max\{\ell(Q),\ell(R) \} \leq \Lambda^2\min\{\ell(Q),\ell(R)\},$ then
\begin{align*}
\angle(P_Q,P_R) \lesssim _{M,\Lambda}\frac{\e}{\alpha^{2d+2}}. 
\end{align*}
\end{lem}


\section{Proof of Theorem \ref{Thm1}}\label{s:Thm1}

Let $X_k^E$ be a sequence of maximally $\rho^k$-separated nets in $E$. For each $k$, let $X_k^F$ be the completion of $X_k^E$ to a maximally $\rho^k$-separated net in $F.$ Let $\mathscr{D}^E$ and $\mathscr{D}^F$ be the cubes from Theorem \ref{cubes} with respect to $X_k^E$ and $X_k^F.$ In this way, for each $Q \in \mathscr{D}^E$, there exists $Q' \in \mathscr{D}^F$ such that 
\begin{align}\label{e:cubeEF}
x_Q = x_{Q'}, \ \ell(Q) = \ell(Q').
\end{align}

Let $Q_0^E \in \mathscr{D}^E$ and let $Q_0^F \in \mathscr{D}^F$ be the cube with the same centre and side length as $Q_0^E.$ To simplify notation we will write $ \mathscr{D}= \mathscr{D}^F$ and $Q_0 = Q_0^F.$ We first reduce to proof of \eqref{e:Thm1} to the proof of \eqref{Suffices1} below. Recall the definition 
\[\mathscr{D}(Q) = \{R \in \mathscr{D} : R \subseteq Q\}.\]
\begin{lem}\label{l:Suffices1}
If
\begin{align}\label{Suffices1}
\sum_{Q \in \mathscr{D}(Q_0)} \beta_E^{d,p}(C_0B_Q)^2 \ell(Q)^d \lesssim \mathscr{H}^d(Q_0) + \sum_{Q \in \mathscr{D}(Q_0)} \check\beta_F^{d,1}(AB_Q)^2 \ell(Q)^d
\end{align}
for some $A \geq C_0,$ then \eqref{e:Thm1} holds.
\end{lem}

\begin{proof}
Assume \eqref{Suffices1} holds. By \eqref{e:cubeEF}, for each cube $Q \in \mathscr{D}^E$ there exists a cube $Q' \in \mathscr{D}$ such that $C_0B_Q \subseteq C_0B_{Q'}.$ Using this with Lemma \ref{l:cont} gives  
\begin{align}
&\diam(Q_0^E)^d + \sum_{Q \in \mathscr{D}^E(Q_0^E)}\beta_E^{d,p}(C_0B_Q)^2 \ell(Q)^d \\
&\hspace{4em} \lesssim\diam(Q_0)^d + \sum_{Q \in \mathscr{D}(Q_0)} \beta_E^{d,p}(C_0B_Q)^2 \ell(Q)^d. 
\end{align}
By \eqref{Suffices1}, Lemma \ref{l:p} and Theorem \ref{dTSP} (because $F$ is lower content regular), we have
\begin{align}
\sum_{Q \in \mathscr{D}(Q_0)} \beta_E^{d,p}(C_0B_Q)^2 \ell(Q)^d
&\stackrel{\eqref{Suffices1}}{\lesssim} \mathscr{H}^d(Q_0) + \sum_{Q \in \mathscr{D}(Q_0)} \check\beta_F^{d,1}(AB_Q)^2 \ell(Q)^d \\
&\stackrel{\eqref{e:p}}{\lesssim} \mathscr{H}^d(Q_0) + \sum_{Q \in \mathscr{D}(Q_0)} \check\beta_F^{d,p}(AB_Q)^2 \ell(Q)^d \\
&\stackrel{\eqref{e:dTSP}}{\lesssim} \diam(Q_0)^d + \sum_{Q \in \mathscr{D}(Q_0)} \check\beta_F^{d,p}(AB_Q)^2 \ell(Q)^d.
\end{align}
Now, let $K = K(C_0,A)$ be the smallest integer such that $(1+A\rho^K) \leq C_0.$ Let $Q \in \mathscr{D}_k$ for some $k \geq K$ and $y \in AB_Q.$ Since $x_Q \in Q^{(K)},$ 
\begin{align}
|y - x_{Q^{(K)}}| \leq A\ell(Q) + \ell(Q^{(K)}) = (A\rho^K + 1)\ell(Q^{(K)}) \leq C_0\ell(Q^{(K)}).
\end{align}
Hence $AB_Q \subseteq C_0B_{Q^{(K)}}.$ Furthermore each cube $Q \in \mathscr{D}$ has at most $C = C(n)$ children (notice the number of descendants is dependent on the the ambient dimension since we do not necessarily know $\beta_F^{d,1}$ is small for an arbitrary cube in $\mathscr{D}$). It follows that each $Q$ has at most $KC$ descendants up to the $K^{th}$ generation. In particular, this is also true for $Q_0.$ By this and Lemma \ref{l:cont}, we have
\begin{align}
\sum_{Q \in \mathscr{D}(Q_0)} \check\beta_F^{d,p}(AB_Q)^2\ell(Q)^d &= \sum_{k=0}^{K-1} \sum_{Q \in \mathscr{D}_k(Q_0)}\check\beta_F^{d,p}(AB_Q)^2\ell(Q)^d  \\
&\hspace{4em} + \sum_{k=K}^\infty \sum_{Q \in \mathscr{D}_k(Q_0)} \check\beta_F^{d,p}(AB_Q)^2\ell(Q)^d\\
&\lesssim_n \diam(Q_0)^d + \sum_{Q \in \mathscr{D}(Q_0)} \check\beta_F^{d,p}(C_0B_Q)^2\ell(Q)^d.
\end{align}
Combing each of the above sets of inequalities gives \eqref{e:Thm4}.

\end{proof}
The rest of this section is devoted to proving \eqref{Suffices1} holds for some $A \geq C_0.$
\begin{defn}\label{d:ST}
A collection of cubes $S \subseteq \mathscr{D}$ is called a \textit{stopping time region} if the following hold.
\begin{enumerate}
\item There is a cube $Q(S) \in S$ such that $Q(S)$ contains all cubes in $S$. 
\item If $Q \in S$ and $Q \subseteq R \subseteq Q(S),$ then $R \in S$.
\item If $Q \in S$, then all siblings of $Q$ are also in $S$. 
\end{enumerate}
We let:
\begin{itemize}
\item $Q(S)$ denote the maximal cube in $S$.
\item $\min(S)$ denote the cubes in $S$ which have a child not contained in $S$.
\item $S(Q)$ denote the unique stopping time regions $S$ such that $Q \in S.$ 
\end{itemize}
\end{defn}

We split $\mathscr{D}(Q_0)$ into a collection of stopping time regions $\mathscr{S}$ where in each stopping time region, $F$ is well-approximated by $E$ and there is good control on a certain Jones type function. Observe that if $Q \in \mathscr{D}$ and $C_0B_Q \cap E = \emptyset$ then $\beta_E^{d,p}(C_0B_Q) = 0,$ and so we will not restart our stopping times on these cubes.


Let $M > 1$ be a large constant (to be fixed later) and $\e >0$ be small (also fixed later). For each $Q \in \mathscr{D}(Q_0)$ such that $E \cap C_0B_Q \not= \emptyset$, we define a stopping time region $S_Q$ as follows. Begin by adding $Q$ to $S_Q$ and inductively, on scales, add cubes $R$ to $S_Q$ if each of the following holds,
\begin{enumerate}
\item $R^{(1)} \in S_Q$,
\item for every sibling $R^\prime$ of $R$, if $x \in R^\prime$ then $\text{dist}(x,E) \leq \e\ell(R^\prime).$ 
\item every sibling $R^\prime$ of $R$ satisfies
\[\sum_{R^\prime \subseteq T \subseteq Q} \check\beta_F^{d,1}(MB_T)^2 < \e^2.\]
\end{enumerate}

\begin{rem}
If $\check\beta_F^{d,1}(MB_Q) \geq \e$ or if there exists $x \in Q$ such that $\dist(x,E) > \e \ell(Q)$ then $S_Q = \{Q\}.$
\end{rem}

\begin{rem}\label{r:int}
For $\e$ small enough, if $R \in S_Q$ for some stopping time region $S_Q,$ then $R \cap E \not=\emptyset$. This follows because $c_0B_R \subseteq R$ and $\dist(x_R,E) \leq \e \ell(R) \leq c_0\ell(R)$ by (2). 
\end{rem}

\begin{rem}
We begin by choosing $\e$ small enough so that each cube $R,$ contained in some stopping time region, has at most $K$ children, where $K$ depends only on $M$ and $d.$ See Remark \ref{r:children} for why this is possible.
\end{rem}

We partition $\{Q \in \mathscr{D}(Q_0) : E \cap C_0B_Q \not=\emptyset\}$ as follows. First, add $S_{Q_0}$ to $\mathscr{S}.$ Then, if $S$ has been added to $\mathscr{S}$ and if $Q \in \text{Child}(R)$ for some $R \in \min(S)$ such that $E \cap C_0B_Q \not= \emptyset,$  also add $S_Q$ to $\mathscr{S}.$ Let $\mathscr{S}$ be the collection of stopping time regions obtained by repeating this process indefinitely. Note that 
\[\sum_{Q \in \mathscr{D}(Q_0)} \beta_E^{d,p}(C_0B_Q)^2 \ell(Q)^d = \sum_{S \in \mathscr{S}} \sum_{Q \in S} \beta_E^{d,p}(C_0B_Q)^2 \ell(Q)^d.\]

For each $S \in \mathscr{S}$ which is not a singleton (i.e. $S \not= \{Q\}$) we plan to find a bi-Lipschitz surface which well approximates $F$ inside $S$. With some additional constraints, the surfaces produced by Theorem \ref{DT} will be bi-Lipschitz. 

\begin{thm}[{\cite[Theorem 2.5]{david2012reifenberg}}]\label{LipDT}
With the same notation and assumptions as Theorem \ref{DT}, assume additionally that there exists $K < \infty$ such that
\begin{align}\label{e:LipDT}
\sum_{k \geq0} \varepsilon_k^\prime(f_k(z))^2 \leq K \ \text{for} \ z \in \Sigma_0
\end{align}
with
\begin{align*}
\varepsilon_k^\prime(x) = \sup\{d_{x_{i,l},10^4r_l}(P_{j,k},P_{i,l}) : j \in J_k, \ \abs{l-k} \leq 2, \ i \in J_k, x \in 10B_{j,k} \cap 10B_{i,l} \}.
\end{align*}
Then $f = \lim f_N = \lim_N \sigma_0 \circ \dots \circ \sigma_N : \Sigma_0 \rightarrow \Sigma$ is $C(K)$-bi-Lipschitz. 
\end{thm}

\begin{lem}\label{Sigma}
There exists $\e >0$ small enough so that for each $S \in \mathscr{S}$, which is not a singleton, there is a surface $\Sigma_{S}$ such that
\begin{align}\label{Closeness}
\emph{dist}(y,\Sigma_{S}) \lesssim \e ^\frac{1}{d+1}\ell(R)
\end{align}
for each $y \in F \cap \tfrac{M}{4}B_R$ where $R \in S$. Also, for each ball $B$, centred on $\Sigma_S$ and contained in $MB_{Q(S)}$, we have
\begin{align}\label{ee:lowerreg}
 \frac{\omega_d}{2}r_B^d \leq \mathscr{H}^d( \Sigma_{S} \cap B) \lesssim r_B^d.
\end{align}
\end{lem}
\begin{proof}
For $k \geq 0$ let $s(k)$ be such that $5\rho^{s(k)} \leq r_k \leq 5\rho^{s(k) -1}.$ For each  $Q \in \mathscr{S},$ let $L_Q$ be the $d$-plane through $x_Q$ such that 
\[\check\beta_F^{d,1}(MB_Q,L_Q) \leq 2\check\beta_F^{d,1}(MB_Q).\]
For each $k$, let $\mathscr{C}_k^\prime$ be a maximal $r_k$-separated net for 
\[\mathscr{C}_k =\{x_Q : Q \in \mathscr{D}_{s(k)} \cap S\}.\]
By Lemma \ref{AngControl}, $\mathscr{C}_k^\prime$ with planes $\{L_Q\}_{Q \in \mathscr{C}_k^\prime}$ satisfy the assumptions of Theorem \ref{DT} with 
\begin{align}\label{e:betae}
\varepsilon_k(x) \lesssim \check\beta^{d,1}_F(MB_Q) <\e
\end{align}
for any $x \in Q$ with $x_Q \in \mathscr{C}_k$. Let $\Sigma_S^\prime$ be the resulting surface from Theorem \ref{DT}. By \eqref{e:lowerreg}, we can choose $\e$ small enough so that for all balls $B$ centred on $\Sigma_S'$ we have
\begin{align}\label{eee:LR}
 \mathscr{H}^d_\infty (\Sigma_S' \cap B) \geq \frac{\omega_d}{2}r_B^d.
\end{align}

We verify that the additional assumption of Theorem \ref{LipDT} is satisfied, thus $\Sigma_S^\prime$ is in fact a bi-Lipschitz surface. Let $x = f(z) \in \Sigma_{S}^\prime$ and set $x_k = f_k(z)$. By the triangle inequality, we have
\begin{align*}
\abs{x-x_k} &\leq \sum_{j \geq k} |x_{j+1} - x_j| \leq \sum_{j \geq k} |\sigma_j(x_j) - x_j| \\
&\stackrel{\eqref{Sigma_ksigma_k}}{\lesssim} \sum_{j \geq k} \varepsilon_j(x)r_j  \stackrel{\eqref{e:betae}}{\lesssim} \sum_{j \geq k} \e r_j \lesssim \e r_k.
\end{align*}
Thus, for $\e$ small enough, $x_k \in B(x,2r_k).$ Let $Q \in \mathscr{D}_{s(k)} \cap S$ such that $x \in Q.$ Suppose $\ell \in \{k,k-1\}$ and $x_{Q'} \in \mathscr{C}'_\ell$ is such that $x_k \in B(x_{Q'},r_\ell)$ (the existence of $x_{Q'}$ is guaranteed by maximality). It follows that 
\begin{align}
|x_{Q'} - x_Q| \leq |x_{Q'} - x_k| + |x_k - x| + |x- x_Q| \leq r_\ell + 2r_k + r_k \leq 13r_k
\end{align}
For $M$ large enough, this gives
\[ 100B(x_{Q'},r_\ell) \subseteq B(x_Q, 13r_k + 100r_\ell)  \subseteq B(x_Q,1100r_k) \subseteq B(x_Q, M\ell(Q)).\]
Then by Lemma \ref{AngPre}, $\varepsilon_k^\prime(x_k) \lesssim \check\beta_F^{d,1}(MB_Q)$. Since, by our stopping time condition, we have control over the sum of coefficients $\check\beta_F^{d,1}(MB_Q)$, we have verified \eqref{e:LipDT}. We define 
\[\Sigma_S = \Sigma_S^\prime \cap MB_{Q(S)}.\]
Thus, for all balls $B$ centred on $\Sigma_S$ such that $B \subseteq MB_{Q(S)}$, left-most inequality in \eqref{ee:lowerreg} follows by \eqref{eee:LR} and the right-most inequality follows since $\Sigma_S'$ is the bi-Lipschitz image of $\R^d.$  

We check \eqref{Closeness}. Let $R \in \mathscr{D}_{s(k)} \cap S$ and $y \in \tfrac{M}{4}B_R.$ Let $R^\prime \in \mathscr{D}_{s(k)} \cap S$ be a cube such that $x_{R^\prime} \in \mathscr{C}_k^\prime$ and $|x_R - x_{R^\prime}| \leq r_k.$ By the triangle inequality
\begin{align}
|y - x_{R^\prime}| \leq |y - x_R| + |x_R - x_{R^\prime}| \leq \frac{M}{4}\ell(R) + r_k \leq \left(\frac{M}{4} + \rho^{-1}\right)\ell(R).
\end{align}
Choosing $M \geq 4\rho^{-1}$ gives $\tfrac{M}{4}B_{R} \subseteq \tfrac{M}{2}B_{R^\prime}.$ Since $\check\beta_F^{d,1}(MB_{R^\prime},L_{R^\prime}) < 2\e,$ by \eqref{CompL_infty}, $\beta^d_{E,\infty}(\tfrac{M}{2}B_{R^\prime}) \lesssim \e^\frac{1}{d+1},$ hence
\begin{align*}
\text{dist}(y, L_{R^\prime}) \lesssim \e^\frac{1}{d+1}r_k \lesssim \e^\frac{1}{d+1}\ell(R). 
\end{align*}
By \eqref{Sigma_kPlane}, there exists $z \in \Sigma_{S,k}'$ such that $\abs{\pi_{L_{R^\prime}}(y) - z} \lesssim \e r_k$. Furthermore, by \eqref{Sigma_kSigma}, $\text{dist}(z,\Sigma_{S}) \lesssim \e r_k$. Combing the previous estimates we see that \eqref{Closeness} holds. 
\end{proof}

\begin{lem}\label{Similar}
For $M \geq 2C_0$,
\begin{align*}
\sum_{S \in \mathscr{S}}\sum_{Q \in S} \beta_E^{d,p}(C_0B_Q)^2\ell(Q)^d \lesssim  \sum_{S \in \mathscr{S}}\sum_{Q \in S} \check\beta_F^{d,p}(MB_Q)^2\ell(Q)^d  +\sum_{S \in \mathscr{S}} \ell(Q(S))^d.
\end{align*}
\end{lem}

To prove Lemma \ref{Similar} we will need apply the smoothing procedure of David and Semmes (see for example \cite[Chapter 8]{david1991singular}). Let us introduce these smoothed cubes and prove a general fact about them. Let 
\begin{align}\label{e:tau}
0 < \tau < \tau_0 = \tfrac{1}{2(1+\rho)}
\end{align}
and, for each $S \in \mathscr{S}$, let Stop($S$) be the collection of maximal cubes in $\mathscr{D}$ such that $\ell(Q) < \tau d_S(Q),$ that is,
\begin{align*}
\text{Stop}(S) = \{Q \in \mathscr{D} : Q \ \text{is maximal so that} \ \ell(Q) < \tau d_S(Q) \}.
\end{align*}
\begin{lem}\label{l:Stop}
Let $S \in \mathscr{S}$ and suppose $R \in \emph{Stop}(S)$ is such that $R \cap 2C_0B_{Q(S)} \not=\emptyset.$ Then there exists $Q = Q(R) \in S$ such that
\begin{align}\label{distsimR}
\tau\dist(R,Q) \leq \frac{4}{\rho}\ell(R)
\end{align}
and
\begin{align}\label{QsimR}
\frac{\tau\rho}{4}\ell(Q) \leq \ell(R)\leq 3C_0\tau\ell(Q).
\end{align}
\end{lem}

\begin{proof}
Let $Q^\prime \in S$ be a cube such that 
\begin{align}\label{e:new}
2d_S(R) \geq  \ell(Q^\prime) + \text{dist}(Q^\prime,R).
\end{align}
By maximality,
\begin{align}
\frac{1}{\rho} \ell(R) &= \ell(R^{(1)}) > \tau d_S(R^{(1)}) \overset{\eqref{Tr}}{\geq} \tau  \left(d_S(R) - 2\ell(R) - 2\ell(R^{(1)}) \right) \\
&= \tau d_S(R) - 2(1+\rho^{-1})\tau \ell(R).
\end{align}
By our choice of $\tau$ (see \eqref{e:tau}), this gives $\tau d_S(R) \leq \tfrac{2}{\rho}\ell(R).$ Then,
\[ \tau \left(\ell(Q^\prime) + \text{dist}(Q^\prime,R) \right) \overset{\eqref{e:new}}{\leq} 2\tau d_S(R) \leq \frac{4}{\rho} \ell(R) . \]
From here, we see that \eqref{distsimR} and the left hand inequality in \eqref{QsimR} are true. If $\ell(R) \leq 3C_0\tau\ell(Q^\prime),$ we can set $Q=Q^\prime$ and the lemma follows.  Otherwise, since $R \cap 2C_0Q(S) \not=\emptyset,$ we have
\[\ell(R) < \tau (\ell(Q(S)) + \dist(R,Q(S))) \leq 3C_0\tau\ell(Q(S)).\]
We can then choose $Q$ to be the smallest ancestor of $Q'$ contained in $Q(S)$ such that \eqref{QsimR} holds. The existence of such a $Q$ is guaranteed by the above inequality. 
\end{proof}

\begin{proof}[Proof of Lemma \ref{Similar}]
First, by Lemma \ref{SubsetError}, we have
\begin{align*}
\sum_{S \in \mathscr{S}}\sum_{Q \in S} \beta_E^{d,p}(C_0B_Q)^2\ell(Q)^d &\lesssim  \sum_{S \in \mathscr{S}}\sum_{Q \in S} \check\beta_F^{d,p}(2C_0B_Q)^2\ell(Q)^d \\
&\hspace{-4em}+ \sum_{S \in \mathscr{S}}\sum_{Q \in S} \left( \frac{1}{\ell(Q)^d} \int_{F \cap 2C_0B_Q} \left( \frac{\text{dist}(x,E)}{\ell(Q)} \right)^p \, d\mathscr{H}_\infty^d(x) \right)^\frac{2}{p}\ell(Q)^d.
\end{align*}
Since $M \geq 2C_0,$ we have the desired bound on the first term by Lemma \ref{l:cont}. To deal with the second term, let
\begin{align*}
I_S^p \coloneqq \sum_{Q \in S} \left( \frac{1}{\ell(Q)^d} \int_{F \cap 2C_0B_Q} \left( \frac{\text{dist}(x,E)}{\ell(Q)} \right)^p \, d\mathscr{H}_\infty^d(x) \right)^\frac{2}{p}\ell(Q)^d.
\end{align*}
If we can show that 
\begin{align}\label{I_S^p}
I_S^p \lesssim \ell(Q(S))^d
\end{align}
for each $S \in \mathscr{S}$ then the lemma follows. The rest of the proof is devoted to proving \eqref{I_S^p}. Note, we may assume that $p \geq 2,$ since for $p < 2,$ we have $I_S^p \lesssim I_S^2$ by Jensen's inequality (Lemma \ref{JensenPhi}).

Let $S \in \mathscr{S}.$ 
In the case that $S = \{Q\}$ is a singleton, since $C_0B_Q \cap E \not=\emptyset,$ we have $\text{dist}(x,E) \lesssim \ell(Q)$ for each $x \in F \cap 2C_0B_Q$. Then, $I^p_S$ reduces to $\ell(Q)$ and \eqref{I_S^p} follows. 

Assume then that $S$ is not a singleton. First, by Lemma \ref{l:subsum}, we can write
\begin{align}
I_S^p \lesssim \sum_{Q \in S} \left( \frac{1}{\ell(Q)^d} \sum_{\substack{R \in \text{Stop}(S) \\ R \cap 2C_0B_Q \not=\emptyset}} \int_R \left( \frac{\dist(x,E)}{\ell(Q)} \right)^p \, d\mathscr{H}^d_\infty(x)\right)^\frac{2}{p} \ell(Q)^d.
\end{align}
Let $R \in \text{Stop}(S)$ be such that $R \cap 2C_0B_{Q} \not=\emptyset$ for some $Q \in S$. By Lemma \ref{l:Stop}, we can find a cube $Q' \in S$ such that 
\begin{align}\label{e:Q'}
\dist(Q',R) \lesssim \ell(R) \quad \text{and}  \quad \ell(R) \sim \ell(Q').
\end{align}
By our stopping time condition (2), since $Q' \in S,$ we have 
\begin{align}\label{e:dist e}
\text{dist}(y^\prime,E) \leq \e \ell(Q')
\end{align}
for all $y' \in Q'.$ Let $y^{Q'} \in Q'$ and $y^R \in R$ be points such that 
\begin{align}\label{e: dist e1}
\dist(R,Q') = |y^{Q'} - y^R|.
\end{align}
Then, for any $y \in R$, we have
\begin{align}\label{e:distE}
\text{dist}(y,E) &\leq |y-y^{Q'}| + \text{dist}(y^{Q'}, E) \stackrel{\eqref{e:dist e}}{\leq} |y-y^{R}| + |y^{R}- y^{Q'}| + \e \ell(Q') \\
&\stackrel{\eqref{e: dist e1}}{\lesssim} \ell(R) + \dist(R,Q') + \e \ell(Q') \stackrel{\eqref{e:Q'}}{\lesssim} \ell(R).
\end{align}
Using this, along with the that fact that $\frac{2}{p} \leq 1$, we get
\begin{align*}
I_S^p \lesssim  \sum_{Q \in S} \left( \sum_{\substack{R \in \text{Stop}(S) \\ R \cap 2C_0B_Q \not= \emptyset}} \frac{\ell(R)^{d+p}}{\ell(Q)^{d+p}}\right)^\frac{2}{p} \ell(Q)^d  \lesssim \sum_{Q \in S} \sum_{\substack{R \in \text{Stop}(S) \\ R \cap 2C_0B_Q \not= \emptyset}}\frac{\ell(R)^{d\frac{2}{p} + 2 }}{\ell(Q)^{d(\frac{2}{p} - 1) + 2}}.
\end{align*}
\textbf{Claim:} For each $k \in \N,$ we have
\begin{align}\label{e:3.5Overlap}
\#\{Q \in \mathscr{D}_k \cap S : R \cap 2C_0B_Q \not= \emptyset\} \lesssim 1.
\end{align}
\begin{proof}[Proof of Claim]
As before, let $Q'$ be the cube from Lemma \ref{l:Stop} for $R$. By \eqref{e:Q'}, we can choose $M$ large enough so that $R \subseteq \tfrac{M}{4}B_{Q'}$ (taking $M \geq 100C_0/\rho$ is sufficient). In particular, $x_R \in \tfrac{M}{4}B_{Q'}$. Then, by Lemma \ref{Sigma}, 
\begin{align}\label{StopSigma1}
\text{dist}(x_R,\Sigma_S) \stackrel{\eqref{Closeness}}{\lesssim} \e^\frac{1}{d+1}\ell(Q') \stackrel{\eqref{QsimR}}{\lesssim} \frac{\e^\frac{1}{d+1}}{\tau}\ell(R)
\end{align}
with $\tau$ as in \eqref{e:tau}. For any $Q \in \mathscr{D}_k \cap S$ such that $R \cap 2C_0B_Q \not= \emptyset,$ we have 
\begin{align}\label{e:R<Q}
\ell(R) < \tau d_S(R) \leq \tau(\ell(Q) + \text{dist}(Q,R)) \lesssim \tau\ell(Q).
\end{align}
Let $x_{R}'$ be the point in $\Sigma_S$ closest to $x_R.$ By \eqref{StopSigma1} and \eqref{e:R<Q} there exists $A>0$ so that if $Q \in \mathscr{D}_k \cap S$ is such that $R \cap 2C_0B_Q \not= \emptyset$ then 
\[Q \subseteq B(x_R',A\ell(Q)) \coloneqq B. \]
Since $\Sigma_S$ is $(C\e,d)$-Reifenberg flat, we can find a plane $P$ through $x_{R}'$ such that 
\begin{align}\label{e:flat}
d_B(P,\Sigma_S) \lesssim \e.
\end{align} Let $x_Q^\prime$ be the point in $\Sigma_S$ which is closest to $x_Q$, then
\begin{align*}
\text{dist}(x_Q , P) \leq |x_Q - x_Q^\prime| + \text{dist}(x_Q^\prime,P) \stackrel{\substack{\eqref{Closeness} \\ \eqref{e:flat}}}{\lesssim} (\e^\frac{1}{d+1} + \e) \ell(Q).
\end{align*}
So, for $\e$ small enough, we have $\text{dist}(x_Q,P) \leq c_0\ell(Q)/2.$ Then \eqref{e:3.5Overlap} follows from Lemma \ref{ENV}.
\end{proof}

Returning to $I_S^p,$ since by assumption, $p < 2d/(d-2),$ or equivalently, $d(2/p -1) +2 >0,$ we can swap the order of integration, apply \eqref{e:3.5Overlap}, and sum over a geometric series and obtain
\begin{align*}
I_S^p \lesssim \sum_{\substack{R \in \text{Stop}(S)  \\ R \cap 2C_0B_{Q(S)} \not=\emptyset}} \sum_{\substack{Q \in S \\ R \cap 2C_0B_Q \not= \emptyset}}\frac{\ell(R)^{d\frac{2}{p} + 2 }}{\ell(Q)^{d(\frac{2}{p} - 1) + 2}} \lesssim  \sum_{\substack{R \in \text{Stop}(S) \\ R \cap 2C_0B_{Q(S)} \not=\emptyset}} \ell(R)^d. 
\end{align*}
By \eqref{StopSigma1}, for $\e \ll \tau^{d+1}$, each $c_0B_R$ carves out a large proportion of $\Sigma_S$ i.e. $\mathscr{H}^d(c_0B_R \cap \Sigma_S) \gtrsim \ell(R)^d.$ Since the $c_0B_R$ are disjoint and $\Sigma_S$ is bi-Lipschitz, we have
\begin{align*}
I_S^p \lesssim \sum_{\substack{R \in \text{Stop} \\ R \cap 2C_0B_{Q(S)} \not=\emptyset}} \mathscr{H}^d(c_0B_R \cap \Sigma_S) \leq \mathscr{H}^d(B_{Q(S)} \cap \Sigma_S) \lesssim \ell(Q(S))^d,
\end{align*}
completing the proof of \eqref{I_S^p} and hence the proof of the lemma. 
\end{proof}

\begin{lem}\label{TopControl}
\begin{align}
\sum_{S \in \mathscr{S}} \ell(Q(S))^d \lesssim \mathscr{H}^d(Q_0) + \sum_{Q \in \mathscr{D}(Q_0)} \check\beta_F^{d,1}(MB_Q)^2\ell(Q)^d.
\end{align}
\end{lem}
Lemma \ref{TopControl} along with Lemma \ref{Similar} finishes the proof of \eqref{Suffices1} (and hence finishes the proof of Theorem \ref{Thm1}). Let $\min\mathscr{S}$ be the collection of minimal cubes from $\mathscr{S},$ i.e. 
\[\min\mathscr{S} = \bigcup_{S \in \mathscr{S}} \{Q \in \min(S) \}. \] 
Each $Q(S)$, with the exception of $Q_0$, is the child of some cube in $\min \mathscr{S}.$ Since any cube in $\mathscr{D}^F$ has at most $C(n)$ children, and $F$ is lower regular, we have 
\begin{align}
\sum_{S \in \mathscr{S}} \ell(Q(S))^d &\leq \ell(Q_0)^d + \sum_{Q \in \min \mathscr{S}} \sum_{R \in \text{Child}(Q)} \ell(R)^d \\
&\lesssim_n \ell(Q_0)^d + \sum_{Q \in \min \mathscr{S}} \ell(Q)^d \\
&\lesssim \mathscr{H}^d(Q_0) + \sum_{Q \in \min \mathscr{S}} \ell(Q)^d. 
\end{align}
So, to prove Lemma \ref{TopControl}, it suffices to show the following. 

\begin{lem}\label{MinControl}
\begin{align}
\sum_{Q \in \min\mathscr{S}} \ell(Q)^d \lesssim \mathscr{H}^d(Q_0) + \sum_{Q \in \mathscr{D}(Q_0)} \check\beta_F^{d,1}(MB_Q)^2\ell(Q)^d. 
\end{align}
\end{lem}
\begin{proof}
We split $\min\mathscr{S}$ into two sub families, $\mathscr{F}_1$ and $\mathscr{F}_2$ where $\mathscr{F}_1$ is the collection of cubes $R$ in $\min\mathscr{S}$ such that $R$ has a child $R^\prime$ with
\begin{align}\label{Stop}
\sum_{\substack{Q \in S(R) \\ Q \supseteq R^\prime}} \check\beta_F^{d,1}(MB_Q)^2 \geq \e^2
\end{align}
and $\mathscr{F}_2 = \min\mathscr{S} \setminus \mathscr{F}_1$ (recall the definition of $S(R)$ from Definition \ref{d:ST}). We deal with $\mathscr{F}_1$ first. Let $R \in \mathscr{F}_1$ and let $R^\prime$ be its child satisfying \eqref{Stop}. Note that if $\check\beta_F^{d,1}(MB_{R^\prime}) \leq \e^2/2,$ then
\begin{align}
\sum_{\substack{Q \in S(R) \\ Q \supseteq R}} \check\beta_F^{d,1}(MB_Q)^2 \geq \e^2/2.
\end{align}
If $\check\beta_F^{d,1}(MB_{R^\prime}) > \e^2/2,$ by Lemma \ref{l:cont} we have 
\begin{align}
\sum_{\substack{Q \in S(R) \\ Q \supseteq R}} \check\beta_F^{d,1}(MB_Q)^2 \geq \check\beta_F^{d,1}(MB_R)^2 \gtrsim \check\beta_F^{d,1}(MB_{R^\prime})^2 \geq \e^2/2.
\end{align}
In either case it follows that
\begin{align}
\sum_{\substack{Q \in S(R) \\ Q \supseteq R}} \check\beta_F^{d,1}(MB_Q)^2 \gtrsim \e^2.
\end{align}
Thus, 
\begin{align}\label{MinCont1}
\sum_{R \in \mathscr{F}_1} \ell(R)^d &\lesssim \frac{1}{\e^2} \sum_{R \in \mathscr{F}_1} \ell(R)^d \sum_{\substack{Q \in S(R) \\ Q \supseteq R}} \check\beta_{F}^{d,1}(MB_Q)^2  \nonumber \\
&\lesssim \sum_{S \in \mathscr{S}} \sum_{Q \in S} \check\beta_F^{d,1}(MB_Q)^2 \sum_{\substack{R \in S \cap \mathscr{F}_1 \\ R \subseteq Q}} \ell(R)^d. 
\end{align}
In the case that $S$ is not a singleton, by Lemma \ref{Sigma}, $\text{dist}(x_R, \Sigma_{S(R)}) \lesssim \e^\frac{1}{d+1}\ell(R).$ Thus, for $\e > 0$, small enough $c_0B_R$ carves out a large proportion of $\Sigma_S$, hence $\mathscr{H}^d( c_0B_R \cap \Sigma_{S(R)}) \gtrsim \ell(R)^d.$ The balls $c_0B_R$ are disjoint and contained in $B_Q$, so
\begin{align}\label{MinCont2}
 \sum_{\substack{R \in S \cap \mathscr{F}_1 \\ R \subseteq Q}} \ell(R)^d \lesssim \sum_{\substack{R \in S \cap \mathscr{F}_1 \\ R \subseteq Q}} \mathscr{H}^d(c_0B_R \cap \Sigma_{S(R)}) \leq \mathscr{H}^d( B_Q \cap \Sigma_{S(R)}) \lesssim \ell(Q)^d.
\end{align}
If $S$ is a singleton, $\Sigma_S$ was not defined but we have $S = \{Q\} = \{R\},$ so the above estimate holds trivially. In either case, combining \eqref{MinCont1} and \eqref{MinCont2} gives
\begin{align}\label{F_1}
\sum_{R \in \mathscr{F}_1} \ell(R)^d \lesssim \sum_{S \in \mathscr{S}} \sum_{Q \in \mathscr{S}} \check\beta_F^{d,1}(MB_Q)^2\ell(Q)^d. 
\end{align}
Let us consider $\mathscr{F}_2.$ For $R \in \mathscr{F}_2,$ we know there exists a child $R^\prime$ of $R$ and a point $x^\prime \in R^\prime$ such that $\text{dist}(x^\prime,E) \geq \e\ell(R^\prime).$ We let $C_R$ be the maximal cube in $\mathscr{D}$ containing $x^\prime$ such that $\ell(C_R) \leq \tfrac{\e\rho}{C_0}\ell(R).$ This implies that
\begin{align}\label{CompC_R}
\frac{\e\rho^2}{C_0}\ell(R) \leq \ell(C_R) \leq \frac{\e\rho}{C_0}\ell(R).
\end{align}
In this way, we have $\dist(y,E) > 0$ for all $y \in 2C_0B_{C_R}$ since
\begin{align}
\e \ell(R) &\leq \dist(x',E) \leq |x' - y| + \dist(y,E) \leq 2C_0\ell(C_R) + \dist(y,E) \\
&\leq 2\e \rho \ell(R) + \dist(y,E).
\end{align}
In particular, 
\begin{align}\label{e:empty}
2C_0B_{C_R} \cap E =\emptyset.
\end{align}
Let us show that the $C_R$ have bounded overlap. Let $x \in F, N \in \N$ and suppose $\{R_i\}_{i=1}^N$ is a finite collection of distinct cubes in $\mathscr{F}_2$ such that 
\[ x \in \bigcap_{i=1}^N C_{R_i}. \]
We begin by showing the that each $\ell(R_i)$ is comparable. Assume without loss of generality that $\ell(R_i) > \ell(R_j)$ for all $i < j.$ In particular, we have 
\[ \ell(R_1) > \ell(R_j) \]
for all $j = 2,\dots,N.$ We also see that 
\begin{align}\label{e:gtr}
\ell(R_j) \geq \frac{\e \rho^2}{C_0} \ell(R_1)
\end{align}
for all $j = 2,\dots,N,$ otherwise, by \eqref{CompC_R} we have $R_N \subseteq C_{R_1}$ which by \eqref{e:empty} gives that $2C_0B_{R_N} \cap E =\emptyset.$ Thus, we arrive at a contradiction since $E \cap 2C_0B_Q \not=\emptyset$ for any $Q \in \mathscr{F}_2$ by our stopping time condition (2) (see Remark \ref{r:int}). We conclude then
\begin{align}\label{ee:sim}
\ell(R_i) \sim \ell(R_j) \quad \text{for all} \ i\not=j. 
\end{align}
Since $c_0B_{R_i} \cap c_0B_{R_j} = \emptyset$ for $i \not= j$ such that $\ell(R_i) = \ell(R_j)$, a standard volume argument coupled with \eqref{ee:sim} gives
\begin{align}\label{N < 1}
N \lesssim_{\e,C_0,n} 1.
\end{align}
Now, since $F$ is lower content regular, we get
\begin{align}\label{F_2}
\sum_{R \in \mathscr{F}_2} \ell(R)^d \stackrel{\eqref{CompC_R}}{\lesssim} \sum_{R \in \mathscr{F}_2} \mathscr{H}^d(C_R) \stackrel{\eqref{N < 1}}{\lesssim} \mathscr{H}^d(Q_0).
\end{align}
Combining \eqref{F_1} and \eqref{F_2} completes the proof of the lemma.
\end{proof}

\section{Proof of Theorem \ref{Thm4}}\label{s:Thm4}

\subsection{Constructing \texorpdfstring{$F$}{F} and its properties}

Let $X_k^E$ be a sequence of maximal $\rho^k$-separated nets in $E$ and let $\mathscr{D}^E$ be the cubes from Theorem \ref{cubes} with respect to these maximal nets. By scaling and translation, we can assume there is $Q_0 \in \mathscr{D}^E_0$ which contains $0.$ Let $M \geq 1$ and $\e,\alpha >0$.
\begin{rem}
The constants $M$ and $\e$ may be different from the previous section. We shall choose $M$ sufficiently large, $\e$ and $\alpha$ shall be chosen sufficiently small.
\end{rem}
We wish to construct $F$ using the Reifenberg parametrisation theorem of David and Toro. Since $E$ is not lower regular, we need the condition that we have $(d+1,\alpha)$-separated points to apply the construction. This will be added to our stopping time conditions. Note, in contrast to the previous section, we shall run our stopping time algorithm over cubes in $E$. 

For $Q \in \mathscr{D}^E(Q_0)$, we define a stopping time region $S_Q$ as follows. First, add $Q$ to $S_Q.$ Then, inductively on scales, we add a cube $R$ to $S_Q$ if $R^{(1)} \in S_Q$ and each sibling $R^\prime$ of $R$ satisfies:
\begin{enumerate}
\item 
\begin{align}
\sum_{Q \supseteq T \supseteq R^\prime} \beta_E^{d,1}(MB_Q)^2 \leq \e^2. 
\end{align}
\item $MB_{R^\prime}$ has $(d+1,\alpha)$ separated points (in the sense of Definition \ref{Sep}). 
\end{enumerate}
We partition $\mathscr{D}^E(Q_0)$ into a collection of stopping time regions $\mathscr{S}:$ Begin by adding $S_{Q_0}$ to $\mathscr{S}.$ Then, if $S$ has been added to $\mathscr{S}$, add $S_Q$ to $\mathscr{S}$ if $Q \in \text{Child}(R)$ for some $R \in \min(S).$ We continue in this way to generate $\mathscr{S}.$ Let
$$\min \mathscr{S} = \bigcup_{S \in \mathscr{S}} \{Q \in \min(S)\}$$
denote the collection of all minimal cubes in $\mathscr{S}.$

\begin{rem}The following is essentially Lemma \ref{Sigma}. The proof follows the same if we take $\epsilon \leq \alpha^{4d+4}.$
\end{rem}

\begin{lem}\label{Sigma1}
There exists $\e >0$ small enough so that for each $S \in \mathscr{S}$, which is not a singleton, there is a surface $\Sigma_{S}$ which satisfies the following. Firstly, if $\Sigma_S'$ denotes the bi-Lipschitz surface from Theorem \ref{LipDT}, then 
\begin{align}\label{e:sigmaM}
\Sigma_S = \Sigma_S' \cap MB_{Q(S)}.
\end{align}
Second, for each $R \in S$ and  $y \in F \cap \tfrac{M}{4}B_R$,
\begin{align}\label{Closeness1}
\emph{dist}(y,\Sigma_{S}) \lesssim \e ^\frac{1}{d+1}\ell(R).
\end{align}
Finally, for each ball $B$ centred on $\Sigma_S$ and contained in $MB_{Q(S)}$, 
\begin{align}\label{ee:lowerreg1}
\frac{\omega_d}{2}r_B^d \leq \mathscr{H}^d( \Sigma_{S} \cap B) \lesssim r_B^d.
\end{align}
\end{lem}

\begin{defn}
If $S = \{Q\}$ is a singleton, we define $\Sigma_S = P_Q \cap MB_Q,$ where $P_Q$ is some $d$-plane through $x_Q$ such that $\beta^{d,p}_E(MB_Q,P_Q) \leq 2\beta_E^{d,p}(MB_Q).$   
\end{defn}

Let $\mathscr{B}_\alpha \subseteq \mathscr{D}^E(Q_0)$ be the set of cubes $Q$ which have a sibling $Q'$ which fails the stopping time condition (2) i.e. $MB_{Q'}$ does not have $(d+1,\alpha)$-separated points. Consider those points in $Q_0$ for which we stopped a finite number of times or we never stopped. That is, we consider 
\[G \coloneqq \left\{ x \in Q_0 : \sum \beta_E^{d,1}(MB_Q)^2\chi_Q(x) < \infty \ \text{and} \ \sum_{\substack{Q \\ Q \in \mathscr{B}_\alpha}} \chi_Q(x) < \infty \right\}. \]
Observe that $G \subseteq \bigcup_{S \in \mathscr{S}} \Sigma_{S}.$ We define $E^\prime = Q_0 \setminus G$ and set 
\[F \coloneqq E^\prime \cup \bigcup_{S \in \mathscr{S}} \Sigma_{S}. \] 

\begin{lem}
$F$ is $(c_1,d)$-lower content regular, with $c_1$ as in Remark \ref{r:consts}
\end{lem}

\begin{proof}
Fix $x \in F$ and $r>0.$ Assume first of all that $x \in \Sigma_{S}$ for some $S \in \mathscr{S}.$ Let $Q = Q(S).$ For any $k \geq 0,$ recalling that $Q^{(k)}$ is the $k^{th}$ generational ancestor of $Q$, we have
\begin{align}
|x - x_{Q^{(k)}}| \leq |x - x_Q| + |x_Q - x_{Q^{(k)}}| \leq M\ell(Q) + \ell(Q^{(k)}) \leq 2M\ell(Q^{(k)}).
\end{align}
Assume first that $r \geq 3M\ell(Q)$ and let $k = k(r) \geq 0$ be the integer such that
\[ 
3M\ell(Q^{(k)}) \leq r \leq 3M\ell(Q^{(k+1)}).
\]
Then,
\[ MB_{Q^{(k)}} \subseteq B(x, 2M\ell(Q^{(k)}) + M\ell(Q^{(k)})) \subseteq B(x,r). \]
The lower regularity follow since
\begin{align}
\mathscr{H}_\infty^d(F \cap B(x,r)) \geq \mathscr{H}_\infty^d(\Sigma_{S(Q^{(k)})}) \overset{\eqref{ee:lowerreg1}}{\geq} \frac{\omega_d}{2}(M\ell(Q^{(k)}))^d \geq \frac{\omega_d \rho^d}{2\cdot 3^{d}} r^d \geq c_1 (2r)^d.
\end{align}
Assume now that $r < 3M\ell(Q(S)).$ If $B(x,r) \subseteq MB_{Q(S)}$ we can trivially apply the lower regularity estimates for $\Sigma_S$, so it suffices to consider the case when $B(x,r) \not\subseteq MB_{Q(S)}.$ We split this into two further sub-cases: either $B(x,r) \cap 10B_{Q(S)}= \emptyset$ or $B(x,r) \cap 10B_{Q(S)}\not= \emptyset.$

In the first sub-case, since by \eqref{e:10}, we have $\Sigma_S \setminus 10B_{Q(S)} = P_{Q(S)} \setminus 10B_{Q(S)},$ the portion of $\Sigma_S$ contained in $B(x,r)$ is just a $d$-plane through $x.$ Since $x$ is contained in $MB_{Q(S)}$, we can find $y \in B(x,r) \cap \Sigma_S$ such that $B(y,r/4) \subseteq MB_{Q(S)} \cap B(x,r)$ and apply the lower regularity estimates for $\Sigma_S$ inside $B(y,r/4)$ to obtain 
\[ \mathscr{H}^d_\infty( F \cap B(x,r)) \geq \mathscr{H}^d_\infty( \Sigma_S \cap B(y,r/4) ) \overset{\eqref{ee:lowerreg1}}{\geq} \frac{\omega_d}{2\cdot 4^{d}} r^d \geq c_1 (2r)^d. \]
See for example the left image in Figure \ref{fig:lowreg}. 

In the second sub-case, since $B(x,r) \not\subseteq MB_{Q(S)}$ but $B(x,r) \cap 10B_{Q(S)}\not= \emptyset,$ it must be that $r$ is comparable with $M\ell(Q).$ For $M$ sufficiently large ($M \geq 100$ is sufficient), we can certainly find $y \in B(x,r) \cap \Sigma_S$ such that $B(y,r/10) \subseteq MB_{Q(S)} \cap B(x,r)$ and apply the lower regularity estimates for $\Sigma_S$ inside $B(y,r/10)$ to get
\[ \mathscr{H}^d_\infty( F \cap B(x,r)) \geq \mathscr{H}^d_\infty( \Sigma_S \cap B(y,r/10) ) \overset{\eqref{ee:lowerreg1}}{\geq} \frac{\omega_d}{2\cdot 10^{d}} r^d \geq c_1 (2r)^d. \]
See for example the right image in Figure \ref{fig:lowreg}.

\begin{figure}[htbp]
  \centering
  \includegraphics[scale = 0.58]{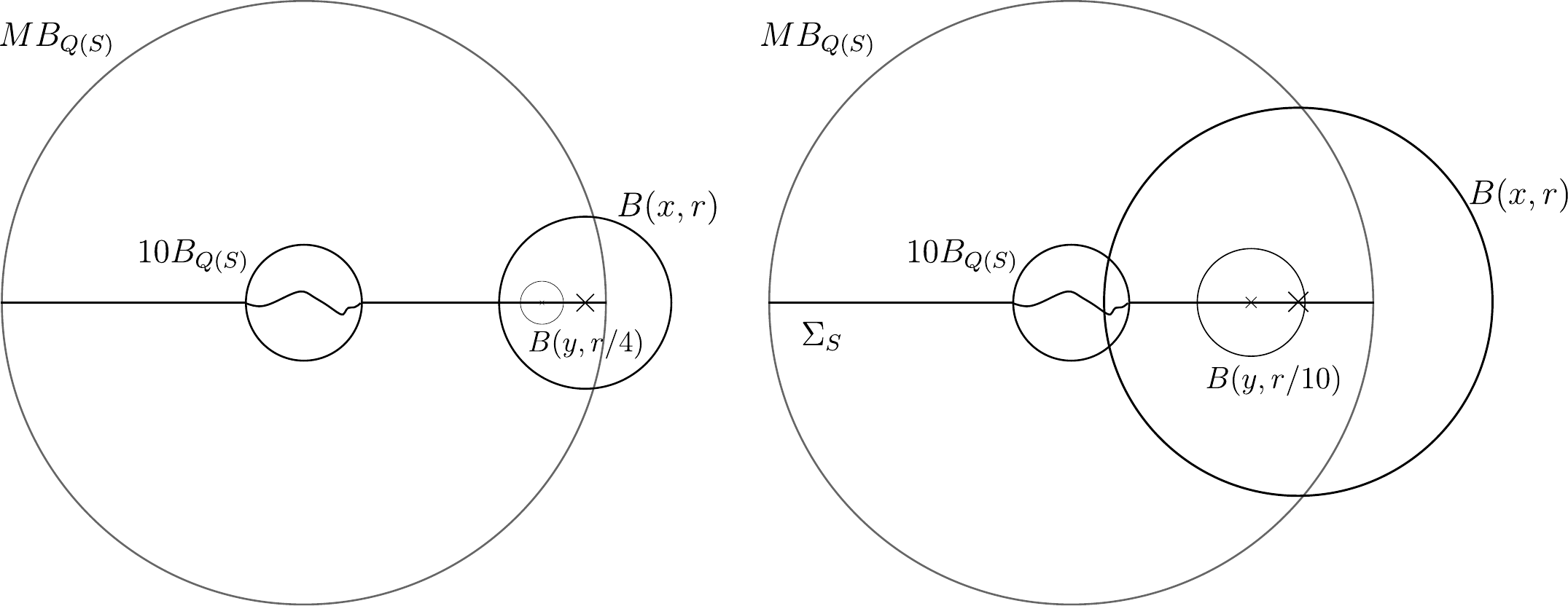}
  \caption{Sub-cases 1 and 2}\label{fig:lowreg}
\end{figure}

Suppose now that $x \in E^\prime.$ Since $E^\prime$ is the collection of points where we stopped an infinite number of times, we may find a sequence of stopping time regions $S_i$ such that $x \in Q(S_{i})$ and $\ell(Q(S_{i})) \downarrow 0.$ We denote by $S$ the stopping time region for which $x \in S$ and $\ell(Q(S)) \leq r/10.$ Then, by \eqref{Closeness1},
\[ \text{dist}(x ,\Sigma_{S}) \lesssim \e^\frac{1}{d+1}\ell(Q(S)) \leq \e^\frac{1}{d+1} \frac{r}{10}.\]
Let $x^\prime$ be the point in $\Sigma_{S}$ closest to $x.$ For $\e$ small enough, $B(x^\prime, r/2) \subseteq B(x,r)$, which by our above considerations gives
\[\mathscr{H}^d(F \cap B(x,r)) \geq \mathscr{H}^d(F \cap B(x^\prime, r/2)) \geq c_1 r^d, \]
 as required.
\end{proof}

We plan to prove \eqref{e:Thm4} with $F$ as above. When proving \eqref{e:Thm4}, it will be useful for us to have a bound for $\mathscr{H}^d(F)$ in terms of $\beta$-numbers.

\begin{lem}\label{TildeF}
Assume that 
\begin{align}\label{finite}
\sum_{Q \in \mathscr{D}^E(Q_0)} \beta^{d,1}_E(MB_Q)^2\ell(Q)^d < \infty.
\end{align}
Then $\mathscr{H}^d(E^\prime) =0$, from which it follows that $F$ is rectifiable and 
\begin{align}\label{FMeasure}
\mathscr{H}^d(F) \lesssim \ell(Q_0)^d +  \sum_{Q \in \mathscr{D}^E(Q_0)} \beta^{d,1}_E(MB_Q)^2\ell(Q)^d.
\end{align}
\end{lem}

Before proving Lemma \ref{TildeF}, we shall need two preliminary results. 

\begin{lem}\label{l:minQ}
Let $S \in \mathscr{S}$ and $Q \in S.$ For $\e>0$ small enough, we have 
\begin{align}\label{e:minQ}
\sum_{\substack{R \in \min(S) \\ R \subseteq Q}} \ell(R)^d \lesssim \ell(Q)^d.
\end{align}

\begin{proof}
Let $R \in \min(S)$ satisfy $R \subseteq Q.$ By \eqref{Closeness1}, the ball $c_0B_R$ carves out a large proportion of $\Sigma_S,$ in particular, $\mathscr{H}^d( \Sigma_S \cap c_0B_R) \gtrsim \ell(R)^d.$ Furthermore, the balls $\{c_0B_R\}$ are disjoint and contained in $B_Q.$ Using this, we have
\begin{align}\label{e1:minS}
\sum_{\substack{R \in \min(S) \\ R \subseteq Q}} \ell(R)^d \lesssim  \sum_{\substack{R \in \min(S) \\ R \subseteq Q}} \mathscr{H}^d( \Sigma_S \cap c_0B_R)  \leq \mathscr{H}^d(\Sigma_S \cap B_Q) \lesssim \ell(Q)^d. 
\end{align}
\end{proof}
\end{lem}

\begin{lem}\label{MinCubes}
\begin{align}
\sum_{R \in \min\mathscr{S}}\ell(R)^d \lesssim \ell(Q_0)^d + \sum_{S \in \mathscr{S}}\sum_{Q\in S} \beta_E^{d,1}(MB_Q)^2\ell(Q)^d.
\end{align}
\end{lem}

\begin{proof}
We split $\min\mathscr{S}$ into two sub families, $\text{Type}_\text{\romup{1}}$ and $\text{Type}_\text{\romup{2}}$, where $\text{Type}_\text{\romup{1}}$ is the collection of cubes $R  \in \min\mathscr{S}$ such that each child $R'$ of $R$ has $(d+1,\alpha)$-separated points, but there is a child $R''$ such that
\begin{align}\label{Stop1}
\sum_{\substack{Q \in S(R) \\ Q \supseteq R''}} \beta_E^{d,1}(MB_Q)^2 \geq \e^2
\end{align}
and $\text{Type}_\text{\romup{2}} = \min\mathscr{S} \setminus \text{Type}_\text{\romup{1}}.$ Controlling cubes in $\text{Type}_\text{\romup{1}}$ is done in a similar way to how we controlled the cubes in $\mathscr{F}_1$ in the proof of Lemma \ref{MinControl}, the only difference being that we construct the surfaces $\Sigma_S$ by Lemma \ref{Sigma1} instead of Lemma \ref{Sigma}. This is because $E$ is not necessary lower regular but we know each cube in a stopping time region (which is not a singleton) has $(d+1,\alpha)$-separated points. This gives

\begin{align}
\sum_{Q \in \text{Type}_\text{\romup{1}}} \ell(Q)^d \lesssim \sum_{Q \in \mathscr{D}^E(Q_0)} \beta_E^{d,1}(MB_Q)^2\ell(Q)^d. 
\end{align}
We now consider cubes in $\text{Type}_\text{\romup{2}}$. We will in fact show that 
\begin{align}\label{e:www}
\sum_{Q \in \text{Type}_\text{\romup{2}}} \ell(Q)^d \lesssim \ell(Q_0)^d +  \sum_{Q \in \text{Type}_\text{\romup{1}}} \ell(Q)^d 
\end{align}
from which the lemma follows immediately. We state and prove three preliminary claims before we proceed with the proof of \eqref{e:www}. 
\bigbreak
\noindent
\textbf{Claim 1:} For all $\delta > 0$ there exists $\alpha >0$ such that if $Q \in \text{Type}_\text{\romup{2}}$ and
\begin{align}
k^* = \floor*{\frac{\log (2 \alpha M\rho/c_0)}{\log \rho}},
\end{align}
then
\begin{align}\label{lowdim1}
\sum_{R  \in \text{Child}_{k^*}(Q)} \ell(R)^d  \leq \delta \ell(Q)^d.
\end{align}
Let $\delta >0$ and assume that $Q \in \text{Type}_\text{\romup{2}}$. By definition, there exists $Q^\prime \in \text{Child}(Q)$ which does \textit{not} have $(d+1,\alpha)$-separated points. This is equivalent to saying $\beta_{E,\infty}^{d-1}(MB_{Q^\prime}) \leq \alpha.$ If $L$ is the $(d-1)$-plane such that 
\[\beta_{E,\infty}^{d-1}(MB_{Q^\prime}) = \beta_{E,\infty}^{d-1}(MB_{Q^\prime},L),\]
then each $R \subseteq Q$ satisfies 
$$\text{dist}(x_R,L) \leq \alpha M\ell(Q^\prime) = \alpha M\rho\ell(Q).$$ 
By our choice of $k^*$, if $R \subseteq Q$ is such that $\ell(R) = \rho^{k^*}\ell(Q)$ then $\text{dist}(x_{R},L) \leq \tfrac{c_0}{2}\ell(R).$ Thus, by Lemma \ref{ENV},
\begin{align}
\sum_{R  \in \text{Child}_{k^*}(Q)} \ell(R)^{d-1} \leq C\ell(Q)^{d-1}.
\end{align}
Multiplying both sides by $\ell(R),$ we obtain
\begin{align}
\sum_{R  \in \text{Child}_{k^*}(Q)} \ell(R)^d  \leq C\ell(Q)^{d-1}\ell(R) \leq C \rho^{k^*} \ell(Q)^{d}.
\end{align}
By taking $\alpha >0$ small enough, we can ensure that $C\rho^{k^*} < \delta.$ This proves the claim. 

\bigbreak

Let $Q \in \text{Type \romup{2}}$. Define $\text{Type}_\text{\romup{1}}(Q)$ to be the maximal collection of cubes in $\text{Type}_\text{\romup{1}}$ contained in $Q$. Let $\text{Tree}(Q)$ be the collection of cubes $R \in \mathscr{D}^E$ such that $R \subseteq Q$ and $R$ is not properly contained in any cube from $\text{Type}_\text{\romup{1}}(Q)$ and let
\[\text{Type}_\text{\romup{2}}(Q) = \text{Tree}(Q) \cap \text{Type}_\text{\romup{2}}.\] 
Recall the definition of $S(Q)$ from Definition \ref{d:ST}. We define sequences of subsets of $\text{Tree}(Q)$, $\mathscr{T}_k$ and $\mathscr{M}_k$ as follows. Let
\[ \mathscr{T}_0 = \{Q\} \quad \text{and} \ \mathscr{M}_0 = \{ R \in \min (S(Q)) : R \subseteq Q\}.  \] 
Then, supposing that $\mathscr{T}_k$ and $\mathscr{M}_k$ have been defined for some integer $k \geq 0,$ we let
\[ \mathscr{T}_{k+1} = \{T \in \text{Tree}(Q) : T \in \text{Child}_{k^*}(R) \ \text{for some} \ R \in \mathscr{M}_k \} \]
and
\[ \mathscr{M}_{k+1} = \{T \in \text{Type}_\text{\romup{2}}(Q): T \ \text{is max so that} \ T \subseteq R \ \text{for some} \ R \in \mathscr{T}_{k+1} \} \]
Recalling the definition of $\text{Des}_{k}(R)$ from \eqref{d:Des}, we have the following:
\bigbreak
\noindent
\textbf{Claim 2:} 
\begin{align}\label{e:TypeII}
\text{Type}_\text{II}(Q) \subseteq \{Q\} \cup \bigcup_{k=0}^\infty \bigcup_{T \in \mathscr{M}_k} \text{Des}_{k^*-1}(T). 
\end{align}
\bigbreak
Let $R \in \text{Type}_{\text{II}}(Q).$ The claim is clearly true for $R =Q$ so let us assume $R \not=Q.$ Let $k_R $ be the largest integer $k \geq 0$ such that there is a cube $T \in \mathscr{M}_k$ with $R \subseteq T.$ The existence of such a $k$ is guaranteed since there exists $\tilde{T} \in \mathscr{M}_0$ such that $R \subseteq \tilde{T}$ and each cube $T \in \mathscr{M}_k$ satisfies $\ell(T) \leq \rho^{kk^*}\ell(Q).$ Let $T_R  \in \mathscr{M}_{k_R}$ be such that $R \subseteq T_R.$ Assume $R \in \text{Child}_k(T_R)$ for some $k \geq k^*.$ If this is the case then we can find a cube $R' \in \text{Child}_{k^*}(T_R)$ such that $R \subseteq R'$, recall by definition that $R' \in \mathscr{T}_{k_R+1}.$ By maximality this implies that there exists some $R'' \in \mathscr{M}_{k_R+1}$ such that $R \subseteq R''$, which contradicts the definition of $k_R.$ It follows that $R \in \text{Des}_{k^*-1}(T_R)$ which completes the proof of the claim. 
\bigbreak
\noindent
\textbf{Claim 3:} There exists $\alpha >0$ small enough so that for each $k \geq 0,$ 
\begin{align}\label{e:geosum}
\sum_{R \in \mathscr{M}_k} \ell(R)^d \leq c_3 \left( \frac{1}{2} \right)^k \ell(Q)^d,
\end{align} 
where $c_3$ is the implicit constant from Lemma \ref{l:minQ}.
\bigbreak
The results holds for $k=0$ by Lemma \ref{l:minQ}. Assume $k \geq 1$ and assume the result holds for $k -1 \geq 0$. We will show that it holds for $k$. Let $R \in \mathscr{M}_k$. By definition, then there exists a unique $T \in \mathscr{T}_k$ such that $R \subseteq T.$ By maximality it must be that $T \in \mathscr{T}_k \cap S(R).$ Then, by Lemma \ref{l:minQ}, Claim 1 and the definitions of $\mathscr{T}_k$ and $\mathscr{M}_k$, we have 
\begin{align}
\sum_{R \in \mathscr{M}_k} \ell(R)^d &\leq \sum_{T \in \mathscr{T}_k} \sum_{\substack{R \in \mathscr{M}_k \\ R \subseteq T}} \ell(R)^d \stackrel{\eqref{e1:minS}}{\leq} C \sum_{T  \in \mathscr{T}_k} \ell(T)^d \stackrel{\eqref{lowdim1}}{\leq} C \delta \sum_{T \in \mathscr{M}_{k-1}} \ell(T)^d \\
&\leq c_3 C\delta \left(\frac{1}{2}\right)^{k-1} \ell(Q)^d. 
\end{align}
Choosing $\delta$ (and hence $\alpha$) small enough so that $C\delta < 1/2$ we get the required result.

\bigbreak

We return to the proof of \eqref{e:www}. Our first goal is to find a bound on the sum over cubes in $\text{Type}_{\text{II}}(Q).$ Suppose $R,R' \in \text{Tree}(Q)$ are such that $R' \subset R.$ Any cube $T \in \mathscr{D}$ such that $R' \subset T \subseteq R$ is either the child of a $\text{Type}_{\text{II}}$ cube or is contained in some stopping that is not a singleton. In either case
\[\beta_{E,\infty}^{d-1}(MB_T) \leq \alpha \quad \text{or} \quad \beta_E^{d,1}(MB_T) \leq \e. \]
So, by Lemma \ref{l:des} (for $\alpha$ and $\e$ small enough), 
\begin{align}
\sum_{T \in \text{Des}_{k^*-1}(R)} \ell(T)^d \lesssim_{\alpha,d,M} \ell(R)^d.
\end{align}
Combing all the above, we get
\begin{align}
\sum_{R \in \text{Type}_\text{\romup{2}}(Q)} \ell(R)^d &\stackrel{\eqref{e:TypeII}}{\leq} \ell(Q)^d + \sum_{k=0}^\infty \sum_{R \in \mathscr{M}_k} \sum_{T \in \text{Des}_{k^*-1}(R)} \ell(T)^d \\
&\lesssim \ell(Q)^d + \sum_{k=0}^\infty \sum_{R \in \mathscr{M}_k} \ell(R)^d \\
&\stackrel{\eqref{e:geosum}}{\lesssim}  \ell(Q)^d + \sum_{k=0}^\infty \left( \frac{1}{2} \right)^k\ell(Q)^d \lesssim \ell(Q)^d. 
\end{align}
Now, for each $Q \in \text{Type}_\text{\romup{2}},$ let $R(Q)$ be the smallest cube in $\text{Type}_\text{\romup{1}}$ which contains $Q$. By construction, if $T \in \text{Tree}(Q)$ for some $Q \in \text{Type}_\text{\romup{2}}$ then $R(T) = R(Q).$ Given this, if $R \in \text{Type}_\text{\romup{1}},$
\begin{align}
\sum_{\substack{Q \in \text{Type}_\text{\romup{2}} \\ R(Q) = R}} \ell(Q)^d &= \sum_{Q \in \text{Child}(R) \cap \text{Type}_\text{II}} \sum_{T \in \text{Type}_\text{II}(Q)}\ell(T)^d \\
&\lesssim \sum_{Q \in \text{Child}(R) \cap \text{Type}_\text{II}}\ell(Q)^d \lesssim \ell(R)^d. 
\end{align}
It may be that there is a cube $Q$ in $\text{Type}_\text{\romup{2}}$ which is not contained in any cube from $\text{Type}_\text{\romup{1}}$. If this is the case, then $Q_0 \in \text{Type}_\text{\romup{2}}$ and $Q \in \text{Type}_\text{II}(Q_0).$ In any case, we get
\begin{align}
\sum_{Q \in \text{Type}_\text{II}}\ell(Q)^d &\leq \sum_{Q \in \text{Type}_\text{II}(Q_0)} \ell(Q)^d +  \sum_{R \in \text{Type}_\text{I}} \sum_{\substack{Q \in \text{Type}_\text{II} \\ R(Q) = R }} \ell(Q)^d \\
&\lesssim  \ell(Q_0)^d + \sum_{Q \in\text{Type}_\text{I}}\ell(Q)^d. 
\end{align}
\end{proof}

\begin{proof}[Proof of Lemma \ref{TildeF}]
Let $x \in E^\prime.$ By definition, for each $r >0$ there exists $Q \in \min\mathscr{S}$ such that $\ell(Q) < r$ and $x \in Q.$ By induction, we can construct a sequence of distinct covers $\mathscr{C}_k \subseteq \min\mathscr{S}$ of $E^\prime$ such that $\ell(Q) < \tfrac{1}{k}$ for each $Q \in \mathscr{C}_k$. With the finite assumption \eqref{finite}, it must be that 
\begin{align}\label{e:lim}
\lim_{k\rightarrow\infty} \sum_{Q \in \mathscr{C}_k}\ell(Q)^d = 0.
\end{align}
Assume towards a contradiction that there exists $\delta >0$ and $K \in \N$ such that $\sum_{Q \in \mathscr{C}_k} \ell(Q)^d > \delta$ for all $k \geq K.$ By this and Lemma \ref{MinCubes} it follows that
\[ \ell(Q_0)^d + \sum_{S \in \mathscr{S}} \sum_{Q \in S} \beta_E^{d,1}(MB_Q)^2\ell(Q)^d \gtrsim \sum_{Q \in \min \mathscr{S}} \ell(Q)^d \geq \sum_{k=1}^\infty \sum_{Q \in \mathscr{C}_k} \ell(Q)^d = \infty \]
which contradicts \eqref{finite} and proves \eqref{e:lim}. The fact that $\mathscr{H}^d(E') = 0$ follows from \eqref{e:lim} since
\[ \mathscr{H}^d(E') = \lim_{k \rightarrow \infty} \mathscr{H}^d_{\frac{1}{k}}(E') \leq \lim_{k\rightarrow \infty} \sum_{Q \in \mathscr{C}_k} \ell(Q)^d = 0. \]
Furthermore
\begin{align}\label{blah}
\begin{split}
\mathscr{H}^d(F) &\leq \mathscr{H}^d\left( \bigcup_{S \in \mathscr{S}} \Sigma_S \right) \leq \sum_{S \in \mathscr{S}} \ell(Q(S))^d \\
&\lesssim \ell(Q_0)^d +  \sum_{Q \in \min\mathscr{S}}\sum_{R \in \text{Child}(Q)}\ell(R)^d \\
&\lesssim \ell(Q_0)^d + \sum_{Q \in \min\mathscr{S}}\ell(Q)^d \\
&\lesssim \ell(Q_0)^d + \sum_{Q \in \mathscr{D}^E(Q_0)} \beta^{d,1}_E(MB_Q)^2\ell(Q)^d,
\end{split}
\end{align}
which proves \eqref{FMeasure}. The inequality from the second to the third lines follows because any $Q \in \min\mathscr{S}$ has at most $K = K(M,d)$ children by Lemma \ref{l:children}.  
\end{proof}

\subsection{Proof of (\ref{e:Thm4})}

Recall the definition of $Q_0^F$ from the statement of Theorem \ref{Thm4}. Just like at the beginning of the proof of Theorem \ref{Thm1}, we can reduce proving \eqref{e:Thm4} to proving 

\begin{align}\label{e:var}
&\text{diam}(Q_0^F)^d + \sum_{Q \in \mathscr{D}^F(Q_0^F)}\check\beta_F^{d,p}(C_0B_Q)^2\ell(Q)^d \\
&\quad   \lesssim \text{diam}(Q_0)^d + \sum_{Q \in \mathscr{D}^E(Q_0)}\beta_E^{d,1}(AB_Q)^2\ell(Q)^d,
\end{align}
that is, we can replace the constant $C_0$ by some larger constant $A$ and set $p=1$ on the right-hand side. Let us breifly outline the strategy for proving \eqref{e:var}. In the statement of Theorem \ref{Thm4} we assume $\diam(Q_0) \geq \lambda \ell(Q_0),$ so we have the following bound on the first term: 
\begin{align}
\diam(Q_0^F) \leq \ell(Q_0^F) = \ell(Q_0) \lesssim_\lambda \diam(Q_0).
\end{align}
Thus, in order to prove \eqref{e:var}, it suffices to bound the second term. To do so, we begin by showing 
\begin{align}\label{e:var1}
\sum_{Q \in \mathscr{D}^E(Q_0)}\check\beta_F^{d,p}(C_0B_Q)^2\ell(Q)^d \lesssim \text{diam}(Q_0)^d + \sum_{Q \in \mathscr{D}^E(Q_0)}\beta_E^{d,1}(AB_Q)^2\ell(Q)^d.
\end{align}
Notice that that the sum on the left hand side is taken over $\mathscr{D}^E$, not $\mathscr{D}^F,$ as is required for \eqref{e:var}. From here, we partition the cubes in $F$ into those which lie close to $E$ and those which do not. We control the sum over $F$-cubes which lie close to $E$ by the corresponding sum over $E$-cubes (which we bound using \eqref{e:var1}) and control the sum over the $F$-cubes far away from $E$ by Theorem \ref{AV}.

Let us begin. As promised, we start with \eqref{e:var1}. From here on, we shall assume 
\[  \sum_{Q \in \mathscr{D}^E(Q_0)}\beta_E^{d,1}(AB_Q)^2\ell(Q)^d < \infty \]
since otherwise \eqref{e:var1} is trivial. 

Let $\{S_i\}$ be an enumeration of the stopping time regions in $\mathscr{S}$ which are not singletons. First, we observe that 

\begin{align}
\begin{split}
\sum_{Q \in \mathscr{D}^E(Q_0)} \check\beta_F^{d,p}(C_0B_Q)^2\ell(Q)^d &= \sum_i \sum_{Q \in S_i}  \check\beta_F^{d,p}(C_0B_Q)^2\ell(Q)^d \\
&\hspace{4em}+ \sum_{\substack{S \in \mathscr{S} \\ S = \{Q\}}} \check\beta_F^{d,p}(C_0B_Q)^2\ell(Q)^d.
\end{split}
\end{align}
If $S \in \mathscr{S}$ is such that $S = \{Q\}$ then $Q \in \min \mathscr{S}.$ Then, since $\check\beta_F^{d,p}(\cdot) \lesssim 1,$ the second term on the right hand side of the above equation is at most some constant multiple of
\begin{align}
\sum_{Q \in \min\mathscr{S}} \ell(Q)^d,
\end{align}
which we bound by Lemma \ref{MinCubes}. Thus, to prove \eqref{e:var1} it suffices to bound the first term. Using the $\beta$-error estimate (Lemma \ref{betaest}), we obtain 
\begin{align}
\begin{split}
 \sum_i \sum_{Q \in S_i}  \check\beta_F^{d,p}(C_0B_Q)^2\ell(Q)^d &\overset{\eqref{betaeq}}{\lesssim} \sum_i \sum_{Q \in S_i}  \beta_E^{d,p}(2C_0B_Q)^2\ell(Q)^d \\
&\hspace{-4em} + \sum_i\sum_{Q \in S_i} \left( \frac{1}{\ell(Q)^d} \int_{F \cap 2C_0B_Q} \left( \frac{\text{dist}(x,E)}{\ell(Q)} \right)^p \, d\mathscr{H}_\infty^d(x) \right)^\frac{2}{p}\ell(Q)^d.
\end{split}
\end{align}
We have a trivial bound for the first term by Lemma \ref{l:cont}. Let us focus on bounding the second term, the proof of which is very similar to the proof of Lemma \ref{Similar}. First, let 
\[ \mathscr{D}' = \bigcup_i S_i. \]
We split $\mathscr{D}'$ into two families. Let $\delta >0$ be small and define
\begin{align*}
\mathscr{G} &= \{Q \in \mathscr{D}^\prime : \text{dist}(x,E) \leq \delta \ell(Q) \ \text{for all} \ x \in F \cap \tfrac{M}{4}B_Q \}, \\
\mathscr{B} &= \mathscr{D}^\prime\setminus \mathscr{G}.
\end{align*} 
To each cubes $Q \in \mathscr{B}$ we shall assign a patch $C_Q$ of $F$. By definition, if $Q \in \mathscr{B}$ then there exists a point $y_Q \in F \cap \tfrac{M}{4}B_Q$ such that $\dist(y_Q,E) > \delta \ell(Q).$ We define 
\[C_Q = F \cap B(y_Q,\delta \ell(Q)/2). \] 
We claim that the balls $\{C_Q\}_{Q \in \mathscr{B}}$ have bounded overlap in $F$. This is the content of the following lemma. 
\begin{lem}\label{l:patches}
The collection of balls $\{C_Q\}_{Q \in \mathscr{B}}$ have bounded overlap in $F$.
\end{lem}
\begin{proof}
Let $x \in F$. We will show
\begin{align}\label{e:patches}
\mathscr{B}_x = \# \{ Q \in \mathscr{B} : x \in C_Q \} \lesssim 1.
\end{align}
First, note that
\[
C_Q \subseteq \tfrac{M}{2}B_Q
\] 
for all $Q \in \mathscr{B}.$ Let $Q,Q^\prime \in \mathscr{B}.$ If $\ell(Q^\prime)  \leq \tfrac{\delta}{2M}\ell(Q)$ and $y \in C_Q$ we have
\begin{align}
|y-x_{Q^\prime}| \geq |y_Q-x_{Q^\prime}| - |y - y_Q| \geq \delta \ell(Q) - \frac{\delta \ell(Q)}{2} \geq M\ell(Q^\prime),
\end{align}
in particular, $C_Q \cap \tfrac{M}{2}B_{Q^\prime} = \emptyset.$ Since $C_{Q^\prime} \subseteq \tfrac{M}{2}B_{Q^\prime},$ we must have that $C_Q \cap C_{Q^\prime} = \emptyset.$ Reversing the role of $Q$ and $Q^\prime$ above, we conclude that if $C_Q \cap C_{Q^\prime} \not= \emptyset$ then $\ell(Q) \sim_{M,\delta} \ell(Q^\prime).$ In particular, \eqref{e:patches} follows if we can show that for each $k$, 
\begin{align}\label{e:4:overlap}
\{Q \in \mathscr{B} \cap \mathscr{D}_k : x \in C_Q \} \lesssim 1,
\end{align}
with constant independent of $k$. Fix $k$ and let $Q,Q^\prime \in \mathscr{B} \cap \mathscr{D}_k$ such that $C_Q \cap C_{Q^\prime} \not=\emptyset$. Then $\tfrac{M}{2}B_Q \cap \tfrac{M}{2}B_{Q^\prime} \not=\emptyset$ which implies $|x_Q - x_{Q^\prime}| \leq M\ell(Q).$ Since, $\beta_E^{d,p}(MB_Q) \leq \e,$ we have $\text{dist}(x_{Q^\prime},P_Q) \lesssim \e^\frac{1}{d+1} \ell(Q) = \e^\frac{1}{d+1} \ell(Q^\prime).$ For $\e$ small enough, \eqref{e:4:overlap} follows from the usual argument via Lemma \ref{ENV}.
\end{proof}
\begin{lem}
We have
\begin{align}
\begin{split}
&\sum_i\sum_{Q \in S_i} \left( \frac{1}{\ell(Q)^d} \int_{F \cap 2C_0B_Q} \left( \frac{\emph{dist}(x,E)}{\ell(Q)} \right)^p \, d\mathscr{H}_\infty^d(x) \right)^\frac{2}{p}\ell(Q)^d \\
&\hspace{4em} \lesssim \ell(Q_0)^d +\sum_{Q \in \mathscr{D}^E(Q_0)} \beta^{d,1}_E(MB_Q)^2\ell(Q)^d. 
\end{split}
\end{align}

\end{lem}

\begin{proof}
By Jensen's inequality (Lemma \ref{JensenPhi}), we may assume $p \geq 2$. Let $\mathscr{D}', \mathscr{G}$ and $\mathscr{B}$ be as above Lemma \ref{l:patches}. First, since $F$ is lower regular and the $C_Q$ have bounded overlap, we get 
\begin{equation}
\begin{aligned}
&\sum_{Q \in \mathscr{B}}  \left( \frac{1}{\ell(Q)^d} \int_{F \cap 2C_0B_Q} \left( \frac{\text{dist}(x,E)}{\ell(Q)} \right)^p \, d\mathscr{H}_\infty^d(x) \right)^\frac{2}{p}\ell(Q)^d  \\
&\hspace{4em} \leq \sum_{Q \in \mathscr{B}} \ell(Q)^d \lesssim \sum_{Q \in \mathscr{B}} \mathscr{H}^d(C_Q) \stackrel{\eqref{e:patches}}{\lesssim} \mathscr{H}^d(F) \\
&\hspace{4em} \overset{\eqref{FMeasure}}{\lesssim} \ell(Q_0)^d+ \sum_{Q \in \mathscr{D}^E(Q_0)} \beta_E^{d,1}(MB_Q)^2\ell(Q)^d. 
\end{aligned}
\end{equation}
Consider a single $S = S_i$. Let $S_\mathscr{G}$ denote $S \cap \mathscr{G}.$ Let $\mathscr{D}^F$ denote the Christ-David cubes for $F$ and let $S^{F}_\mathscr{G}$ be the smoothed out cubes in $F$ with respect to $S_\mathscr{G}.$ That is, we let 
\[0 < \tau < \tau_0= \frac{1}{2(1+\rho)}\] 
and define 
\begin{align*}
S^{F}_\mathscr{G} = \{Q \in \mathscr{D}^F : Q  \ \text{is maximal with}\ \ell(Q) < \tau d_{S_\mathscr{G}}(Q) \}.
\end{align*}
Let $R \in S_\mathscr{G}^F.$ 
\bigbreak
\noindent\textbf{Claim 1.} For each $y \in R,  \ \dist(y,E) \lesssim \ell(R).$ 
\bigbreak
By a direct analogue of Lemma \ref{l:Stop} (whose proof is exactly the same), there exists $Q \in  S_\mathscr{G}$ such that $\tau \text{dist}(Q,R) \lesssim \ell(R)$ and $\tau \ell(Q) \sim \ell(R).$ Let $y^Q$ be the point in $Q$ closest to $y.$ Then, 
\begin{align}
\text{dist}(y,E) \leq \dist(y,y^Q) + \dist(y^Q,E) \lesssim \tfrac{\ell(R)}{\tau} + \delta \ell(Q) \lesssim \ell(R).
\end{align}

\bigbreak
\noindent \textbf{Claim 2.} We have $\text{dist}(x_R,\Sigma_{S}) \lesssim \frac{\e^\frac{1}{d+1}}{\tau} \ell(R).$ 
\bigbreak
Let $Q \in S_\mathscr{G}$ be as in the proof of Claim 1. We can chose $M$ large enough (depending on $\tau$) so that $R \subseteq \tfrac{M}{4}B_Q.$ Then, by Lemma \ref{Sigma1} and the fact that $\ell(R) \sim \tau \ell(Q),$
\begin{align}\label{StopSigma}
\text{dist}(x_R,\Sigma_S) \overset{\eqref{Closeness1}}{\lesssim} \e^\frac{1}{d+1}\ell(Q) \lesssim \frac{\e^\frac{1}{d+1}}{\tau} \ell(R).
\end{align}

\noindent \textbf{Claim 3.} For each $k \in \N$, 
\[\#\{Q \in \mathscr{D}^E_k \cap S_\mathscr{G} : R \cap 2C_0B_Q \not=\emptyset\} \lesssim 1.\]
\bigbreak
If $Q \in \mathscr{D}^E_k \cap S_\mathscr{G}$ is such that $R \cap 2C_0B_Q\not=\emptyset,$ then 
\begin{align}
\ell(R) \leq \tau ( \ell(Q) + \text{dist}(R,Q)) \lesssim \ell(Q). 
\end{align}
Let $x_R'$ be the point in $\Sigma_S$ closest to $x_R.$ By the above and \eqref{StopSigma} there exists a constant $A >0$ so that if $Q \in \mathscr{D}_k \cap S$ is such that $R \cap 2C_0B_Q \not=\emptyset$ then 
\[Q \subseteq B(x_R',A\ell(Q)) \coloneqq B.\]
Since $\text{dist}(x_Q,\Sigma_{S}) \lesssim \e ^\frac{1}{d+1}\ell(Q),$ the balls $c_0B_Q$ carve out a large proportion of $\Sigma_S.$ Then, by \eqref{ee:lowerreg1},
\begin{align}
\sum_{\substack{Q \in S_\mathscr{G} \\ R \cap 2C_0B_Q \not=\emptyset}} \ell(Q)^d \lesssim \sum_{\substack{Q \in S_\mathscr{G} \\ R \cap 2C_0B_Q \not=\emptyset}} \mathscr{H}^d(\Sigma_S \cap c_0 B_Q) \leq \mathscr{H}^d(\Sigma_S \cap B) \lesssim \ell(Q)^d,
\end{align} 
which proves the claim. 
\bigbreak
Since we have assumed $p>2$, we can apply Lemma \ref{l:subsum} and Claim 1 to get
\begin{align}
&I \coloneqq \sum_{Q \in S_\mathscr{G}} \left( \frac{1}{\ell(Q)^d} \int_{F \cap 2C_0B_Q} \left( \frac{\text{dist}(x,E)}{\ell(Q)} \right)^p \, d\mathscr{H}_\infty^d(x) \right)^\frac{2}{p}\ell(Q)^d \\
&\hspace{4em} \lesssim \sum_{Q \in S_\mathscr{G}} \left(\sum_{\substack{R \in S^{F}_\mathscr{G} \\ R \cap 2C_0B_Q \not= \emptyset}}\frac{\ell(R)^{d+p}}{\ell(Q)^{d+p}} \right)^\frac{2}{p} \ell(Q)^d \lesssim \sum_{Q \in S_\mathscr{G}}\sum_{\substack{R \in S^{F}_\mathscr{G} \\ R \cap 2C_0B_Q \not= \emptyset}} \frac{\ell(R)^{d\frac{2}{p}+2}}{\ell(Q)^{d(\frac{2}{p}-1) +2}}
\end{align}
By Claim 3, we swap the order of integration and sum over a geometric series to get,
\begin{align*}
I \lesssim \sum_{\substack{R \in S^{F}_\mathscr{G}  \\ R \cap 2C_0B_{Q(S)} \not=\emptyset}} \sum_{\substack{Q \in S_\mathscr{G} \\ R \cap 2C_0B_Q \not= \emptyset}}\frac{\ell(R)^{d\frac{2}{p} + 2 }}{\ell(Q)^{d(\frac{2}{p} - 1) + 2}} \lesssim  \sum_{\substack{R \in S^{F}_\mathscr{G} \\ R \cap 2C_0B_{Q(S)} \not=\emptyset}} \ell(R)^d. 
\end{align*}
By \eqref{StopSigma}, for $\e$ small enough, the ball $c_0B_R$ carves out a large proportion of $\Sigma_{S}$ for each $R \in S^{F}_\mathscr{G}$ , i.e. $\mathscr{H}^d(c_0B_R \cap \Sigma_{S}) \gtrsim \ell(R)^d.$ By \eqref{ee:lowerreg1}, using the fact the $c_0B_R$ are disjoint, we have
\begin{align*}
I \lesssim \sum_{\substack{R \in S^{F}_\mathscr{G} \\ R \cap 2C_0B_{Q(S)} \not=\emptyset}} \mathscr{H}^d(c_0B_R \cap \Sigma_{S}) \leq \mathscr{H}^d(B_{Q(S)} \cap \Sigma_{S}) \lesssim \ell(Q(S))^d.
\end{align*}
Hence,
\begin{align}
&\sum_i\sum_{Q \in S_i} \left( \frac{1}{\ell(Q)^d} \int_{F \cap 2C_0B_Q} \left( \frac{\text{dist}(x,E)}{\ell(Q)} \right)^p \, d\mathscr{H}_\infty^d(x) \right)^\frac{2}{p}\ell(Q)^d \\
&\hspace{4em} \lesssim \sum_{S \in \mathscr{S}} \ell(Q(S))^d \overset{\eqref{blah}}{\lesssim} \ell(Q_0)^d +\sum_{Q \in \mathscr{D}^E(Q_0)} \beta^{d,1}_E(MB_Q)^2\ell(Q)^d
\end{align}

\end{proof}

This takes care of \eqref{e:var1}. Now, we will now control the sum over cubes in $\mathscr{D}^F.$ For $k \geq 0,$ let 
$$\mathscr{B}_k^E = \bigcup_{Q \in \mathscr{D}_k^E(Q_0)}MB_Q.$$

\begin{lem}
Define
\begin{align}
\emph{Top}_F = \{Q \in \mathscr{D}^F(Q_0^F) : Q \ \emph{is max such that} \ (C_0+M)\ell(Q) < \emph{dist}(x_Q,E)\}.
\end{align}
Let $k \in \N$ and $Q \in \emph{Top}_F \cap \mathscr{D}_k^F.$ Then 
\begin{align}\label{e:top1}
C_0B_Q \cap \mathscr{B}_l^E = \emptyset \quad \text{for all} \ l \geq k
\end{align}
and
\begin{align}\label{e:top2}
C_0B_Q \subseteq 3\mathscr{B}^E_l \quad \text{for all} \ 0 \leq l \leq k-1
\end{align}
\end{lem}
\begin{proof}
Let $Q \in \text{Top}_F \cap \mathscr{D}_k^F$. By definition, \eqref{e:top1} is immediate so let us prove \eqref{e:top2}.  By maximality, we have
\begin{align}\label{e:above}
\text{dist}(x_Q,E) \leq |x_Q - x_{Q^{(1)}}| + \text{dist}(x_{Q^{(1)}},E) \leq (1 + C_0+M)\rho^{-1}\ell(Q).
\end{align}
Let $z_Q$ be the point in $E$ closest to $x_Q$ and let $Q'$ be the cube in $\mathscr{B}^E_{k-1}$ such that $z_Q \in Q'.$ Then, for $y \in C_0B_Q,$ we have 
\begin{align}
|y-x_{Q'}| &\leq |y - x_Q| +  |x_Q - z_Q| + |z_Q - x_{Q'}| \\
&\stackrel{\eqref{e:above}}{\leq}  C_0\ell(Q) + (1+C_0 +M)\rho^{-1}\ell(Q) + \ell(Q') \\
&\leq 3M\rho^{-1}\ell(Q) = 3M\ell(Q')
\end{align}
for $M$ large enough. This implies $y \in 3MB_{Q'}.$ Since $y$ is arbitrary point in $C_0B_Q,$ we have
\[C_0B_Q \subseteq 3MB_{Q'} \subseteq 3\mathscr{B}^E_{k-1}.\]
Clearly, this also implies that $C_0B_Q \subseteq 3\mathscr{B}^E_{l}$ for all $0 \leq l \leq k-1.$
\end{proof}

\begin{lem}\label{l:topoverlap}
The collection of balls $\{B_Q\}_{Q \in \emph{Top}_F}$ have bounded overlap with constant dependent on $n$. 
\end{lem}

\begin{proof}
Let $x \in \R^n$ and let
\[ \mathscr{Q}_x = \{Q \in \text{Top}_F : x \in B_Q \}. \] 
We first show that each cube in $\mathscr{Q}_x$ has comparable size. Let $Q,Q' \in \mathscr{Q}_x$ and assume without loss of generality that $\ell(Q) \leq \ell(Q').$ Since $Q' \in \text{Top}_F$ and $B_Q \cap B_{Q'} \not=\emptyset$ we have
\begin{align}
\ell(Q') &< (C_0 + M)^{-1} \dist(x_{Q'},E) \leq   (C_0 + M)^{-1} (|x_Q - x_{Q'}| + \dist(x_Q,E)) \\
&\leq (C_0 + M)^{-1}( 2\ell(Q') + \dist(x_Q,E) ). 
\end{align}
Taking $M$ large enough so that $2(C_0 +M)^{-1} \leq \tfrac{1}{2},$ we can rearrange the above equation to give
\begin{align}
\ell(Q') \leq \frac{2}{C_0 +M} \dist(x_Q,E) \stackrel{\eqref{e:above}}{\lesssim} \ell(Q)
\end{align}
which proves that
\begin{align}\label{e:samesize}
\ell(Q) \sim \ell(Q').
\end{align} 
By a standard volume argument, for each $k \in \N$ we have 
\begin{align}
\# ( \mathscr{Q}_x \cap \mathscr{D}_k ) \lesssim_{n} 1,
\end{align}
which when combined with \eqref{e:samesize} finishes the proof of the lemma. 
\end{proof}

\begin{lem}\label{Top}
\begin{align}
\sum_{Q \in \emph{Top}_F} \sum_{R \subseteq Q}\check\beta_F^{d,p}(C_0B_R)^2\ell(R)^d \lesssim \ell(Q_0^E)^d+ \sum_{Q \in \mathscr{D}^E(Q_0)} \beta_E^{d,p}(MB_Q)^2\ell(Q)^d. 
\end{align}
\end{lem}
\begin{proof}

Let $Q \in \text{Top}_F \cap \mathscr{D}^F_k$ and let $\mathscr{S}_Q \subseteq \mathscr{S}$ be the collection of stopping time regions such that $C_0B_Q \cap \Sigma_{Q(S)} \not= \emptyset.$ Since $C_0B_Q \cap \mathscr{B}^E_l = \emptyset$ for all $l \geq k,$ by \eqref{e:top1}, it must be that $Q(S) \in \mathscr{D}^E_l$ for some $0 \leq l \leq k-1.$ By \eqref{e:top2} it follows that 
\[C_0B_Q \subseteq 3MB_{Q(S)}\]
for each $S \in \mathscr{S}_Q.$ We wish to use that fact that in each of these balls, $F$ is well approximated by a union of planes. At the minute this is not quite the case. It could be that $C_0B_Q$ intersects $MB_{Q(S)}$ at the boundary for some $S \in \mathscr{S}_Q$ (see for example Figure \ref{BAUP1}). As such we must extend each of the surfaces $\Sigma_S$.  

\begin{figure}[h]
  \centering
  \includegraphics[scale = 0.42]{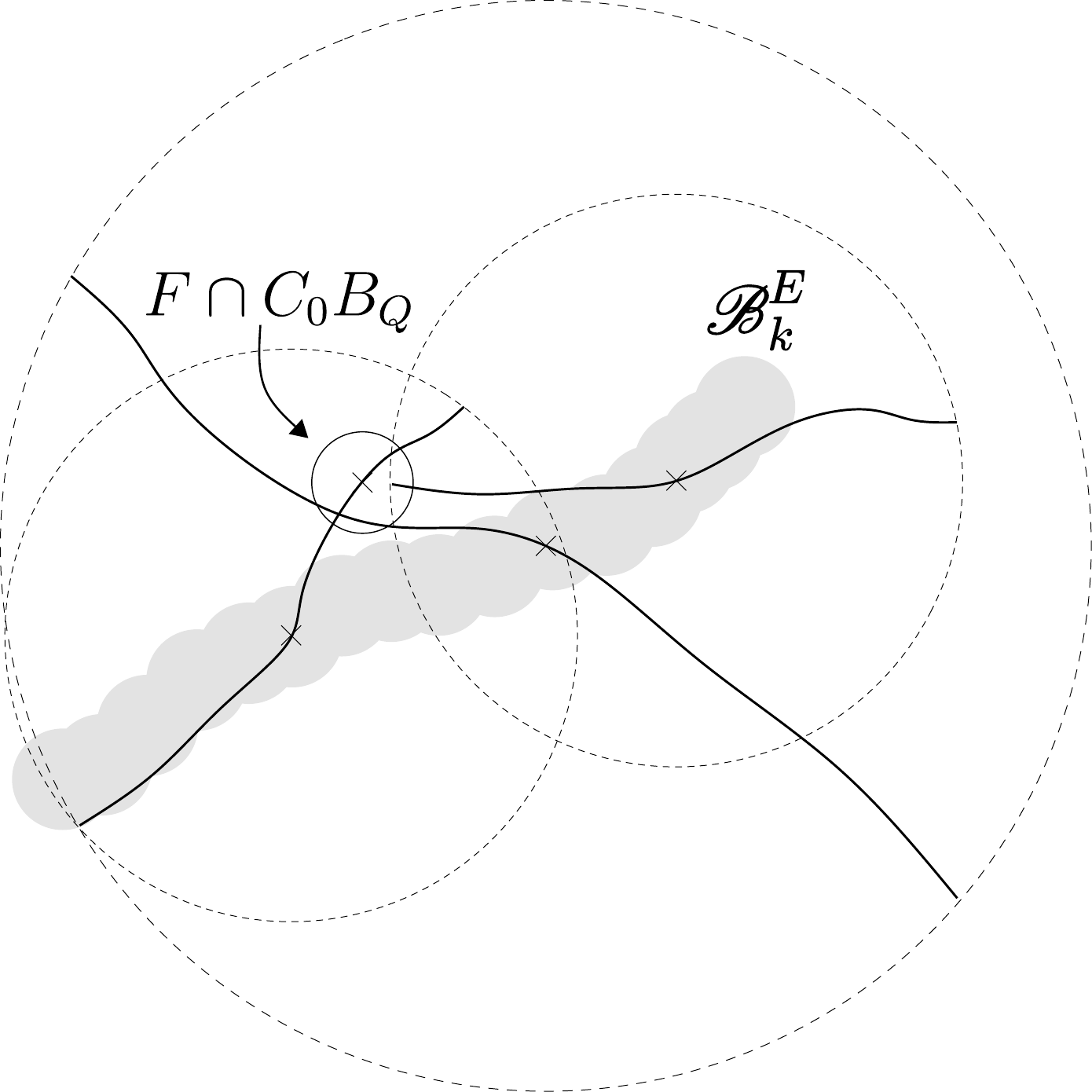}
\caption{$F$ is not necessarily well approximated by a union of planes.}\label{BAUP1}
\end{figure}

Recall from Lemma \ref{Sigma1} that $\Sigma_S'$ is the unbounded bi-Lipschitz surface from Theorem \ref{LipDT}. For each $S \in \mathscr{S},$ let $\tilde{\Sigma}_S$ be the surface obtained by restricting $\Sigma_S'$ to $6MB_{Q(S)},$ i.e. 
\begin{align}
\tilde{\Sigma}_S = \Sigma_S^\prime \cap 6MB_{Q(S)}.
\end{align}
Compare this to how we define $\Sigma_S$ in \eqref{e:sigmaM}. This ensures that each $\tilde{\Sigma}_S$ is $(C\e,d)$-Reifenberg flat in $3MB_{Q(S)}.$ Clearly
\begin{align}
F \cap C_0B_Q \subseteq E' \cup \left( \bigcup_{S \in \mathscr{S}_Q} \tilde{\Sigma}_{S} \right) \cap C_0B_Q.
\end{align}
Define 
\[ F_Q = E' \cup \left( \bigcup_{S \in \mathscr{S}_Q} \tilde{\Sigma}_S \right) \cup \left( \bigcup_{S \in \mathscr{S}\setminus\mathscr{S}_Q} \Sigma_S \right). \] 
We can show that $F_Q$ is lower regular in exactly the same way as we did for $F$ (we shall omit the details). 

It is easy to see that $F \subseteq F_Q.$ For each $k \geq 0$ let $X_k^{F_Q}$ be the completion of $X_k^F$ to a maximal $\rho^k$-separated net for $F_Q.$ Let $\mathscr{D}^{F_Q}$ denote the cubes for $F_Q$ from Theorem \ref{cubes} with respect to the sequence $X_k^{F_Q}.$ In this way, for each $R \in \mathscr{D}^F$ there exists a corresponding cube $\tilde{R} \in \mathscr{D}^{F_Q}$ such that $x_R = x_{\tilde{R}}$ and $\ell(R) = \ell(\tilde{R}).$ It is clear then, that 
\[C_0B_R = C_0B_{\tilde{R}}. \] 
Now, let $\tilde{Q}$ be the cube described above for $Q$ and let $\tilde{R} \in \mathscr{D}^{F_Q}$ be such that $\tilde{R} \subseteq \tilde{Q}$. We aim to show that $\tilde{R}$ is well approximated by a union of planes. If $S \in \mathscr{S}_Q$ and $\tilde{\Sigma}_{S} \cap C_0B_{\tilde{R}} \not= \emptyset,$ let $x_S \in \tilde{\Sigma}_S \cap C_0B_{\tilde{R}}$ be a point of intersection. Since 
\[C_0B_{\tilde{R}} \subseteq B(x_S,3C_0\ell(\tilde{R})) \subseteq 3MB_{Q(S)}\]
and $\tilde{\Sigma}_S$ is $(C\e,d)$-Reifenberg flat in $3MB_{Q(S)}$, we can find a plane $L_S$ through $x_S$ such that 
\begin{align}
d_{C_0B_{\tilde{R}}}(\tilde{\Sigma}_S,L_S) \leq C\e. 
\end{align}
Let 
\begin{align}
 U_{\tilde{R}}  \coloneqq \bigcup_{\substack{S \in \mathscr{S}_Q \\ C_0B_{\tilde{R}} \cap \tilde{\Sigma}_{S} \not= \emptyset}} L_S.
\end{align} 
By construction we have
\begin{align}
F_Q \cap C_0B_{\tilde{R}}= \left( E' \cap C_0B_{\tilde{R}}  \right) \cup \bigcup_{\substack{S \in \mathscr{S}_Q \\ C_0B_{\tilde{R}} \cap \tilde{\Sigma}_{S} \not= \emptyset}} \tilde{\Sigma}_S.
\end{align}
Since each point in $E' \cap C_0B_{\tilde{R}}$ is arbitrarily close to $\tilde{\Sigma}_S$ for some $S \in \mathscr{S}_Q$, it follows that 
\begin{align}
d_{C_0B_{\tilde{R}}}(F_Q,U_{\tilde{R}}) \leq C\e,
\end{align}
i.e. $\tilde{R} \not\in \text{BAUP}(C_0,C\e)$.
See the below Figure \ref{BAUP2}. 
\begin{figure}[h]
  \centering
  \includegraphics[scale = 0.66]{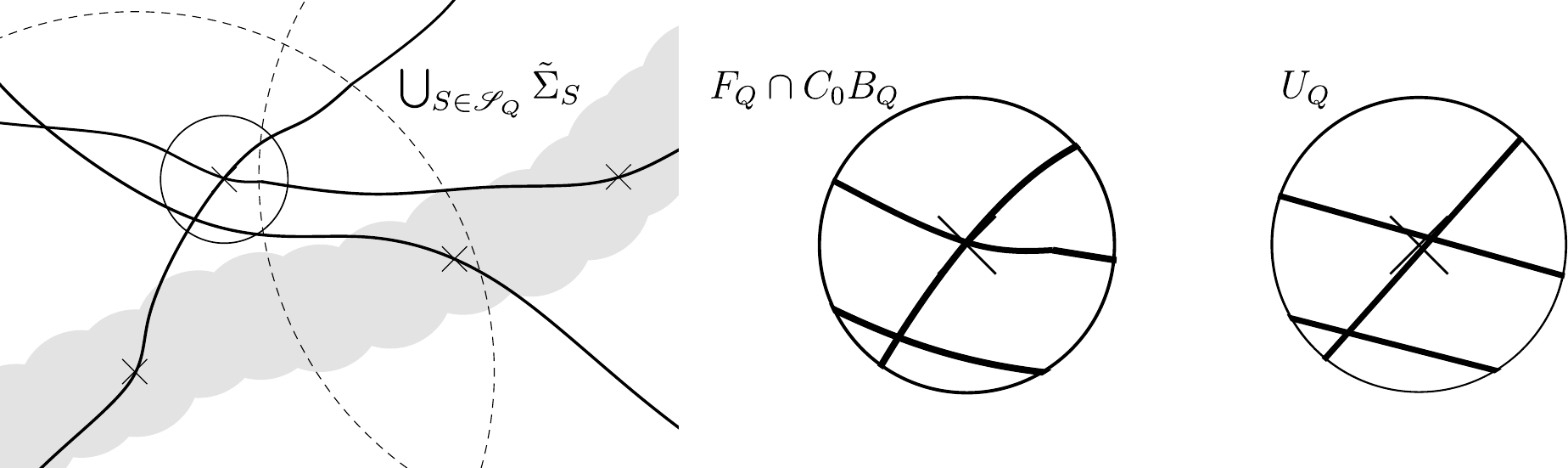}
\caption{An illustration of the above argument for $R =Q$. Compare the extended surface shown above to the original surface shown in Figure \ref{BAUP1}.}\label{BAUP2}
\end{figure}
Since $\tilde{R}$ and $Q \in \text{Top}_F$ we arbitrary, we have $\text{BAUP}(\tilde{Q},C_0,C\e) = 0$ for all $Q \in \text{Top}_F.$ Using that fact that $F \subseteq F_Q$, the correspondence between cubes in $\mathscr{D}^F$ and $\mathscr{D}^{F_Q}$ and Theorem \ref{AV}, we get
\begin{align}
\sum_{R \subseteq Q}\check\beta_F^{d,p}(C_0B_R)^2\ell(R)^d \leq \sum_{\tilde{R} \subseteq \tilde{Q}}\check\beta_{F_Q}^{d,p}(C_0B_{\tilde{R}})^2\ell(\tilde{R})^d \stackrel{\eqref{e:AV}}{\lesssim} \mathscr{H}^d(\tilde{Q}).
\end{align}
Since by Lemma \ref{l:topoverlap} the collection of balls $\{B_Q\}_{Q \in \text{Top}_F}$ have bounded overlap, the same is true for the cubes $\{\tilde{Q}\}_{Q \in \text{Top}_F}.$ If we define
\begin{align*}
\tilde{F} = E^\prime \cup \bigcup_{S \in \mathscr{S}} \tilde\Sigma_S,
\end{align*}
then $\tilde{Q} \subseteq \tilde{F}$ for all $Q \in \text{Top}_F$ which gives
\begin{align*}
\sum_{Q \in \text{Top}_F}\sum_{R \subseteq Q}\check\beta_F^{d,p}(C_0B_Q)^2\ell(Q)^d &\lesssim \sum_{Q \in \text{Top}_F} \mathscr{H}^d(\tilde{Q}) \lesssim \mathscr{H}^d(\tilde{F})  \\
&\lesssim \sum_{S \in \mathscr{S}} \ell(Q(S))^d \\
&\lesssim \ell(Q_0)^d + \sum_{Q \in \mathscr{D}^E(Q_0)} \beta^{d,1}_E(MB_Q)^2\ell(Q)^d,
\end{align*}
where the last inequality follows from \eqref{blah}.

\end{proof}

\begin{lem}\label{l:Up} Let $\emph{Up}_F$ be the collection of cubes in $\mathscr{D}^F$ which are not properly contained in any cube from $\emph{Top}_F$. Then 
\[ \sum_{Q \in \emph{Up}_F} \check\beta_F^{d,p}(C_0B_Q)^2\ell(Q)^d \lesssim \ell(Q_0)^d + \sum_{Q \in \mathscr{D}^E(Q_0)} \beta^{d,1}_E(2MB_Q)^2\ell(Q)^d.\]
\end{lem}
\begin{proof} 
Let $Q \in \text{Up}_F\cap \mathscr{D}^F_k.$ By construction, we have $\text{dist}(x_Q,E) < (C_0 +M)\ell(Q)$ so there exists a cube $Q^\prime \in \mathscr{D}^E_{k}$ such that $C_0B_Q \cap MB_{Q^\prime} \not=\emptyset.$ In particular $C_0B_Q \subseteq 2MB_{Q'}$, which by Lemma \ref{l:cont} implies 
\begin{align}\label{e:end1}
\check\beta^{d,p}_F(C_0B_Q) \lesssim \check\beta^{d,p}_F(2MB_{Q'}).
\end{align}
For $Q \in \mathscr{D}^E,$ a standard volume argument gives
\begin{align}\label{e:end2}
\{R \in \text{Up}_F : R^\prime = Q \} \lesssim_n 1.
\end{align}
Then,
\begin{align}\label{Up}
\begin{split}
\sum_{Q \in \text{Up}_F} \check\beta_F^{d,p}(C_0B_Q)^2\ell(Q)^d &\leq \sum_{Q \in \mathscr{D}^E(Q_0)} \sum_{\substack{R \in \text{Up}_F \\ R^\prime =Q}} \check\beta_F^{d,p}(C_0B_R)^2\ell(R)^d \\
&\stackrel{\substack{\eqref{e:end1} \\ \eqref{e:end2}}}{\lesssim_n} \sum_{Q \in \mathscr{D}^E(Q_0)}  \check\beta_F^{d,p}(2MB_Q)^2\ell(Q)^d \\
&\lesssim \ell(Q_0)^d + \sum_{Q \in \mathscr{D}^E(Q_0)}  \check\beta_F^{d,p}(C_0B_Q)^2\ell(Q)^d \\
&\stackrel{\eqref{e:var1}}{\lesssim} \ell(Q_0)^d + \sum_{Q \in \mathscr{D}^E(Q_0)} \beta^{d,1}_E(2MB_Q)^2\ell(Q)^d. 
\end{split}
\end{align}
The third inequality follows for the following reason: for $Q$ small enough (depending on $C_0$ and $M$) we can find a larger cube $Q'$ such that $MB_Q \subseteq C_0B_{Q'}$. We use this along with the fact that any cube has a bounded number of descendants up to the $K^{th}$ generation, say, with constant dependent on $n$ and $K$. The sum of the larger cubes is absorbed into the first term since again we can control the number of these cubes. This is what we did in the proof of Lemma \ref{l:Suffices1}, see there for more details.
\end{proof}
The proof of Theorem \ref{Thm4} is finished by noting that 
\begin{align}
 \sum_{Q \in \mathscr{D}^F(Q_0^F)}\check\beta_F^{d,p}(C_0B_Q)^2\ell(Q)^d &= \sum_{Q \in \text{Up}_F} \check\beta_F^{d,p}(C_0B_Q)^2\ell(Q)^d \\
&\hspace{4em}+ \sum_{Q \in \text{Top}_F}\sum_{R \subseteq Q}  \check\beta_F^{d,p}(C_0B_R)^2\ell(R)^d
\end{align}
and applying Lemma \ref{Top} and Lemma \ref{Up}. 

\bibliography{Ref}
\bibliographystyle{alpha}

\end{document}